\documentclass{article}

\usepackage[T1]{fontenc}
\usepackage{lmodern}
\usepackage[utf8x]{inputenx}
\usepackage{framed}
\usepackage{listings}
\usepackage{vmargin}
\usepackage{mathrsfs, mathenv}
\usepackage{amsmath, amsthm, amssymb, amsfonts, amscd}
\usepackage{graphicx}
\usepackage{epstopdf}
\usepackage[svgnames]{xcolor}
\usepackage{hyperref}
\usepackage[capitalise]{cleveref}
\hypersetup{citecolor=blue, colorlinks=true, linkcolor=black}
\usepackage{cite}
\usepackage{pgfplots}
\usepackage{tikz}
\usepackage{color}
\usepackage{lipsum}
\usepackage[ruled, vlined]{algorithm2e}
\crefname{algocf}{Algorithm}{Algorithms}
\usepackage{autonum}
\usepackage[usestackEOL]{stackengine}
\usepackage{stmaryrd} 
\usepackage[title]{appendix}

\theoremstyle{plain}
\newtheorem{theorem}{Theorem}[section]

\newtheorem{lemma}[theorem]{Lemma}

\numberwithin{equation}{section}

\theoremstyle{definition}
\newtheorem{definition}[theorem]{Definition}

\theoremstyle{remark}
\newtheorem{remark}[theorem]{Remark}
\newtheorem{assumption}[theorem]{Assumption}
\newtheorem{example}[theorem]{Example}

\ifpdf
  \DeclareGraphicsExtensions{.eps,.pdf,.png,.jpg}
\else
  \DeclareGraphicsExtensions{.eps}
\fi

\usepackage{mathtools}

\usepackage{booktabs}

\usepackage{pgfplots}
\usepackage{tikz}
\usetikzlibrary{patterns,arrows,decorations.pathmorphing,backgrounds,positioning,fit,matrix}
\usepackage[labelfont=bf]{caption}
\usepackage{here}
\usepackage[font=normal]{subcaption}

\usepackage{enumitem}
\setlist[itemize]{leftmargin=.5in}
\setlist[enumerate]{leftmargin=.5in,topsep=3pt,itemsep=3pt,label=(\roman*)}

\newcommand{\email}[1]{\href{#1}{#1}}
\newcommand{\TheTitle}{A probabilistic finite element method based on random meshes: Error estimators and Bayesian inverse problems} 
\newcommand{\TheAuthors}{A. Abdulle, G. Garegnani}
\title{{\TheTitle}}
\newcommand*\samethanks[1][\value{footnote}]{\footnotemark[#1]}
\author{Assyr Abdulle\thanks{Institute of Mathematics, \'Ecole Polytechnique F\'ed\'erale de Lausanne (\email{assyr.abdulle@epfl.ch}, \email{giacomo.garegnani@epfl.ch})}
	\and
	Giacomo Garegnani\samethanks}
\date{}

\usepackage{amsopn}

\newcommand{\abs}[1]{\left\lvert #1 \right\rvert}
\newcommand{\norm}[1]{\left\lVert #1 \right\rVert}
\newcommand{\dbrack}[1]{\left\llbracket#1\right\rrbracket}
\renewcommand{\phi}{\varphi}
\renewcommand{\theta}{\vartheta}

\newcommand{\iid}{\ensuremath{\stackrel{\text{i.i.d.}}{\sim}}}

\newcommand{\R}{\mathbb{R}}

\newcommand{\defeq}{\coloneqq}
\newcommand{\eqdef}{\eqqcolon}

\newcommand{\E}{\operatorname{\mathbb{E}}}
\newcommand{\Eb}[1]{\E\left[#1\right]}

\newcommand{\Hell}{d_{\mathrm{H}}}
\newcommand{\dd}{\, \mathrm{d}}
\renewcommand{\d}{\mathrm{d}}

\newcommand{\eval}[1]{\bigr\rvert_{#1}}

\ifpdf
\hypersetup{
	pdftitle={\TheTitle},
	pdfauthor={\TheAuthors}
}
\fi

\definecolor{leg1}{RGB}{0,114,189}
\definecolor{leg2}{RGB}{217,83,25}
\definecolor{leg3}{RGB}{237,177,32}
\definecolor{leg4}{RGB}{126,47,142}
\definecolor{leg5}{RGB}{119,172,48}

\definecolor{leg21}{RGB}{62,38,169}
\definecolor{leg22}{RGB}{46,135,247}
\definecolor{leg23}{RGB}{55,200,151}
\definecolor{leg24}{RGB}{254,195,56}

\begin{document}
\maketitle 

\begin{abstract} We present a novel probabilistic finite element method (FEM) for the solution and uncertainty quantification of elliptic partial differential equations based on random meshes, which we call random mesh FEM (RM-FEM). Our methodology allows to introduce a probability measure on standard piecewise linear FEM. We present a posteriori error estimators based uniquely on probabilistic information. A series of numerical experiments illustrates the potential of the RM-FEM for error estimation and validates our analysis. We furthermore demonstrate how employing the RM-FEM enhances the quality of the solution of Bayesian inverse problems, thus allowing a better quantification of numerical errors in pipelines of computations. 
\end{abstract}



\section{Introduction}

In recent years, there has been a growing interest in developing and analyzing probabilistic counterparts of traditional numerical methods spanning most areas of computational mathematics. This gave rise to the field of Probabilistic Numerics (PN), whose founding principles and aims are summarized in the review papers \cite{OaS19, HOG15, COS19}. All methods belonging to the field of PN share the idea of introducing a probability measure on the solution of traditional numerical methods. The underlying rationale is to quantify the uncertainty due to numerical errors in a probabilistic manner, rather than with standard error estimates. Indeed, a probability measure over approximate solutions can be readily pushed through a pipeline of computations, thus justifying the need of probabilistic methods especially when the solution of the problem at hand is employed as the input of a subsequent analysis. A typical example of computational pipelines for which probabilistic methods are successfully employed is given by Bayesian inverse problems, where introducing a probability measure on the forward model allows for a better quantification of the uncertainty in the inversion procedure.

\subsection{Literature Review}

Several contributions to the field of PN concern differential equations. For ordinary differential equations (ODEs), the methodologies can be roughly split in two different areas. In \cite{Ski92, KeH16, KSH20, TKS19, SSH19, CCC16, SDH14} the authors present a series of schemes which rely in different measure on Bayesian filtering techniques. These methodologies proceed by updating Gaussian measures over the numerical solution with filtering formulae and evaluations of the right-hand side of the ODE, which are interpreted as observations. While being not involved computationally, analyzing the convergence properties of this class of methods is not always possible, and one can only marginally rely on standard techniques for this purpose. A valuable effort in this sense can be found in \cite{KSH20}, where the authors show rates of convergence of the mean of the Gaussian measure towards the exact solution. A different approach is presented in the series of works \cite{CGS17, LSS19b, AbG20, TLS18, TZC16}, where the authors propose probabilistic schemes which are based on perturbing randomly the approximate solution and on letting evolve these perturbations through the dynamics of the ODE. In this manner, it is possible to obtain empirical probability measures over the otherwise deterministic numerical solution. A random perturbation can be applied directly to the state, as it was presented and analysed for one-step methods in \cite{CGS17, LSS19b}, with a particular focus on implicit schemes in \cite{TLS18} and for multistep methods in \cite{TZC16}. Another approach, which was presented in \cite{AbG20}, consists in perturbing the scheme itself by randomizing the time steps of a Runge--Kutta method. This allows to maintain certain geometric properties of the deterministic scheme in its probabilistic counterpart, such as the conservation of invariants or the symplecticity. 

There has been a keen interest from the PN community on developing probabilistic numerical solvers for partial differential equations (PDEs), too \cite{COS17, COS17b, OCA19, CCC16, CGS17, Owh15, Owh17, OwZ17, RPK17, RPK17b,GFY21}. In \cite{COS17b}, the authors present a meshless Bayesian method for PDEs, which they then apply to inverse problems in \cite{COS17}, and in particular to a challenging time-dependent instance drawn from an engineering application in \cite{OCA19}. Their methodology consists of placing a Gaussian prior on the space of solutions, thus updating it with evaluations of the right-hand side, which are interpreted as noisy observations. A similar idea has been presented in \cite{CCC16}, where the main focus are time-dependent problems, and in \cite{RPK17, RPK17b}, where the method is recast in the framework of machine learning algorithms. In \cite{Owh17, OwZ17}, a probabilistic approach involving gamblets is applied to the solution of PDEs with rough coefficients and by multigrid schemes, with a particular interest to reducing the complexity of implicit algorithms for time-dependent problems \cite{OwZ17}. Moreover, in \cite{Owh15} the author presents a Bayesian reinterpretation of the theory of homogenization for PDEs, which can be seen as a contribution to the field of PN. To our knowledge, the only perturbation-based finite element (FE) probabilistic scheme for PDEs is presented in \cite{CGS17}, where the authors randomize FE bases by adding random fields endowed with appropriate boundary conditions, thus obtaining an empirical measure over the space of solutions. By tuning the covariance of these random fields, they obtain a consistent characterization of the numerical error, which can then be employed to solve Bayesian inverse problems and to quantify the uncertainty over their numerical solution.

\subsection{Our Contributions}

In this work we present a probabilistic finite element method (FEM) which is based on a randomization of the mesh, and which we call RM-FEM. The idea underlying our method stems from both \cite{CGS17}, where the authors propose a probabilistic FEM based on random perturbations, and from \cite{AbG20}, in which the first instance of randomizing the discretization instead of the solution itself is presented. In the context of ODEs, a careful randomization of the time step in Runge--Kutta methods allows to maintain certain convergence and geometric properties, either path-wise or in the mean-square sense. In a similar fashion, creating a probability measure on the space of solutions by randomizing the mesh has the advantage that each sample is a FEM solution itself, and therefore a projection of the exact solution on some random finite-dimensional space. 

Keeping in mind the fundamental goal of PN, we consider the problem of employing probabilistic methods to quantify numerical errors in the context of PDEs. Indeed, in \cite{CGS17, AbG20, KSH20} and other works concerning ODEs, the authors show that the probabilistic solution converges to the true solution with the same rate as the deterministic method, which represents a consistency result. No work so far shows that PN methods can be readily employed for an a posteriori estimation of the error. Some forms of adaptivity for nonlinear ODEs based on probabilistic information can be found in \cite{ChC19, SSH19, BHT20}, where the arguments are based on heuristics but are not rigorously analyzed. In this work, we construct and present a posteriori error estimators which can be readily employed for mesh adaptation in elliptic PDEs. Our estimators are entirely based on probabilistic information, are simple to compute and do not entail considerable computational cost. We present an analysis in the one-dimensional case that shows that our error estimators based on the RM-FEM are equivalent to a classical estimator by Babu\v{s}ka and Rheinboldt \cite{BaR81}, which employs the jumps of the derivative of the solution at the nodes to quantify the numerical errors. Our one-dimensional theoretical analysis is complemented by a series of numerical experiments confirming the validity of our theory in higher dimensions.

As stated above, probabilistic numerical methods are especially appealing when employed in pipeline of computations such as Bayesian inverse problems. In particular, employing deterministic methods for approximating forward maps leads to overly confident posterior measures, which can be corrected by appropriate probabilistic approximations. Similarly to \cite{CGS17, AbG20}, we show in this paper how the RM-FEM can be employed to construct empirical distributions over the forward problem and compute a random posterior measure, solution to the inverse problem in the Bayesian sense. The solution is consistent asymptotically with respect to the mesh spacing, but its quality is enhanced if the latter is relatively large, i.e., if the forward model is approximated cheaply. 

\subsection{Outline}

The outline of the paper is as follows. In \cref{sec:RMFEM} we state the problem of interest, introduce the RM-FEM and the main assumptions and notation required by our analysis. We then present the two main applications of the RM-FEM, i.e., a posteriori error estimators and Bayesian inverse problems, in \cref{sec:APosteriori,sec:BIP}, respectively. For both applications, a series of numerical experiments in the one and two-dimensional cases illustrate the usefulness and efficiency of the RM-FEM. In \cref{sec:ProbErrEst} we present a rigorous a priori and a posteriori error analysis. Finally, in \cref{sec:Conclusion} we draw our conclusions.

\section{Random Mesh Finite Element Method}\label{sec:RMFEM}

\subsection{Notation}

Let $d = 1, 2, 3$ and $D \subset \R^d$ be an open bounded domain with sufficiently smooth boundary $\partial D$. For $v \in \R^d$, we denote by $\norm{v}_2$ the Euclidean norm on $\R^d$. We denote by $L^2(D)$ the space of square integrable functions, by $(\cdot, \cdot)$ the natural $L^2(D)$ inner product, and by $H^p(D)$ the Sobolev space of functions with $p$ weak derivatives in $L^2(D)$. Moreover, we denote by $H^1_0(D)$ the space of functions in $H^1(D)$ vanishing on $\partial D$ in the sense of traces, by $H^{-1}(D)$ the dual of $H_0^1(D)$ and by $\langle \cdot, \cdot \rangle$ the natural pairing between $H^{-1}(D)$ and $H_0^1(D)$. We equip the space $H^1_0(D)$ with the norm $\norm{v}_{H_0^1(D)} = \norm{\nabla v}_{L^2(D)}$, i.e. the $H^1(D)$ seminorm. 

For an event space $\Omega$, with a $\sigma$-algebra $\mathcal A$ and a probability measure $P$, we let the triple $(\Omega, \mathcal A, P)$ denote a probability space. For an event $A \in \mathcal A$, we say that $A$ occurs almost surely (a.s.) if $P(A) = 1$. For $n \in \mathbb N$ we call random variables the measurable functions $X \colon \Omega \to \R^n$, and denote by $L^2(\Omega)$ the space of square integrable random variables, with associated inner product. Denoting by $\mathcal B(\R^n)$ the Borel $\sigma$-algebra on $\R^n$, we say that a probability measure $\mu_X$ on the measurable space $(\R^n, \mathcal B(\R^n))$ satisfying $\mu_X(B) = P(X^{-1}(B))$ for all $B \in \mathcal B(\R^n)$ is the measure induced by $X$, or equivalently the distribution of $X$. For a set of random variables $\{X_i\}_{i=1}^n$ which are independent and identically distributed, we say they are i.i.d., and denoting by $\mu$ their common induced measure on $(\R^n, \mathcal B(\R^n))$, we write $\{X_i\}_{i=1}^n \iid \mu$. 

\subsection{Problem and Method Presentation}
Let $\kappa \in L^\infty(D, \R^{d\times d})$, $f \in H^{-1}(D)$ and $u$ be the weak solution of the partial differential equation (PDE)
\begin{equation}\label{eq:PDE}
\begin{alignedat}{2}
	-\nabla \cdot (\kappa \nabla u) &= f, \quad &&\text{in } D, \\
	u &= 0, &&\text{on } \partial D,
\end{alignedat}
\end{equation}
i.e., the function $u \in V \equiv H^1_0(D)$ satisfying
\begin{equation}\label{eq:PDEWeak}
	a(u, v) = F(v), \qquad a(u, v) \defeq \int_D \kappa \nabla u \cdot \nabla v \dd x, \quad F(v) \defeq \langle f, v \rangle,
\end{equation}
for all functions $v \in V$. We assume there exist positive constants $\underline{\kappa}$ and $\bar \kappa$ such that for all $\xi \in \R^d$
\begin{equation}
	\underline \kappa \norm{\xi}_2^2 \leq \kappa \xi \cdot \xi \leq \bar \kappa \norm{\xi}_2^2,
\end{equation}
where $\norm{\cdot}_2$ is the Euclidean norm on $\R^d$, so that there exist constants $m, M > 0$ such that for all $u, v \in V$ it holds
\begin{equation}
	\abs{a(u, v)} \leq M\norm{u}_V \norm{v}_V, \quad \abs{a(u, u)} \geq m \norm{u}_V^2.
\end{equation}
The Lax--Milgram theorem then guarantees that the problem \eqref{eq:PDEWeak} is well-posed.

Let $N$ be a positive integer and let $\mathcal T_h = \bigcup_{i=1}^N K_i$ be a partition of $D$, where for all $i = 1, \ldots, N$, the element $K_i \subset D$ is a segment, triangle or tetrahedron for $d = 1, 2, 3$ respectively. We denote by $h_i = \mathrm{diam}(K_i)$ the radius of the smallest ball containing $K_i$, and by $h = \max_i h_i$ the maximum radius, indexing the mesh $\mathcal T_h$. We denote by $\mathcal V_h$ the set of all vertices of the elements of $\mathcal T_h$, and in particular as $\mathcal V_h^I \subset \mathcal V_h$ the set of vertices which do not lie on the boundary of $D$, and by $\mathcal V_h^B = \mathcal V_h \setminus \mathcal V_h^I$. Moreover, we denote by $N_I$ the number of internal vertices, i.e., $N_I = \abs{\mathcal V_h^I}$. We assume the partition to be conforming, i.e., if two elements have non-empty intersection, than the latter consists of a point (for $d = 1$), of either a vertex or a side (for $d = 2$), and of either a vertex, a segment or a face (for $d = 3$). We then denote by $V_h \subset V$, $\dim(V_h) < \infty$ the space of continuous piecewise linear finite elements on $\mathcal T_h$, i.e., 
\begin{equation}\label{eq:FESpace}
	V_h \defeq \{v \in V \colon v\eval{K} \in \mathbb P_1, \; \forall K \in \mathcal T_h\},
\end{equation}	
where $\mathbb P_1$ is the space of linear functions. Let us remark that imposing $u_h = 0$ on $\partial D$ yields $\dim(V_h) = N_I$. The FEM proceeds by finding $u_h \in V_h$ such that
\begin{equation}\label{eq:FESolution}
	a(u_h, v_h) = F(v_h),
\end{equation}
for all $v_h \in V_h$, which is equivalent to solving the linear system $A\mathbf u = \mathbf f$, where
\begin{equation}
	 \mathbf u_j = u_h(x_j), \quad x_j \in \mathcal V_h^I, \quad A_{ij} = a(\phi_j, \phi_i), \quad \mathbf f_j = F(\phi_j), \quad i,j = 1, \ldots, N_I,
\end{equation}
and where $\{\phi_j\}_{j=1}^{N-1}$ are the Legendre basis functions defined on the internal vertices of $\mathcal T_h$. The assumptions on $\kappa$ guarantee that $A$ is symmetric positive definite, and in turn that $\mathbf{u}$ is uniquely defined and the problem \eqref{eq:FESolution} is well-posed. 

We now introduce the random-mesh finite element method (RM-FEM), which is based on a random perturbation of the mesh $\mathcal T_h$ obtained by moving the internal vertices. First, we here detail how we build perturbed meshes and which kind of random perturbations we consider to be admissible. Let $p \geq 1$, $\alpha \defeq \{\alpha_i \colon \Omega \to \R^{d}\}_{i=1}^{N_I}$ be a sequence of random variables and let us define the set of internal points $\widetilde{\mathcal V}_h^I = \{\widetilde x_i\}_{i=1}^{N_I}$ where
\begin{equation}\label{eq:PerturbedPoints}
	\widetilde x_i \defeq x_i + h^p \alpha_i.
\end{equation}
We then define the set of perturbed vertices as $\widetilde{\mathcal V}_h = \widetilde{\mathcal V}_h^I \cup \mathcal V_h^B$, i.e., the vertices on the boundary are left unchanged. The perturbed mesh is then simply $\widetilde{\mathcal T}_h = \bigcup_{i=1}^N \widetilde K_i$, where each element $\widetilde K_i$ has the same vertices as its corresponding element $K_i$ in the original mesh, modulo the random perturbation \eqref{eq:PerturbedPoints}. In other words, we compute the internal points of the perturbed mesh following \eqref{eq:PerturbedPoints}, and keep the connectivity structure of the original mesh $\mathcal T_h$. Clearly, the mesh so defined is not conforming for any sequence of random variable $\alpha$, for which we therefore introduce an assumption.

\begin{assumption}\label{as:meshPerturbation} The sequence of random variables $\alpha$ is such that 
	\begin{enumerate}
		\item\label{as:meshPerturbation_sym} its components $\alpha_i$ admit densities $F_{\alpha_i}$ with respect to the Lebesgue measure on $\R^d$, which satisfy $\mathrm{supp}(F_{\alpha_i}) \subset B_{r_i}$, where $B_{r_i} \subset \R^d$ is the ball centered in the origin and of radius $r_i > 0$, and which are radial, i.e., $F_{\alpha_i}(x) = F_{\alpha_i}(\norm{x}_2)$,
		\item\label{as:meshPerturbation_conf} the perturbed mesh $\widetilde{\mathcal T}_h$ is conforming a.s.
	\end{enumerate}
\end{assumption}
Let us remark that the assumption \ref{as:meshPerturbation_sym} actually implies for all $p \geq 1$ the assumption \ref{as:meshPerturbation_conf} a.s., provided the radii $r_i$ are chosen small enough. We assume in \ref{as:meshPerturbation_sym} the densities $F_{\alpha_i}$ to be radial functions so that the random perturbations do not have a privileged direction. 

\begin{example}\label{ex:ExRandomPerturbation} In the one-dimensional case, let $0 = x_0 < x_1 < \ldots < x_N = 1$ so that we have $N_I = N-1$. Denoting $K_i = (x_i, x_{i-1})$ we call $\bar h_i$ the minimum element size for the two intervals sharing the point $x_i$ as a vertex, i.e., $\bar h_i \defeq \min\{h_i, h_{i+1}\}$. Then, a choice of random variables satisfying \cref{as:meshPerturbation} is given by
	\begin{equation}
		\alpha_i =  \left(h^{-1}\bar h_i\right)^p \bar \alpha_i, \quad i = 1, \ldots, N-1, \quad \{\bar \alpha_i\}_{i=1}^{N-1} \iid \mathcal U\left(\left(-\frac12, \frac12\right)\right),
	\end{equation}
	where for a set $D \in \R^d$ we denote by $\mathcal U(D)$ the uniform distribution over $D$. With this choice, indeed, we have that $\widetilde x_i < \widetilde x_{i-1}$ a.s., and therefore the perturbed mesh is conforming. In the two-dimensional case, we introduce for $i = 1, \ldots, N_I$ the notation
  	\begin{equation}
  		\Delta_i = \{K \in \mathcal T_h\colon K \text{ has } x_i \text{ as a vertex} \}.
  	\end{equation}
  	Analogously to the one-dimensional case, we write $\bar h_i \defeq \min_{j:K_j \in \Delta_i}  h_j$. In this case, it is possible to verify that choosing for all $i = 1, \ldots, N_I$
  	\begin{equation}
  		\alpha_i = (h^{-1}\bar h_i)^p \bar \alpha_i, \quad i = 1, \ldots, N_I, \quad \{\bar \alpha_i\}_{i=1}^{N_I} \iid \mathcal U\left(B_{1/2}\right),
  	\end{equation}
  	then $\alpha$ satisfies \cref{as:meshPerturbation}. We verify this graphically in \cref{fig:Mesh}, where we show a realization of the perturbed mesh based on a generic Delaunay mesh and on a structured mesh on $D = (0,1)^2$ along with the sets where the perturbed points are constrained to belong a.s. We notice that for $p > 1$ the magnitude of the perturbations clearly tends to vanish. Finally, we remark that similar admissible perturbations can be introduced in higher dimensions.
\end{example}

\begin{figure}
	\centering
	\begin{tabular}{ccc}
		\includegraphics[]{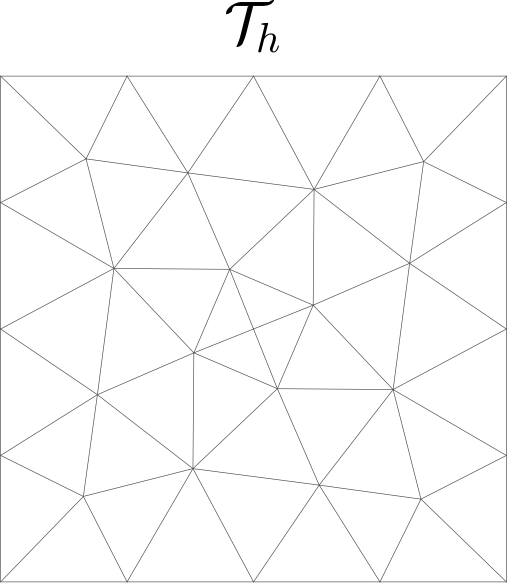} & \includegraphics[]{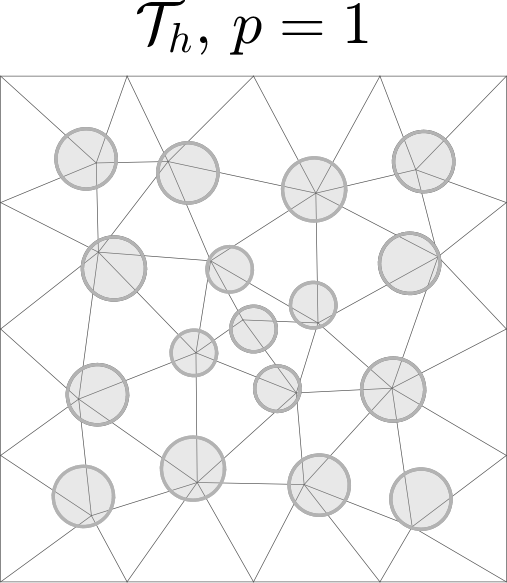} & \includegraphics[]{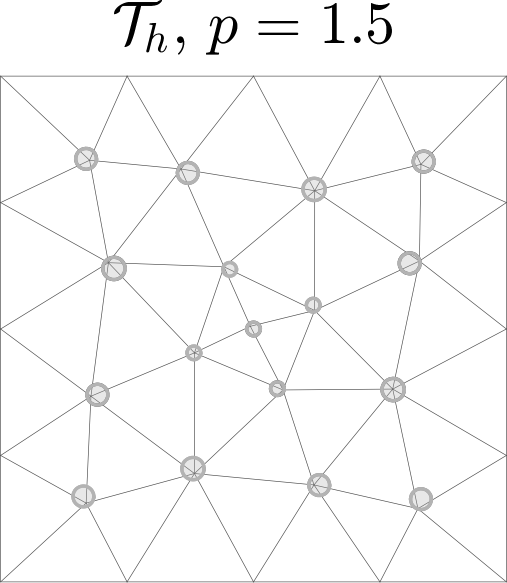} \\
		\includegraphics[]{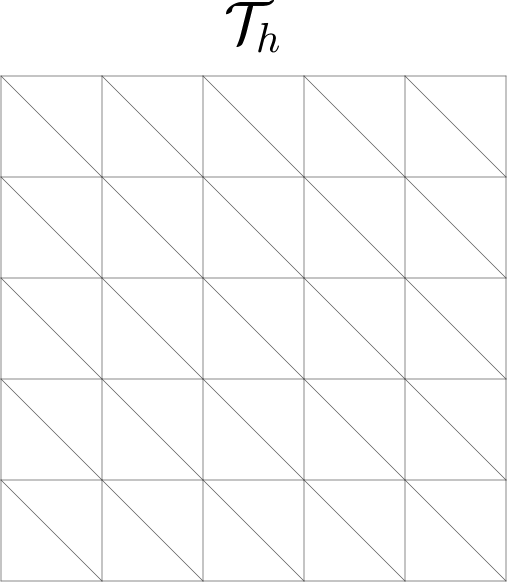} & \includegraphics[]{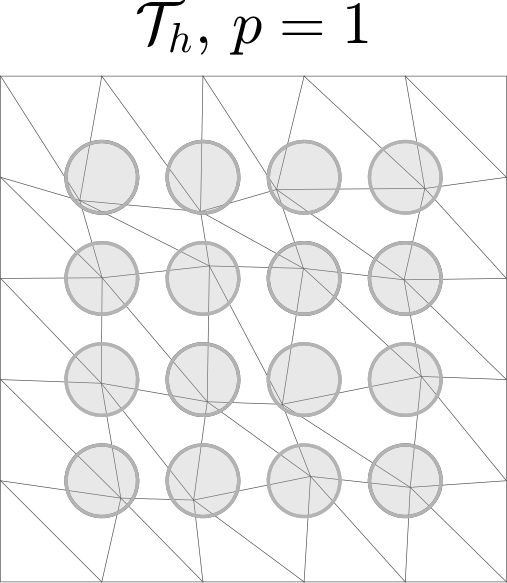} & \includegraphics[]{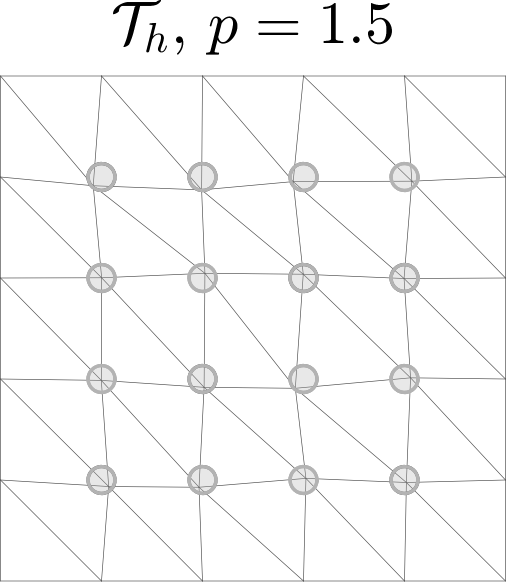}
	\end{tabular}
	\caption{A realization of $\widetilde{\mathcal T}_h$ for $p = \{1, 1.5\}$ based on two meshes $\mathcal T_h$ of $D = (0,1)^2$. On the first line, a Delaunay mesh. On the second line, a structured mesh. The regions where the perturbed points are included a.s. are depicted by light grey circles.}
	\label{fig:Mesh}
\end{figure}

Having defined the perturbed mesh, we now proceed with describing the RM-FEM. Let $\widetilde V_h$ be the space of continuous piecewise linear finite elements on $\widetilde {\mathcal T}_h$. Let moreover $\{\widetilde \phi_i\}_{i=1}^{N_I}$ be the Legendre basis functions defined on the internal vertices of $\widetilde {\mathcal T}_h$ and $\widetilde {\mathcal I} \colon \mathcal C^0(D) \cap V \to \widetilde V_h$ be the Lagrange interpolation operator onto $\widetilde V_h$, i.e., for a function $v \in \mathcal C^0(D) \cap V$ and for $x \in D$ we define
\begin{equation}
	\widetilde{\mathcal I}v(x) \defeq \sum_{i=1}^{N_I} v(x_i) \widetilde \phi_i(x). 
\end{equation}
We are then interested in the two functions belonging to the finite element space $\widetilde V_h$ whose definition we give below.
\begin{definition}\label{def:ProbInterp} Let $u_h \in V_h$ be defined in \eqref{eq:FESolution}. We define the RM-FEM interpolant as the random function $\widetilde{\mathcal I} u_h \in \widetilde V_h$, where $\widetilde{\mathcal I}$ is the Lagrange interpolant onto $\widetilde V_h$.
\end{definition}
\begin{definition}\label{def:ProbSol} Given the random finite element space $\widetilde V_h$, we define the RM-FEM solution as the unique random function $\widetilde u_h \in \widetilde V_h$ such that
\begin{equation}
	a(\widetilde u_h, \widetilde v_h) = F(\widetilde v_h),
\end{equation}
for all $\widetilde v_h \in \widetilde V_h$.	
\end{definition}

\begin{remark} Clearly, either for any fixed $p \geq 1$ and $h \to 0$ or for any fixed $h < 1$ and $p \to \infty$, the functions $u_h$, $\widetilde{\mathcal I} u_h$ and $\widetilde u_h$ tend to coincide. We visualize this for $u_h$ and $\widetilde u_h$ in \cref{fig:RMFEM}, where we simply fix $\kappa = 1$ and the right-hand side $f$ such that $u = \sin(2\pi x)$ in \eqref{eq:PDE}, and consider the effects of increasing $p$ and decreasing $h$. For this simple problem, we notice that for $p = 2$ and $N = 20$ the FEM solution $u_h$ and the RM-FEM solution $\widetilde u_h$ are almost indistinguishable.
\end{remark}

\begin{remark} All the quantities distinguished by a tilde (e.g., $\widetilde{\mathcal T}_h$, $\widetilde V_h$, $\widetilde{\mathcal I}$) are random variables with values in appropriate spaces. For example $\widetilde u_h$ is a random function $\widetilde u_h \colon \Omega \times D \to \R$, such that $\Omega \times D \ni (\omega, x) \mapsto \widetilde u_h(\omega, x)$. For economy of notation, in the following we drop the argument $\omega$ from all random variables.
\end{remark}

\begin{remark}\label{rem:pCoeff} The coefficient $p$ in \eqref{eq:PerturbedPoints} has the same role as the coefficient identified by the same symbol in both \cite{AbG20, CGS17}, i.e., it controls the variability of the probabilistic solutions by tuning the variability of the noise which is applied to the method. 
\end{remark}

\begin{remark}\label{rem:BoundaryPoints} Let us remark that the RM-FEM interpolant $\widetilde{\mathcal I} u_h$ is well-defined even allowing the vertices of $\mathcal T_h$ which lay on the boundary $\partial D$ to be perturbed, as far as the perturbation moves them inside the domain $D$. The random RM-FEM interpolant $\widetilde {\mathcal I} u_h$ does not in this case belong to the space $V$ in this case since it is not defined on the whole domain $D$ and does not satisfy boundary conditions. For practical applications, one can nevertheless employ the RM-FEM interpolant defined on a smaller domain, which results from a perturbation of all vertices of $\mathcal T_h$, including those on the boundaries.	
\end{remark}

\begin{figure}[t]
	\centering
	\begin{tabular}{ccc}
		\includegraphics[]{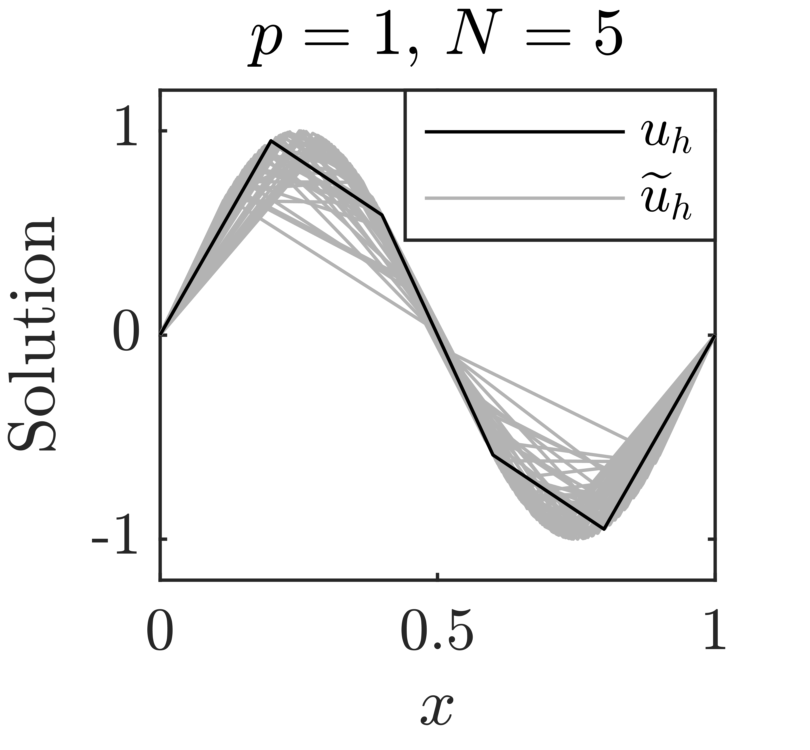} & \includegraphics[]{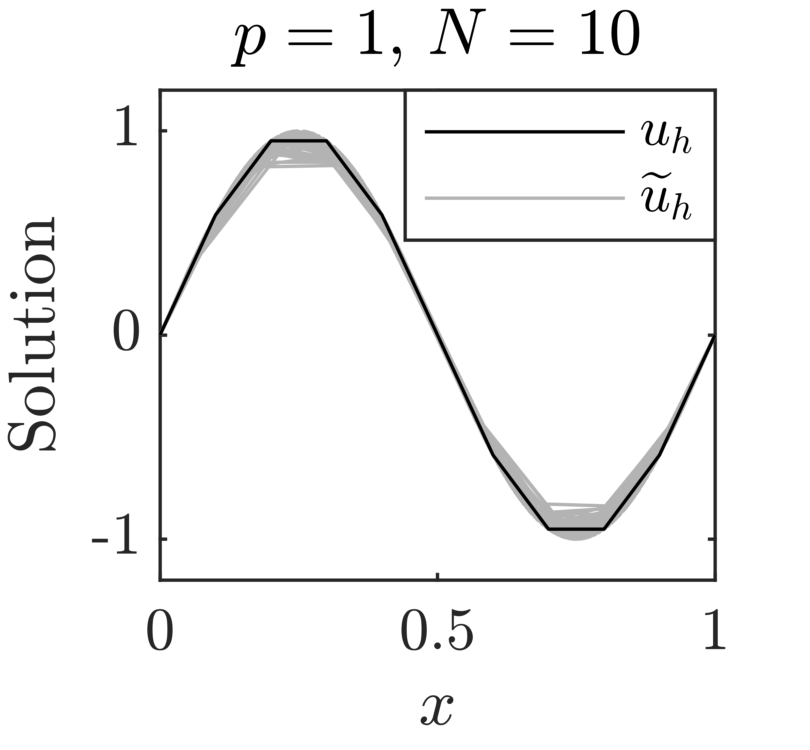} & \includegraphics[]{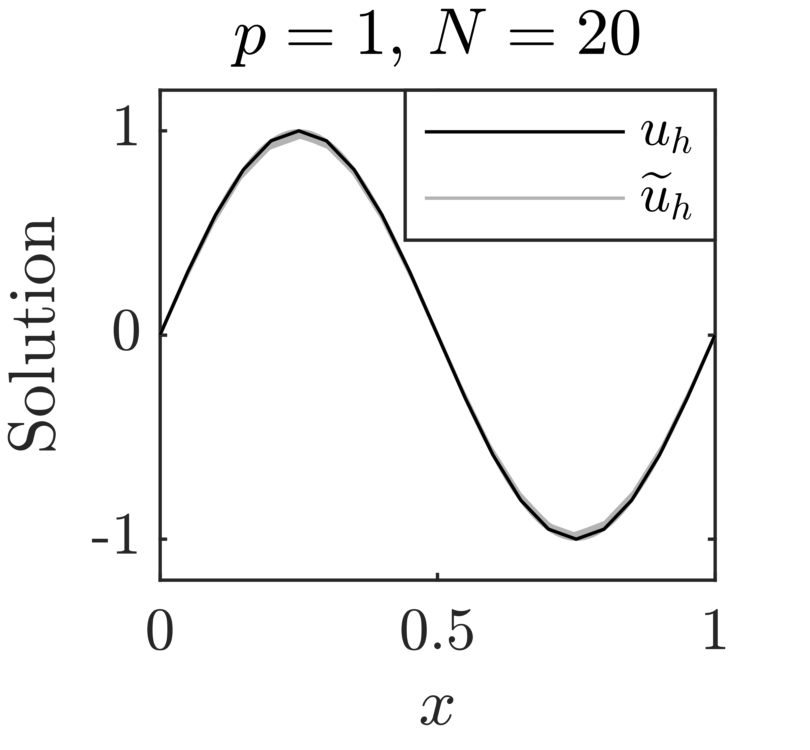} \vspace{0.1cm}\\ 
		\includegraphics[]{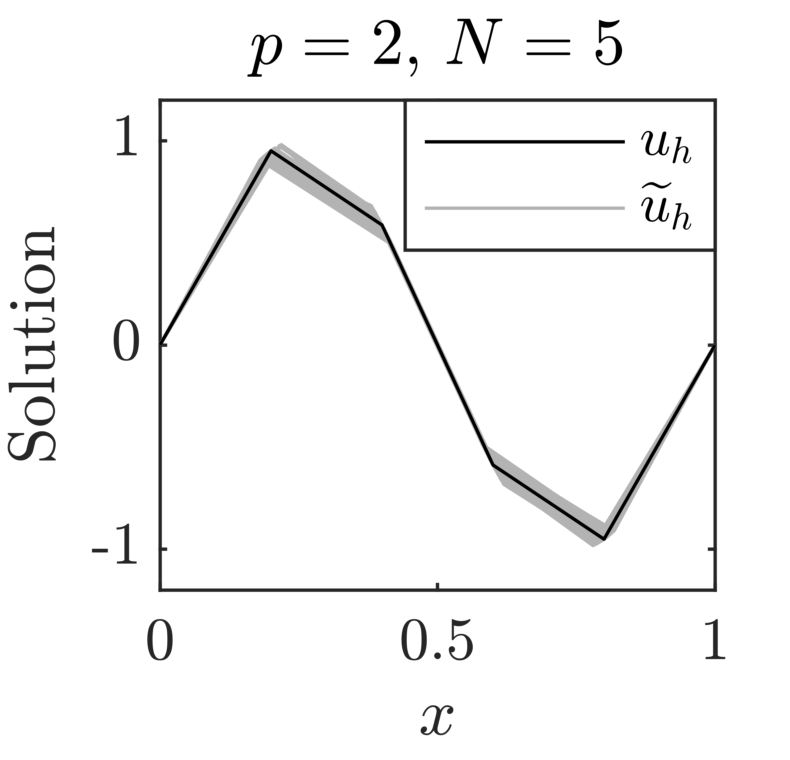} & \includegraphics[]{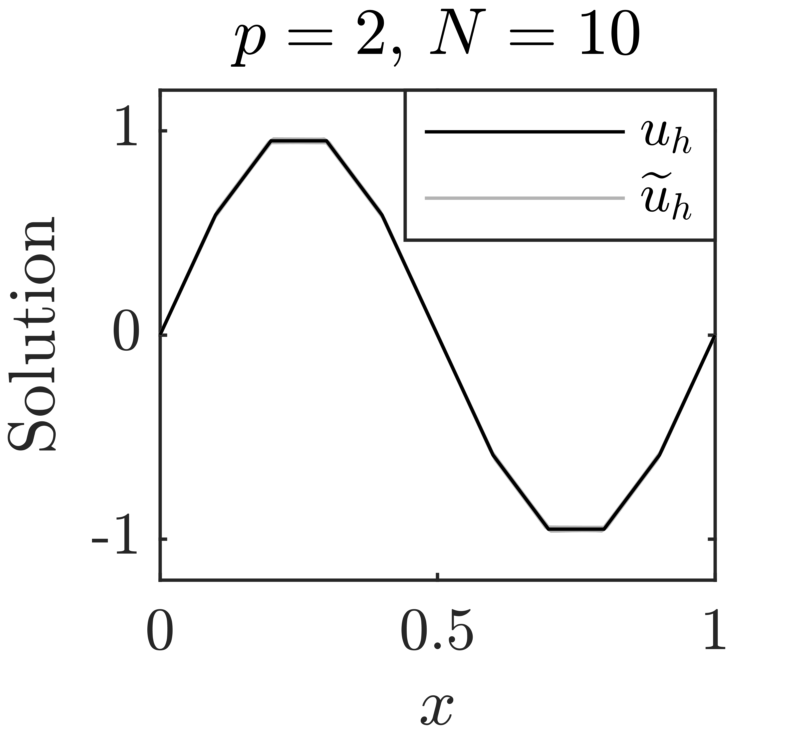} & \includegraphics[]{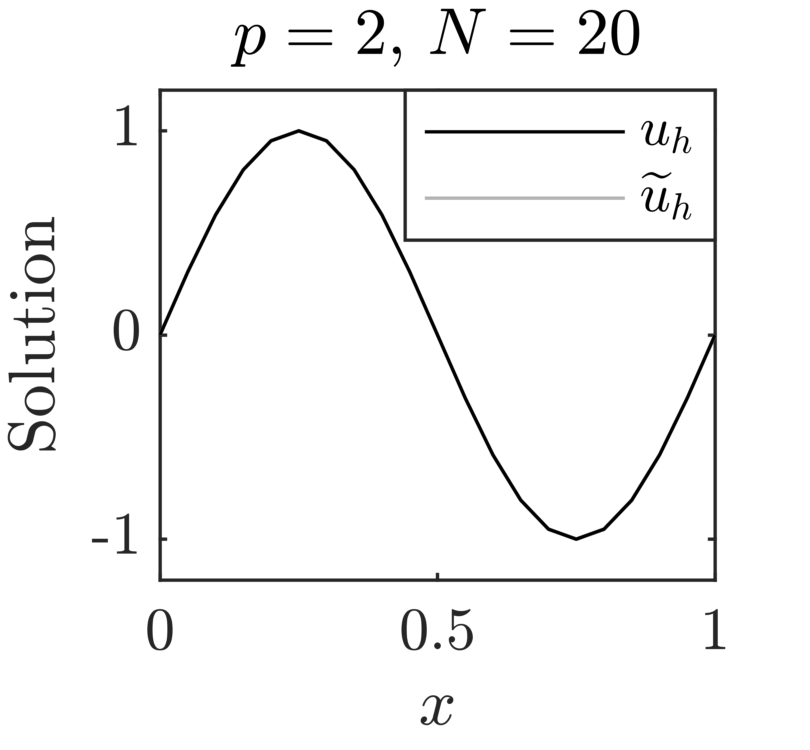}
	\end{tabular}
	\caption{Comparison between the RM-FEM and the FEM solutions. We display the solution $u_h$ and $50$ realizations of $\widetilde u_h$, by row respectively for $p = \{1, 2\}$ and by column for $N = \{5, 10, 20\}$. }
	\label{fig:RMFEM}
\end{figure} 

Before proceeding with the two main applications of the RM-FEM, i.e., a posteriori error estimators and Bayesian inverse problems, we state an a priori error estimate, which suggests how to balance the sources of error due to numerical discretization and to the randomization of the method, respectively.

\begin{theorem}\label{thm:APriori} Let the solution $u$ of \eqref{eq:PDE} be $u\in H^2(D)$. Then, it holds
	\begin{equation}
		\norm{\widetilde u_h - u}_V \leq Ch, \quad \text{a.s.},
	\end{equation}
	for a constant $C > 0$ independent of $h$. Moreover, if $p = 1$ in \eqref{eq:PerturbedPoints} the numerical and discretization errors are balanced with respect to $h$, i.e., it holds
	\begin{equation}
		\norm{u_h - \widetilde u_h}_V = \mathcal O(h) = \mathcal O(\norm{u - u_h}_V), \quad \text{a.s.}
	\end{equation}
\end{theorem}

This results indicates that one should fix $p = 1$ in \eqref{eq:PerturbedPoints} in order to obtain a family of probabilistic solutions whose statistical properties should reflect the true error. This is crucial when the RM-FEM is employed in a pipeline of computations such as Bayesian inverse problems, which will be presented in detail in \cref{sec:BIP}. The proof of \cref{thm:APriori} is elementary and discussed in \cref{sec:APriori}.

\section{A Posteriori Error Estimators based on the RM-FEM}\label{sec:APosteriori}

The first and foremost application of the RM-FEM is deriving a posteriori error estimators which are entirely based on the statistical information carried on by the mesh perturbation. We say that a quantity $\mathcal E_h$ is an a posteriori error estimator if it gives an error estimate on the numerical approximation and is computable only by knowledge of the numerical solution. Moreover, if there exist constants $C_{\mathrm{up}}$ and $C_{\mathrm{low}}$ independent of $h$ and of $u$ such that 
\begin{equation}\label{eq:APosterioriBound}
	C_{\mathrm{low}}\mathcal E_h \leq \norm{u-u_h}_V \leq  C_{\mathrm{up}} \mathcal E_h,
\end{equation}
we say that the a posteriori error estimator is reliable and efficient, respectively. Indeed, the upper bound above guarantees that when the estimator is small, so is the numerical error. The lower bound, instead, gives an insurance on the quality of the estimator, as it shows that the estimation of the error is not exceedingly pessimistic. There exist in the literature a huge number of a posteriori error estimators, and we refer the reader to the surveys given e.g. in \cite{Ver13, AiO00}. Most a posteriori error estimators are expressed in the form
\begin{equation}\label{eq:GenericAPosteriori}
	\mathcal E_h = \left(\sum_{K \in \mathcal T_h} \eta_K^2\right)^{1/2},
\end{equation} 
where the $\eta_K$ are local quantities depending on the solution and the data on the element $K$ and its neighbors. For example, in the two-dimensional case a valid a posteriori error estimator is given by the expression of its local components
\begin{equation}\label{eq:APosteriori2D}
	\eta_K^2 = h_K^2 \norm{f}^2_{L^2(K)} + h_K \norm{\dbrack{\nabla u_h \cdot \nu_K}}_{L^2(\partial K)}^2,
\end{equation}
where $\dbrack{\cdot}$ is the jump operator and $\nu_K$ denotes the unitary vector normal to the boundary of $K$ (see e.g. \cite[Section 3]{Ver94} or \cite[Chapter 2]{AiO00}). Other a posteriori error estimators are based on recovered gradients, which are employed as surrogates of the gradient of the exact solution to estimate the error. A notable member of these methodologies is the Zienkiewicz--Zhu (ZZ) patch recovery technique \cite{ZiZ92b, ZiZ92}, which is proved to be superconvergent on special meshes, and which is in practice widely employed on any mesh. 

It has been heuristically noted for ODEs in \cite{BHT20, ChC19, SSH19} that information on the variability of a probabilistic solution can be employed to estimate the error and thus adapt the numerical discretization. Indeed, building probabilistic solution to otherwise deterministic problems should pursue the goal of quantifying numerical errors through uncertainty. Guided by this observation, we now introduce two probabilistic error estimators for elliptic PDEs.

\begin{definition}\label{def:ProbErrEst} Let $\widetilde{\mathcal I} u_h$ be the RM-FEM interpolant defined in \cref{def:ProbInterp} and for each $K \in \mathcal T_h$, let us denote by $\widetilde K \in \widetilde{\mathcal T}_h$ its corresponding element in $\widetilde {\mathcal T}_h$. We define the first RM-FEM a posteriori error estimator as
	\begin{equation}
		\widetilde{\mathcal E}_{h,1} \defeq \left(\sum_{K \in \mathcal T_h} \widetilde{\eta}_{K,1}^2\right)^{1/2}, \quad \text{ with } \quad \widetilde{\eta}_{K,1}^2 = h_K^{-(p-1)} \Eb{\norm{\nabla (u_h - \widetilde{\mathcal I} u_h)}_{L^2(\widetilde K)}^2}.
	\end{equation}
	Moreover, we define the second RM-FEM a posteriori error estimator as
	\begin{equation}
		\widetilde{\mathcal E}_{h,2} \defeq \left(\sum_{K \in \mathcal T_h} \widetilde{\eta}_{K,2}^2\right)^{1/2}, \quad \text{ with } \quad \widetilde{\eta}_{K,2}^2 = h_K^{-(2p-2)} \abs{K} \Eb{\norm{\nabla u_h\eval{K} - \nabla \widetilde{\mathcal I} u_h\eval{\widetilde K}}^2}.
	\end{equation}
\end{definition}

\begin{remark} The scaling factors $h_K^{-(p-1)}$ and $h_K^{-(2p-2)}$ in the definition of $\widetilde{\mathcal \eta}_{K,1}$ and $\widetilde{\mathcal \eta}_{K,2}$ are necessary to obtain well-calibrated error estimators. This is made clearer in the one-dimensional case by the analysis presented in \cref{sec:ErrAnalysis}. For higher dimensions, they can be partially explained with the ansatz \eqref{eq:APrioriGenP}, especially for the first estimator $\widetilde{\mathcal E}_{h,1}$, and they appear in practice to be the correct scaling. 
\end{remark}

\begin{remark}\label{rem:SuperMesh} Computing the estimator $\widetilde {\mathcal E}_{h,1}$ is more involved than the estimator $\widetilde {\mathcal E}_{h,2}$. Indeed, for the latter it is sufficient to compute the interpolant $\widetilde{\mathcal I} u_h$ and the gradients over each element of $u_h$ and of the interpolant. For $\widetilde {\mathcal E}_{h,1}$, instead, one has to compute on each element $\widetilde K$ the quantity
	\begin{equation}
		\norm{\nabla(u_h - \widetilde{\mathcal I} u_h)}_{L^2(\widetilde K)}.
	\end{equation}
	By construction, each element $\widetilde K$ overlaps with the elements corresponding to its neighbors in the original mesh in a non-trivial manner, and if $d > 1$ one has to rely to the construction of a ``super-mesh'' (see e.g. \cite{CGR18,CrF20}) such that on each of its elements the quantity $\nabla(u_h - \widetilde{\mathcal I} u_h)$ is constant. A super-mesh has to be built for each realization of the perturbed mesh $\widetilde{\mathcal T}_h$, which could therefore be expensive.
\end{remark}

In this article, we show in the one-dimensional case that the estimators given in \cref{def:ProbErrEst} are reliable and efficient in the sense of \eqref{eq:APosterioriBound}. In the statement of our theoretical result, which is given below, we make use of a quantity $\Lambda \in \R$ which is of higher order in most practical scenarios and which is defined as
\begin{equation}\label{eq:BabuskaLambdaIntro}
	\Lambda^2 \defeq h^\zeta \sum_{j=1}^N \int_{K_j} (f(x) + C_j)^2 \dd x,
\end{equation}
where for each $K_j$ the real constant $C_j$ will be specified in the analysis of \cref{sec:ErrAnalysis} (see e.g. \cite[Equation (8.7)]{BaR81}). Moreover, we consider one-dimensional meshes which are $\lambda$-quasi-uniform, i.e., we assume there exists a constant $\lambda \in (1, \infty)$ such that it holds
\begin{equation}
	\frac1\lambda \leq \frac{h_j}{h_{j-1}} \leq \lambda, \quad j = 2, \ldots, N,
\end{equation}
uniformly in $h$. Finally, we consider perturbations satisfying
\begin{equation}
	\alpha_i =  \left(h^{-1}\bar h_i\right)^p \bar \alpha_i, \quad i = 1, \ldots, N-1,
\end{equation}
where $\bar h_i = \min\{h_i,h_{i+1}\}$ and for a i.i.d. sequence of random variables $\{\bar\alpha_i\}_{i=1}^{N-1}$ such that $\abs{\bar\alpha_1} \leq 1/2$ a.s. These perturbations are indeed the same as the ones presented in \cref{ex:ExRandomPerturbation}, but without the assumption of $\{\bar\alpha_i\}_{i=1}^N$ to be uniformly distributed, which is not necessary in the following. In practice, a uniform distribution is nevertheless advisable, as it is still general enough and satisfies the radial assumption of \cref{as:meshPerturbation}\ref{as:meshPerturbation_sym}. We moreover introduce the following technical assumption on the perturbation.
\begin{assumption}\label{as:AssumptionAPosteriori} Let the family of meshes $\mathcal T_h$ be $\lambda$-quasi-uniform, let $p$ be the coefficient introduced in \eqref{eq:PerturbedPoints}  and assume that for all $h$ and $p$ there exists $C > 0$ such that
	\begin{equation}
		4h^{p-1}\frac{\E\abs{\bar \alpha_1}^2}{\E\abs{\bar \alpha_1}} + C < 1+\lambda^{-(p-1)}.
	\end{equation}
\end{assumption}
\begin{remark}\label{rem:AssumptionAPosteriori} We note that \cref{as:AssumptionAPosteriori} holds for $p > 1$ and $h$ sufficiently small, and is therefore not restrictive in practice.
\end{remark}

We can now state the main result involving a posteriori error estimators.

\begin{theorem}\label{thm:MainThmAPosteriori} Let the dimension $d = 1$, let $p > 1$ in \eqref{eq:PerturbedPoints} and let \cref{as:meshPerturbation} hold. Moreover, let $\widetilde{\mathcal E}_{h,1}$, $\widetilde{\mathcal E}_{h,2}$ and $\Lambda$ be given in \cref{def:ProbErrEst} and \eqref{eq:BabuskaLambdaIntro} respectively and let the family of meshes $\mathcal T_h$ be $\lambda$-quasi-uniform. Then, there exists $C > 0$ independent of $h$ and of the solution $u$ such that it holds for $k \in \{1,2\}$
	\begin{equation}
		\norm{u-u_h}_V \leq \widetilde C (\widetilde{\mathcal E}_{h,k}^2 + \Lambda^2)^{1/2},
	\end{equation}
	up to higher order terms in $h$ and under \cref{as:AssumptionAPosteriori} for $k=1$. If additionally $\kappa \in \mathcal C^2(D)$ and $f \in \mathcal C^1(D)$, then there exist constants $\widetilde C_{\mathrm{low}}$ and $\widetilde C_{\mathrm{up}}$ independent of $h$ and of the solution $u$ such that for $k \in \{1,2\}$ it holds
	\begin{equation}
		\widetilde C_{\mathrm{low}} \widetilde{\mathcal E}_{h,k} \leq \norm{u-u_h}_V \leq \widetilde C_{\mathrm{up}} \widetilde{\mathcal E}_{h,k},
	\end{equation}
	up to higher order terms in $h$ and under \cref{as:AssumptionAPosteriori} for $k=1$.
\end{theorem}
Let us notice that the estimators given in \cref{def:ProbErrEst} involve the computation of an expectation with respect to the random perturbations of the mesh, and therefore a Monte Carlo simulation is needed in practice. Let $N_{\mathrm{MC}}$ be a positive integer, $k \in \{1,2\}$ and $\{\widetilde E_{h,k}^{(i)}\}_{i=1}^{M}$ be i.i.d. realizations of the estimator $\widetilde{\mathcal E}_{h,k}$, obtained with independent perturbations of the mesh. Then, in practice we compute
\begin{equation}\label{eq:ErrEstMC}
	\widetilde E_{h,k} \defeq \frac1{N_{\mathrm{MC}}} \sum_{i=1}^{N_{\mathrm{MC}}} \widetilde E_{h,k}^{(i)}.
\end{equation}

\begin{remark}\label{rem:CompCost} It could be suggested that the application of Monte Carlo techniques increases dramatically the simulation time. We argue that in practice the computational overhead is not relevant, mainly for three reasons. First, it has been proved in \cite{AbG20} that the variance of Monte Carlo estimators drawn from probabilistic numerical methods decreases with respect to the discretization size $h$. Hence, the number of simulations $M$ does not need to be large, nor increasing if $h \to 0$, to guarantee a good quality of the estimator. The same arguments hold for the RM-FEM, too. Second, the Monte Carlo estimation is completely parallelizable, thus reducing the cost by a factor equal to the number of available computing units. Finally, the computation of the RM-FEM interpolant $\widetilde{\mathcal I}u_h$ is not computationally involved, neither when it is repeated $N_{\mathrm{MC}}$ times.
\end{remark}

\subsection{Numerical Experiments}\label{sec:NumExp_ErrEst}

We now present numerical experiments on one and two-dimensional test cases to demonstrate the validity of our a posteriori error estimators. In particular, we are interested in determining whether the probabilistic error estimators introduced in \cref{def:ProbErrEst} are indeed reliable estimators for the numerical error in the FEM, and in employing these estimators for local refinements of the mesh. Setting a tolerance $\gamma > 0$, our goal is building a mesh $\mathcal T_h$ such that
\begin{equation}\label{eq:APosterioriGoal}
	\frac{\norm{u-u_h}_V}{\norm{u_h}_V} \leq \gamma.
\end{equation}
Replacing the numerator with $\widetilde{\mathcal E}_{h,k}$, $k \in \{1,2\}$, we notice that the condition \eqref{eq:APosterioriGoal} is satisfied if it holds for all $K \in \mathcal T_h$
\begin{equation}\label{eq:APosterioriLocalCondition}
	\widetilde \eta_{K,k} \leq \frac{\gamma \norm{u_h}_V}{\widetilde C_{\mathrm{up}}\sqrt{N}} \eqdef \gamma_{\mathrm{loc}}.
\end{equation}
Indeed, in this case
\begin{equation}
	\norm{u-u_h}_V^2 \leq \widetilde C_{\mathrm{up}}^2 \widetilde {\mathcal E}_{h,k}^2 =  \widetilde C_{\mathrm{up}}^2 \sum_{K\in \mathcal T_h} \widetilde \eta_{K,k}^2 \leq \gamma^2 \norm{u_h}_V^2,
\end{equation}
and thus \eqref{eq:APosterioriGoal} holds. Let us remark that $\widetilde C_{\mathrm{up}}^2$ is not known a priori in practice, and therefore we just decide to employ the condition \eqref{eq:APosterioriLocalCondition} fixing $\widetilde C_{\mathrm{up}} = 1$ in our experiments. We therefore adapt the mesh by computing the local contributions and comparing them with $\gamma_{\mathrm{loc}}$, thus locally refining the mesh if the condition \eqref{eq:APosterioriLocalCondition} is not met, and coarsening if the local estimators are excessively small with respect to $\gamma_{\mathrm{loc}}$.

In the following we employ for both the one and the two-dimensional cases the uniform distributions given in \cref{ex:ExRandomPerturbation} for the random perturbations of the points. In light of \cref{lem:EquivProb1Jump} and \cref{lem:EquivProb2Jump}, we decide to correct the estimators by normalizing them with respect to the random perturbations. In particular, in the following, the estimators are normalized as $\widetilde{\mathcal E}_{h,1} \leftarrow \widetilde{\mathcal E}_{h,1} / \E\norm{\bar \alpha_1}$ and $\widetilde{\mathcal E}_{h,2} \leftarrow \widetilde{\mathcal E}_{h,2} / \E\norm{\bar \alpha_1}^2$.

\subsubsection{One-Dimensional Case}\label{sec:NumExp_1D}
\begin{figure}[t!]
	\centering
	\begin{tabular}{cc}
		\includegraphics[]{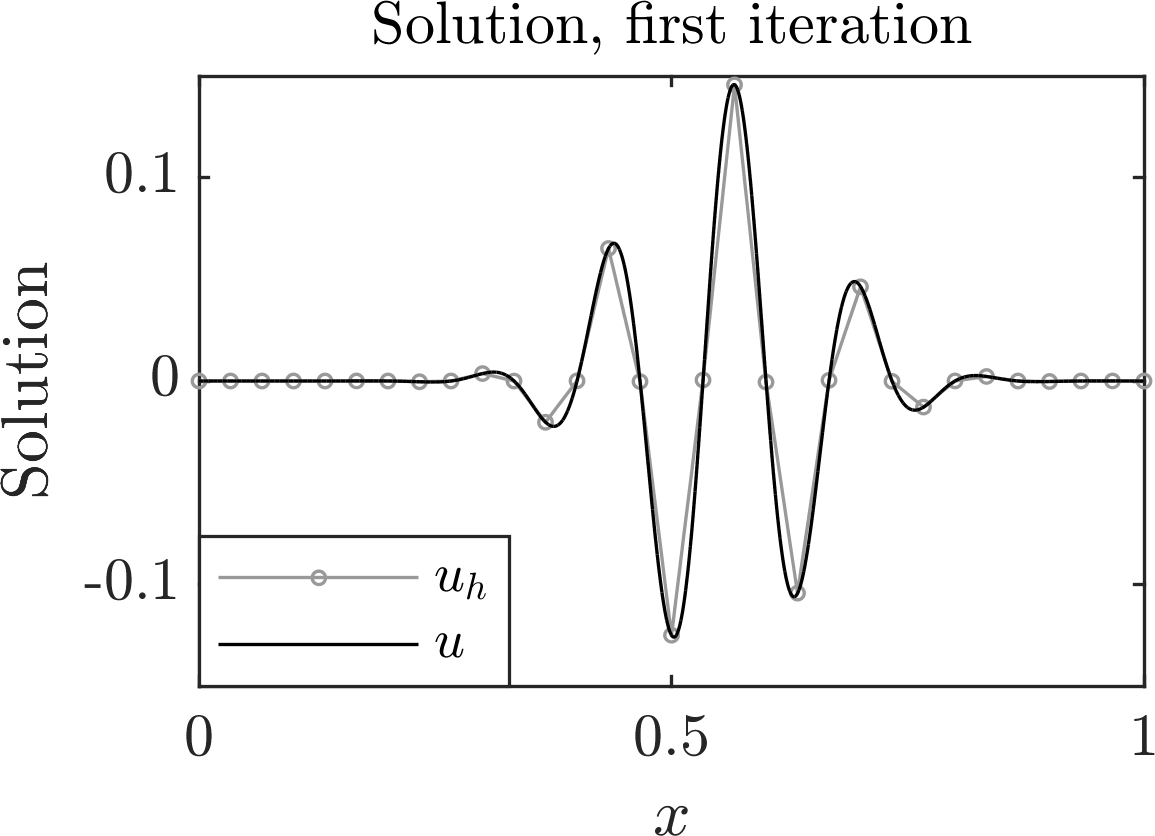} & \includegraphics[]{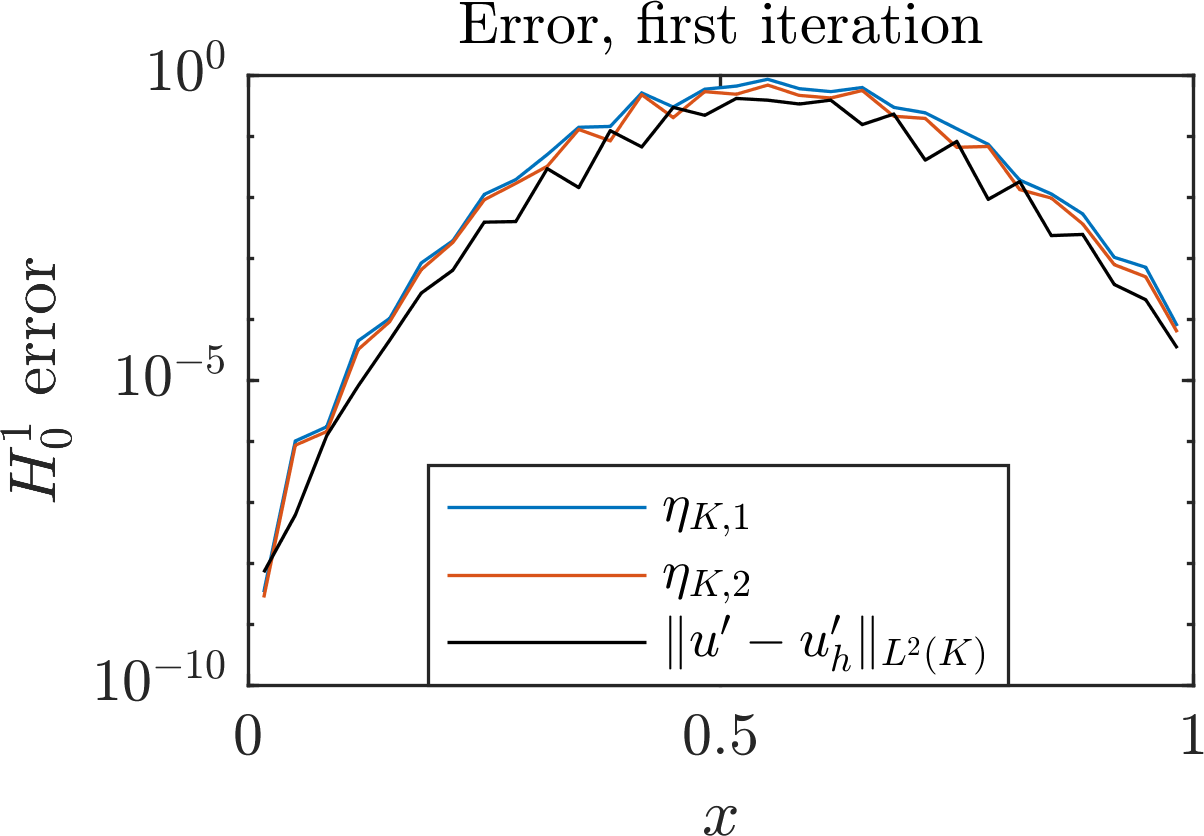} \vspace{0.1cm} \\
		\includegraphics[]{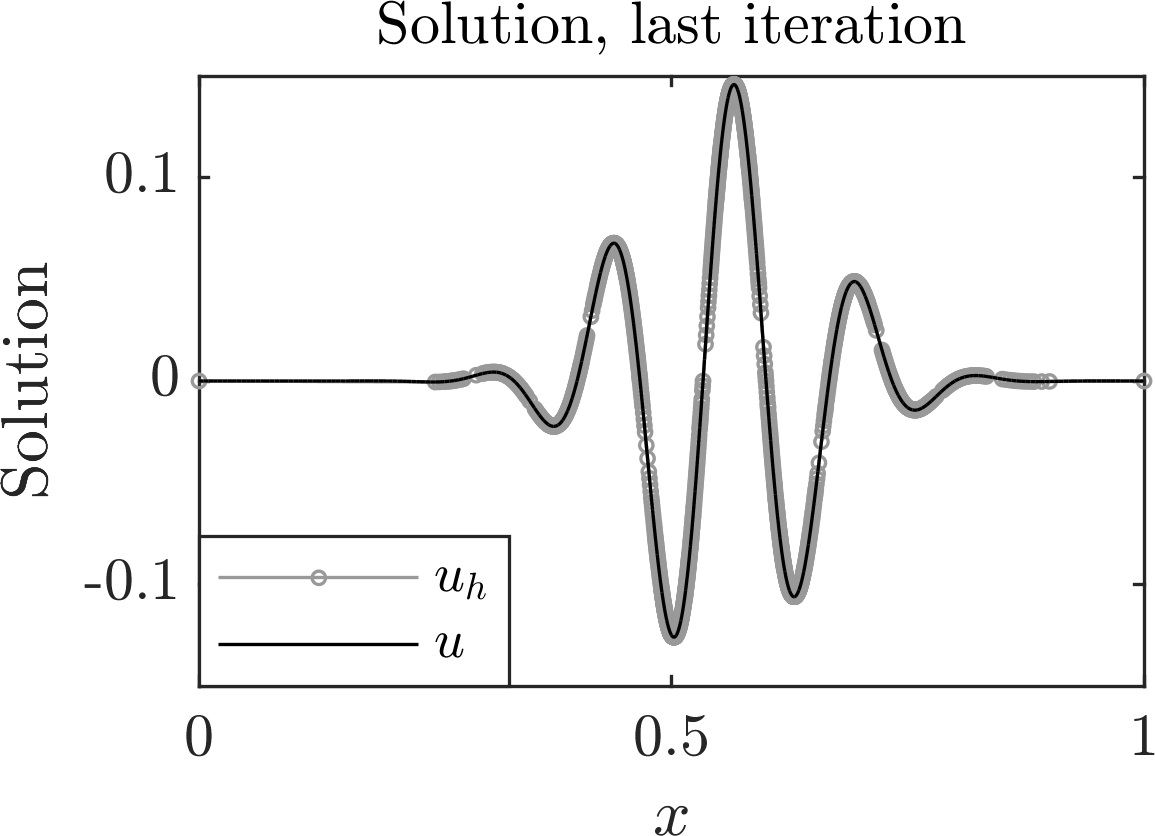} & \includegraphics[]{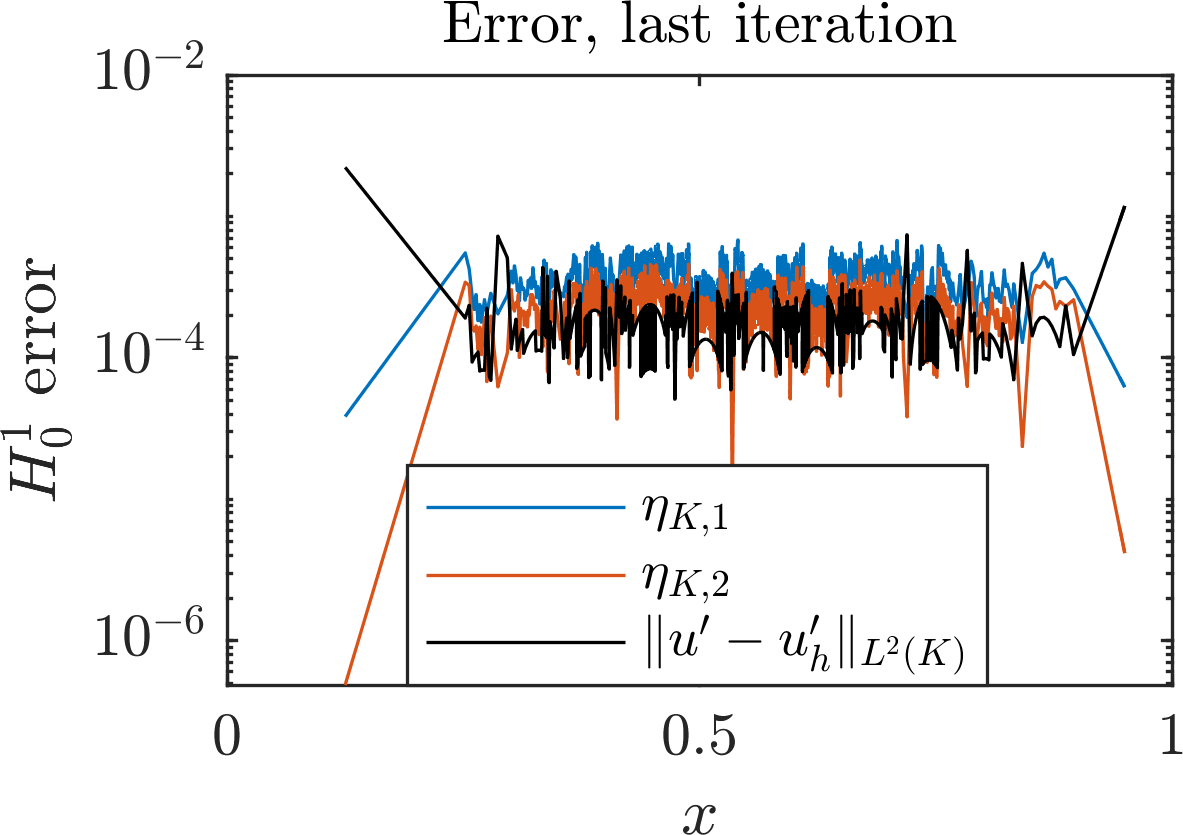} \vspace{0.1cm} \\
		\includegraphics[]{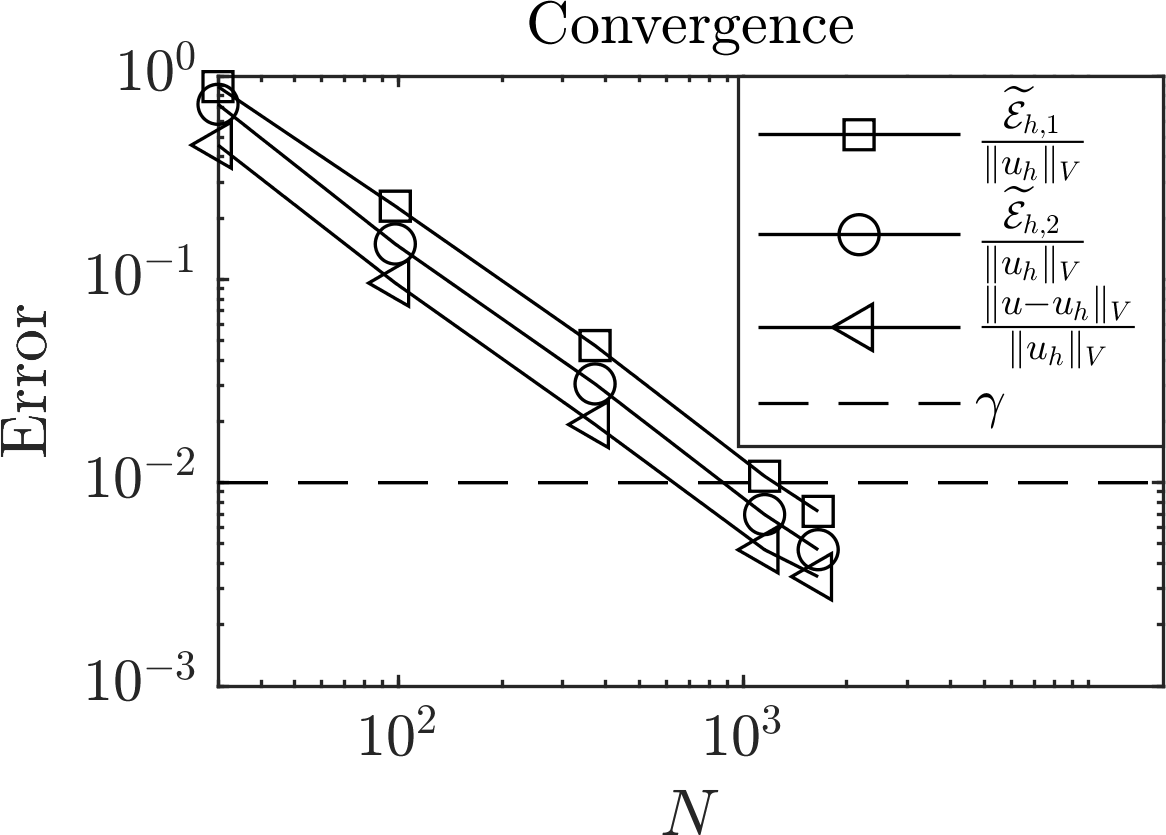} & \includegraphics[]{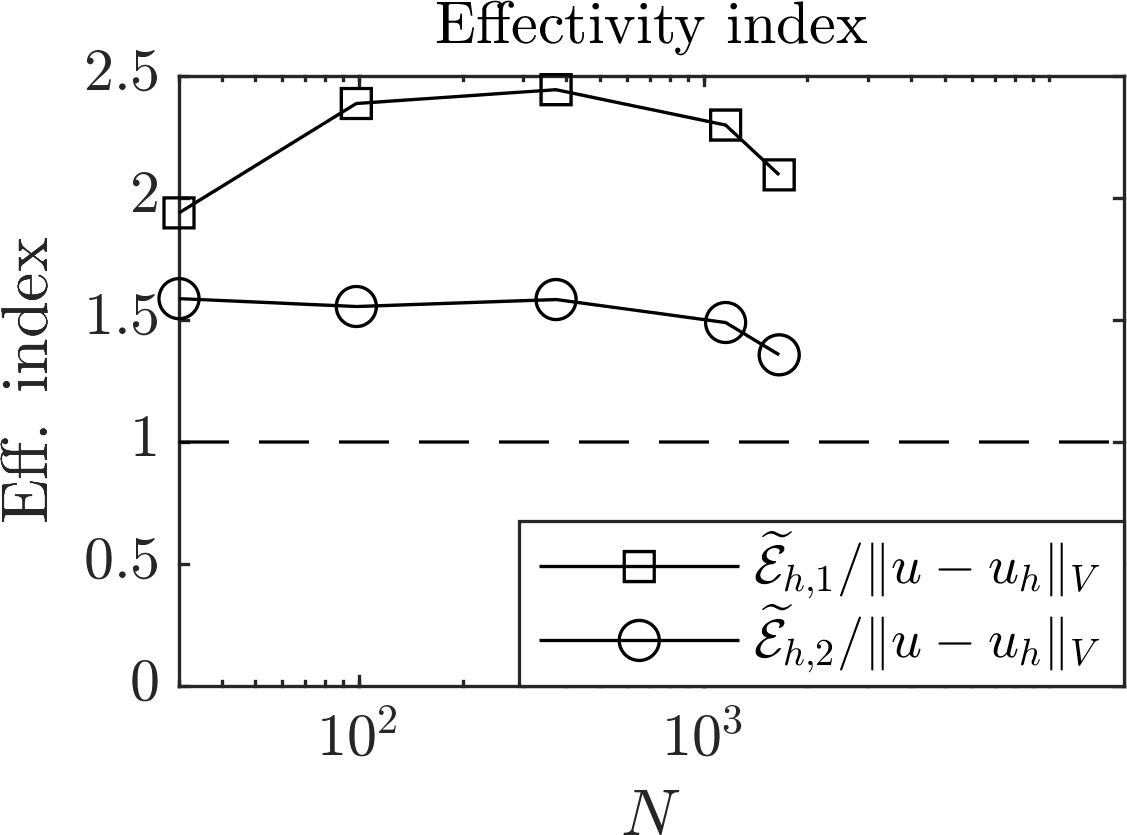} 
	\end{tabular}
	\caption{Results for the one-dimensional experiment of \cref{sec:NumExp_ErrEst}. First and second rows: numerical and exact solutions $u_h$ and $u$ on the left, local contributions to the error estimators of \cref{def:ProbErrEst} compared with the true error $\norm{u-u_h}_K$ on the right, at initialization and termination of the adaptivity procedure. Third row: on the left convergence of the global error $\norm{u-u_h}_V$ and of the estimator $\mathcal E_h$ until the tolerance $\gamma$, on the right the effectivity index.}
	\label{fig:ErrEst}
\end{figure}

We first consider $d = 1$ and the two-point boundary value problem \eqref{eq:PDE_1d} with $\kappa$ and the exact solution $u$ given by
\begin{equation}
	\kappa(x) = 1 + x^3, \quad u(x) = x^3 \sin(a\pi x) \exp(-b(x-0.5)^2) ,
\end{equation}
where we fix $a = 15$ and $b = 50$, and where we choose the right-hand side $f$ so that $u$ is indeed the solution. As a goal, we set the tolerance $\gamma = 10^{-2}$ in \eqref{eq:APosterioriGoal} and stop the algorithm when condition \eqref{eq:APosterioriLocalCondition} is met by all elements of the mesh. We consider the RM-FEM implemented with uniform random variables as in \cref{ex:ExRandomPerturbation} and fix $p = 3$ in \eqref{eq:PerturbedPoints}. Moreover, we consider $N_{\mathrm{MC}} = 20$ realizations of the probabilistic mesh to approximate the error estimator as in \eqref{eq:ErrEstMC}. We then compute both the error estimators given in \cref{def:ProbErrEst} and employ $\widetilde{\mathcal E}_{h,1}$ for adapting the mesh by refinement and coarsening, guided by the condition \eqref{eq:APosterioriLocalCondition}. The adaptivity algorithm is initialized with a mesh $\mathcal T_h$ built on $N = 30$ elements of equal size and proceeds by refinement and coarsening. Results, given in \cref{fig:ErrEst}, confirm the validity of our probabilistic error estimators. In particular, we remark that the local error estimators succeed in identifying the regions where the mesh has to be refined, thus getting a solution with an approximately equal distribution of the error over the domain. Both probabilistic estimators, moreover, succeed in bounding the global error until the tolerance is reached, with the estimator $\widetilde{\mathcal E}_{h,2}$ which appears to be more efficient than $\widetilde{\mathcal E}_{h,1}$.

\subsubsection{Two-Dimensional Case}\label{sec:NumExp_2D}

We now present two numerical experiments conducted in the two-dimensional case. In particular, for both experiments we only focus on the computation of $\widetilde{\mathcal E}_{h,2}$ in \cref{def:ProbErrEst}, since in view of \cref{rem:SuperMesh} this second estimator is computationally easier to implement than $\widetilde {\mathcal E}_{h,1}$ for $d > 1$. To account for errors on the boundary elements, we decide for these experiments to perturb all points, including those on the boundaries, following \cref{rem:BoundaryPoints}. In order for $\widetilde{\mathcal I}u_h$, and thus $\widetilde{\mathcal E}_{h, 2}$ to be well-defined, we reflect the perturbed boundary points symmetrically to the boundary $\partial D$ in case they are outside the domain. For both experiments, we implement the RM-FEM with a uniform distribution for the random perturbations, as described in \cref{ex:ExRandomPerturbation}. Moreover, we fix $p = 3$ and compute the Monte Carlo approximation \eqref{eq:ErrEstMC} on $N_{\mathrm{MC}} = 500$ realizations of the random mesh. For the adaptivity algorithm, we start from a coarse mesh and apply regular local refinements if the condition \eqref{eq:APosterioriLocalCondition} is not met by the local error estimator $\widetilde \eta_{K,2}$. In the two-dimensional case we do not apply coarsening to the mesh.

\begin{figure}[t]
	\centering
	\begin{tabular}{cc}
	\includegraphics[]{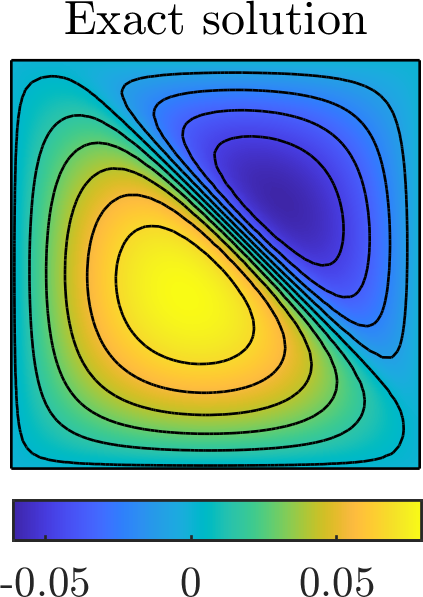} & \includegraphics[]{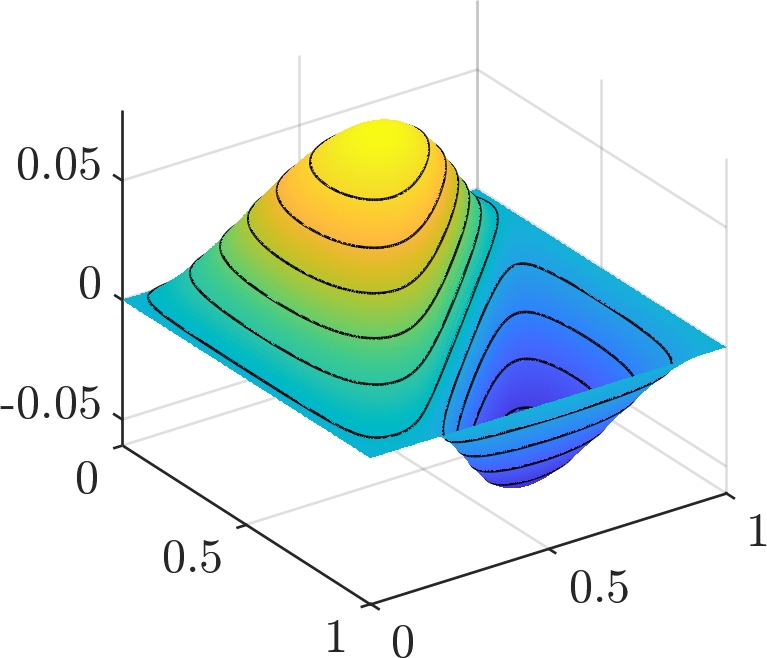}
	\end{tabular}
	\caption{Exact solution $u_1$ for the experiment of \cref{sec:NumExp_2D}. Both the contour and the three-dimensional view highlight the steep gradient that features $u_1$.}
	\label{fig:ExSol2D_ZZ}
\end{figure}
\begin{figure}[t]
	\centering
	\begin{tabular}{cc}
		\includegraphics[]{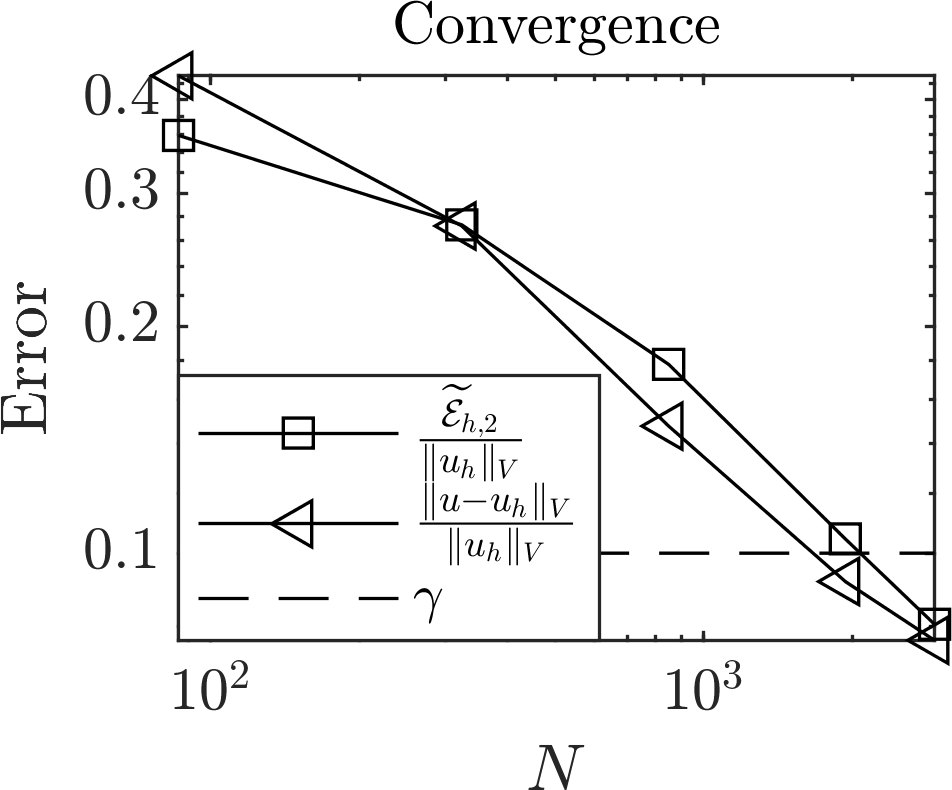} & \includegraphics{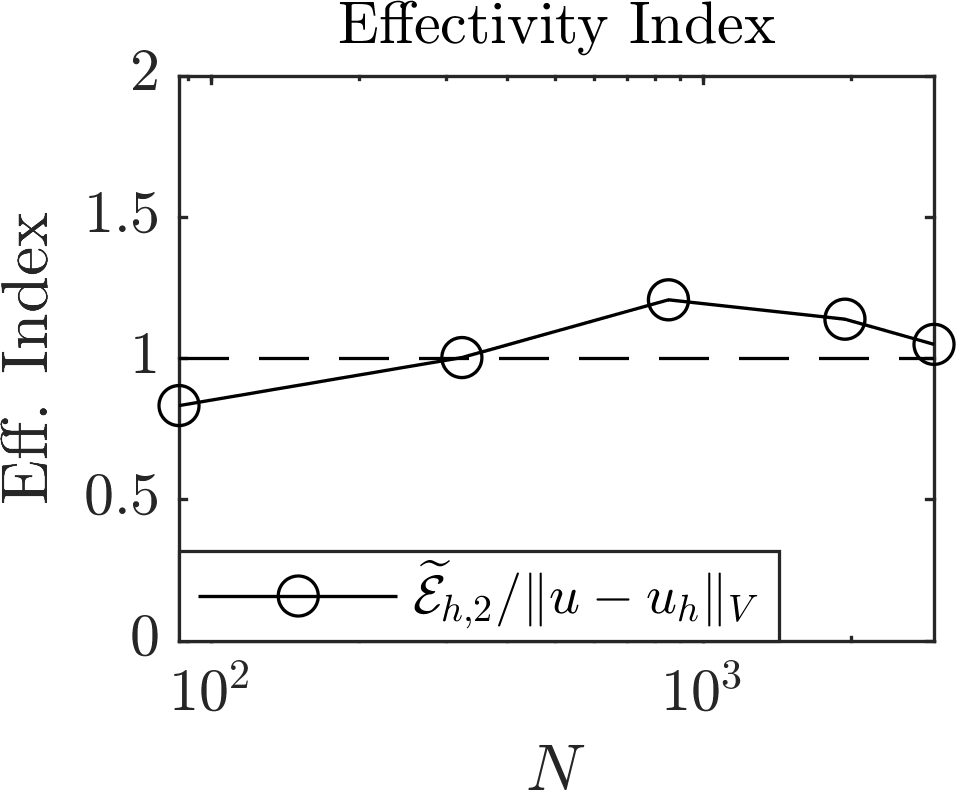}
	\end{tabular}
	\caption{Error convergence and effectivity index for the first experiment (function $u_1$) of \cref{sec:NumExp_2D}}
	\label{fig:Conv2D_ZZ}
\end{figure}

\begin{figure}[t]
	\centering
	\begin{tabular}{cccc}
		Iter. 2 & Iter. 3 & Iter. 4 & Iter. 5 \vspace{0.2cm} \\
		\includegraphics{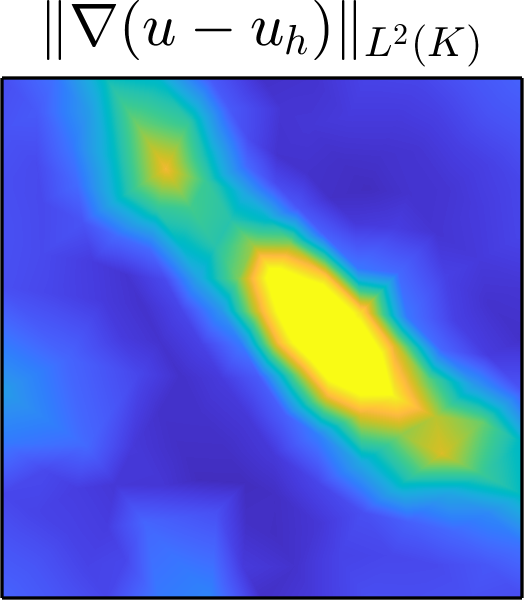} & \includegraphics{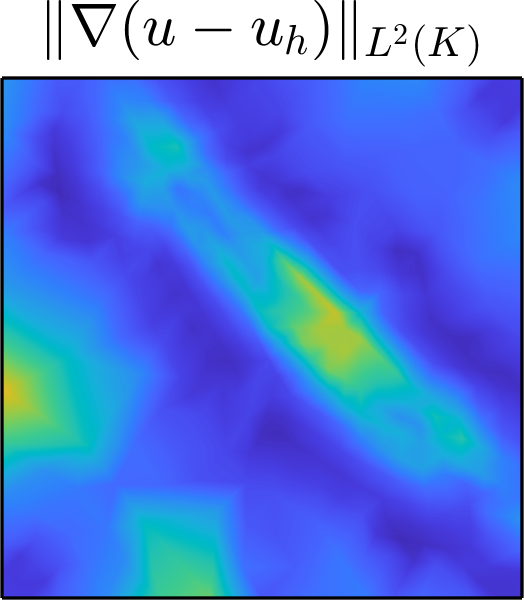} & \includegraphics{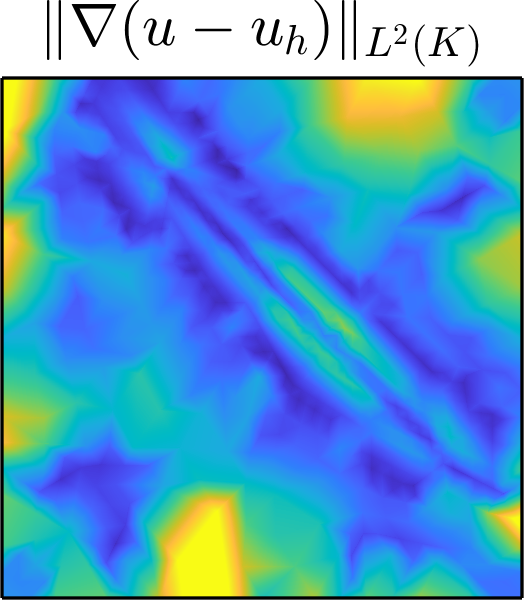} & \includegraphics{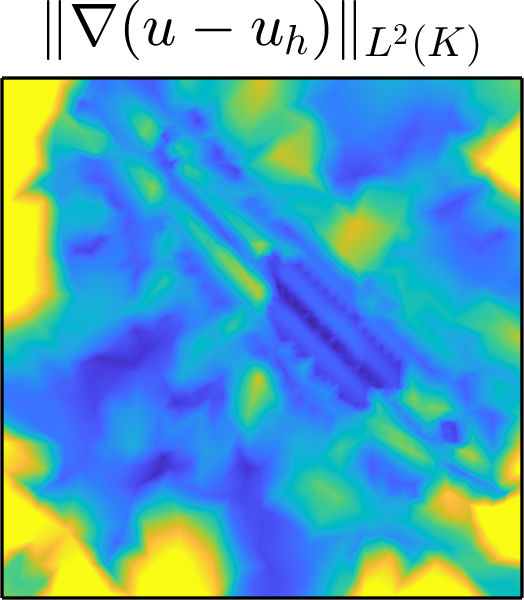} \vspace{0.2cm}\\
		\includegraphics{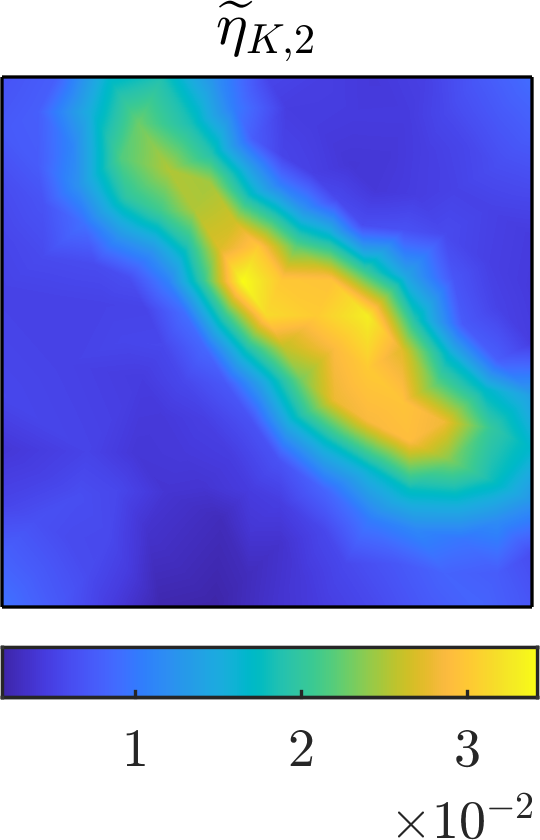} & \includegraphics{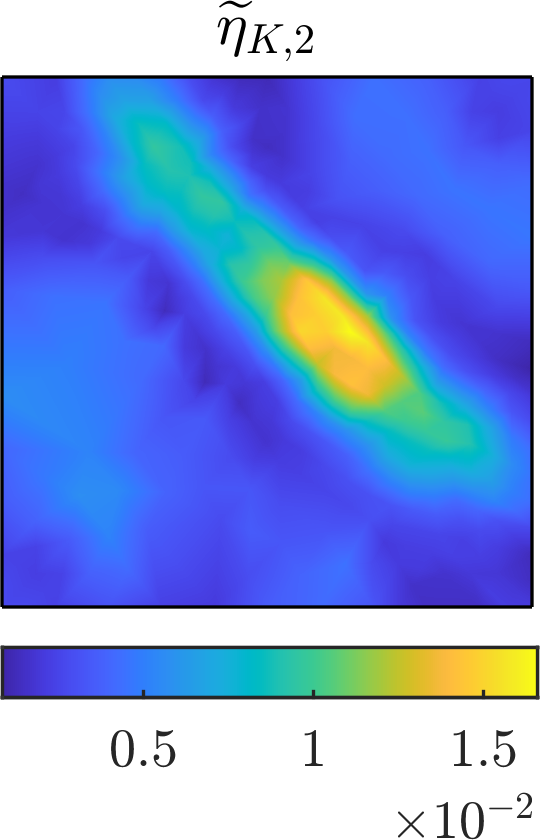} & \includegraphics{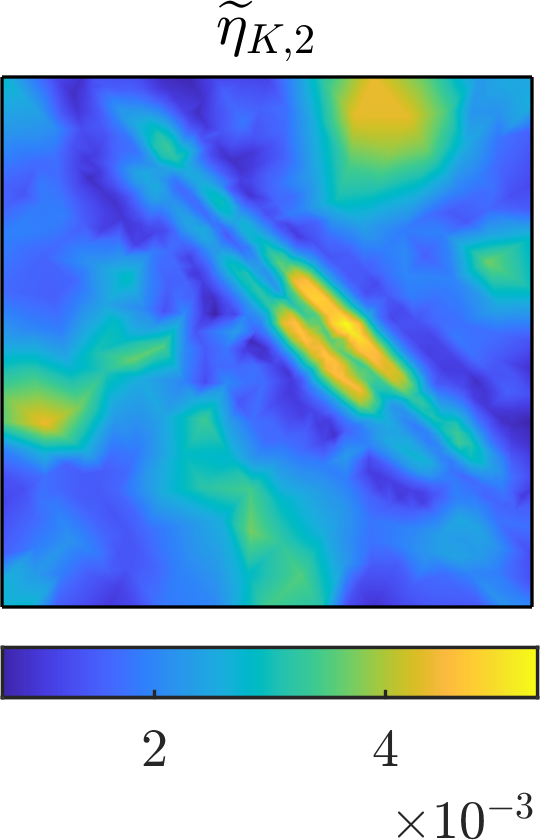} & \includegraphics{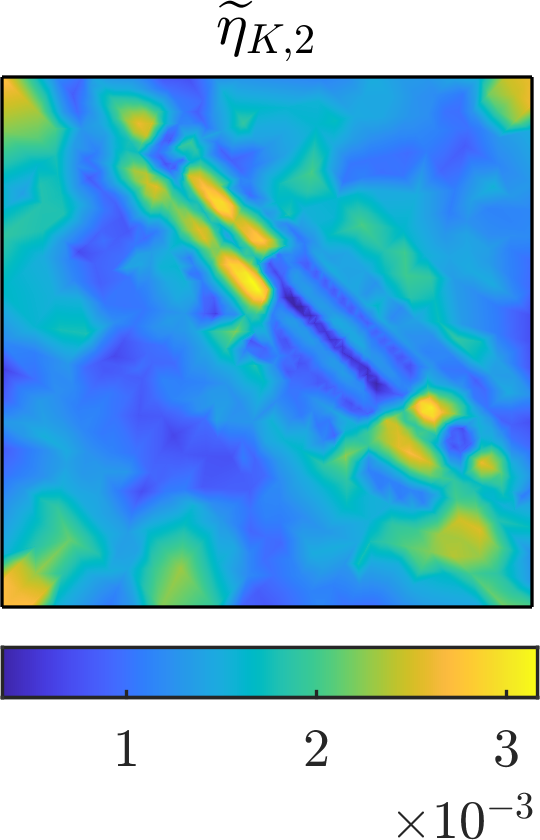}\vspace{0.2cm} \\ 
		\includegraphics{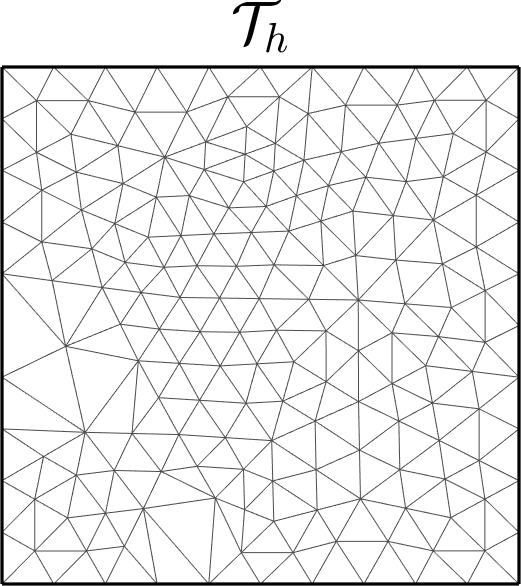} & \includegraphics{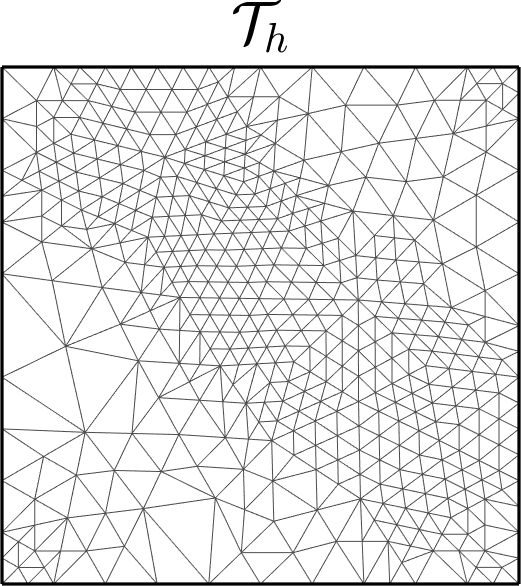} & \includegraphics{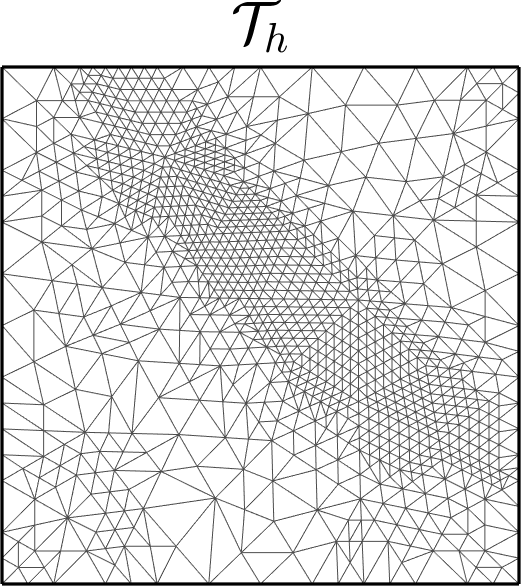} & \includegraphics{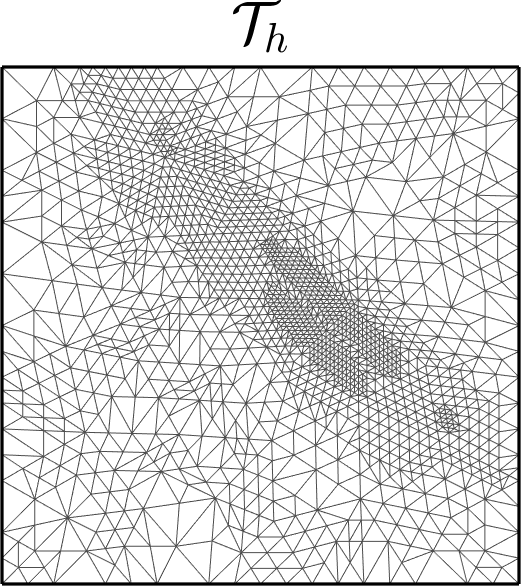}
	\end{tabular}
	\caption{Per row: True local error, local error estimator $\widetilde \eta_{K,2}$ and mesh $\mathcal T_h$ at each iteration of the adaptivity algorithm for the function $u_1$ of \cref{sec:NumExp_2D}. The color bar is shared by the first and the second rows.}
	\label{fig:Adapt2D_ZZ}
\end{figure}

We first consider $D = (0,1)^2$, the conductivity $\kappa = 1$, so that \eqref{eq:PDE} reduces to $-\Delta u = f$ with homogeneous Dirichlet boundary conditions. Moreover, we choose the right-hand side $f$ such that
\begin{equation}
	u_1(x,y) = -x (1 - x)y(1-y) \arctan\left( \beta\left(\frac{x + y}{\sqrt{2}} - \frac{4}{5}\right)\right),
\end{equation}
where $\beta > 0$. The solution has a steep transition around the line $\{y = 4\sqrt{2}/5 - x\}$, whose steepness is proportional to the parameter $\beta$. In \cref{fig:ExSol2D_ZZ}, we show the exact solution for $\beta = 20$, which we fix for this experiment. We initialize the adaptivity procedure with a mesh with maximum element size $h = 1/5$ and proceed with adaptation until a tolerance $\gamma = 0.1$. In \cref{fig:Conv2D_ZZ} we show the convergence of $\widetilde {\mathcal E}_{h,2}$ with respect to the convergence of the true error, as well as the the effictivity index for this experiment. We can see that the estimator indeed captures the error globally. In \cref{fig:Adapt2D_ZZ}, we show the behavior of the local contributions $\widetilde \eta_{K,2}$ with respect to the true error on each element, as well as the mesh adaptation. We can see that the error estimator succeeds in identifying the region where gradients are the steepest and proposes a mesh which appears adapted to this problem.  

We then consider the L-shaped domain with the re-entrant corner on the origin, i.e. $D = (-1, 1)^2 \setminus (-1, 0)^2$. We set $\kappa = 1$, $f = 0$ and fix a inhomogeneous Dirichlet boundary conditions $u = g$ on $\partial D$, with $g$ chosen such that the exact solution satisfies
\begin{equation}
	u_2(r, \theta) = r^{2/3}\sin\left(\frac{2}{3}\left(\theta + \frac{\pi}{2}\right)\right),
\end{equation}
where $(r,\theta) \in \R^+ \times (0, 2\pi]$ are the polar coordinates in $\R^2$. The exact solution of this problem is given in \cref{fig:ExSol2D_L}. Let us remark that the gradient of the exact solution is singular at the re-entrant corner, and we expect the mesh to be refined consequently at the singularity. For this experiment, we fix the tolerance $\gamma = 0.03$, and initialize the mesh to have a maximum element size of $h = 1/3$. Results, given in \cref{fig:Conv2D_L} and \cref{fig:Adapt2D_L}, show on the one hand that the estimator reproduces well the behavior of the global error during adaptation, and on the other hand that the mesh is progressively refined at the singularity as expected. 

\begin{figure}[t]
	\centering
	\includegraphics[]{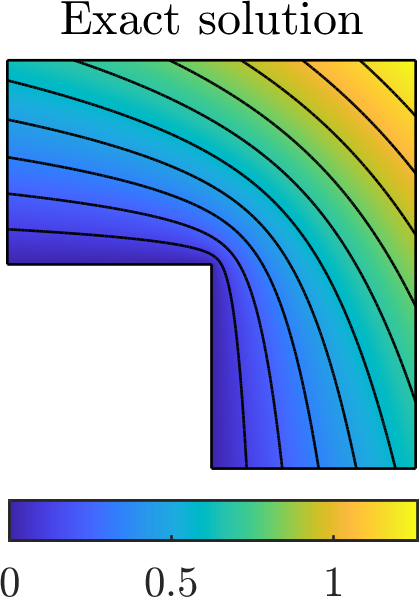}
	\caption{Exact solution $u_2$ for the experiment of \cref{sec:NumExp_2D}}
	\label{fig:ExSol2D_L}
\end{figure}
\begin{figure}[t]
	\centering
	\begin{tabular}{cc}
		\includegraphics[]{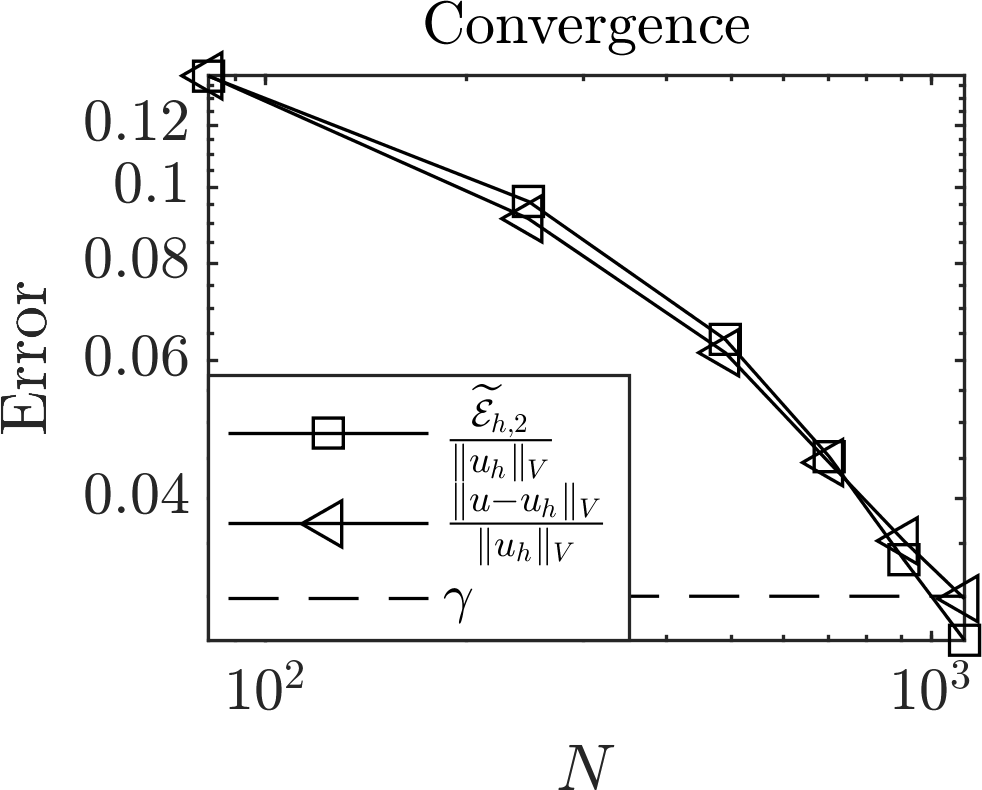} & \includegraphics{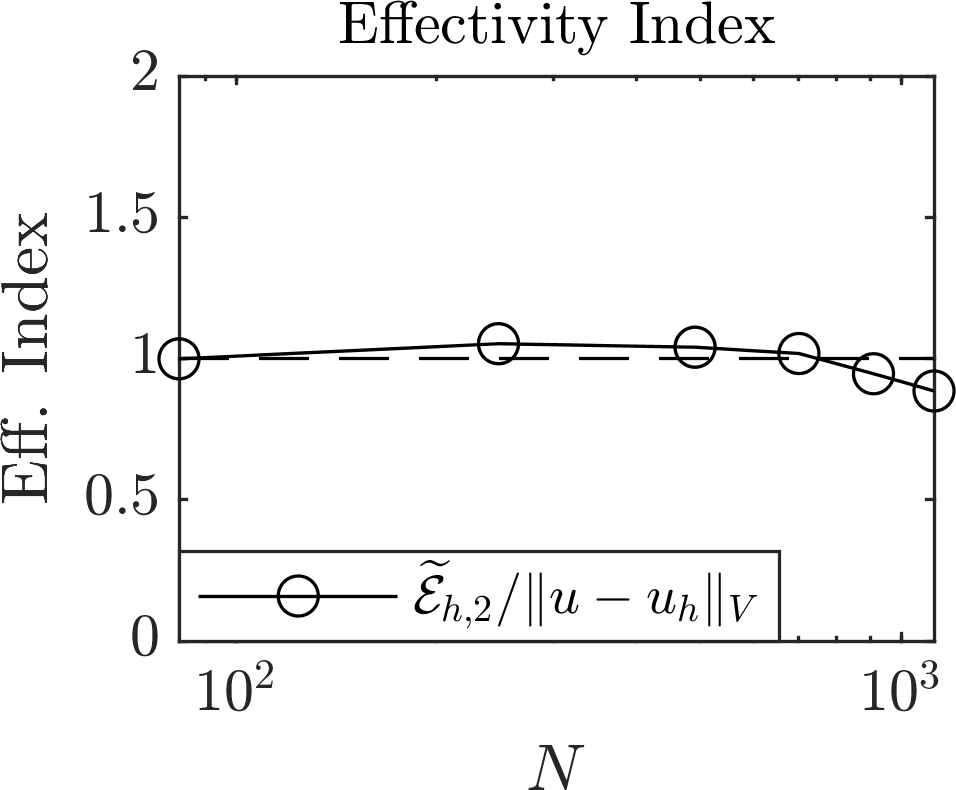}
	\end{tabular}
	\caption{Error convergence and effectivity index for the second experiment (function $u_2$) of \cref{sec:NumExp_2D}}
	\label{fig:Conv2D_L}
\end{figure}
\begin{figure}[t]
	\centering
	\begin{tabular}{cccc}
		Iter. 2 & Iter. 3 & Iter. 4 & Iter. 5 \vspace{0.2cm} \\
		\includegraphics{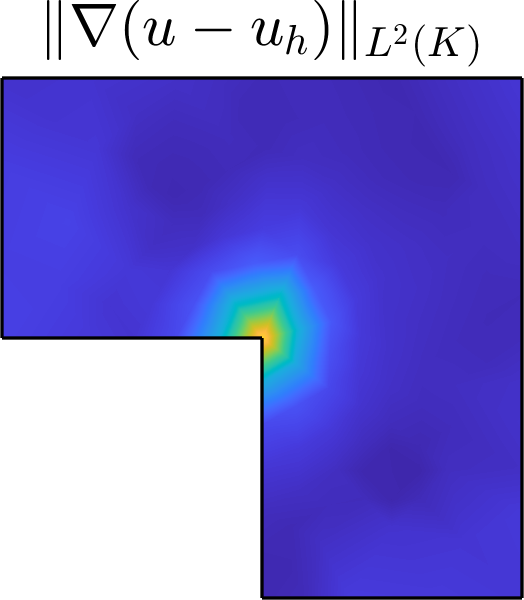} & \includegraphics{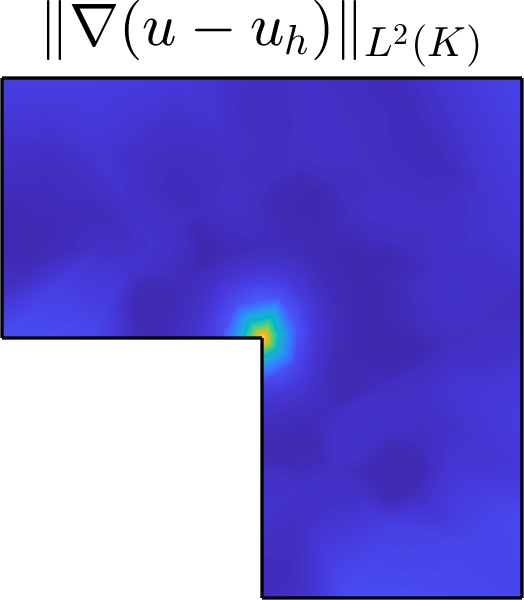} & \includegraphics{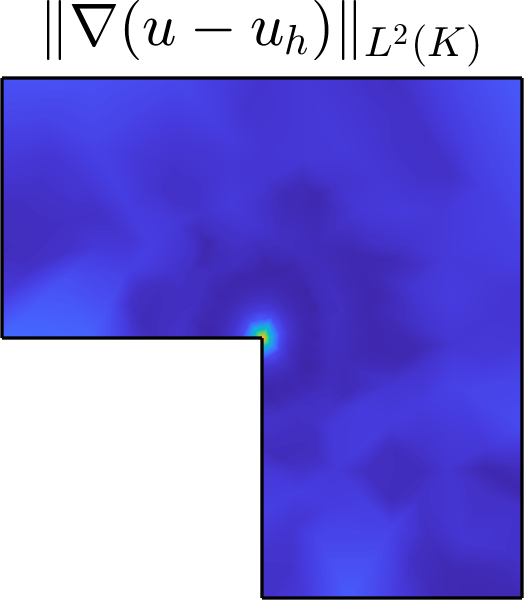} & \includegraphics{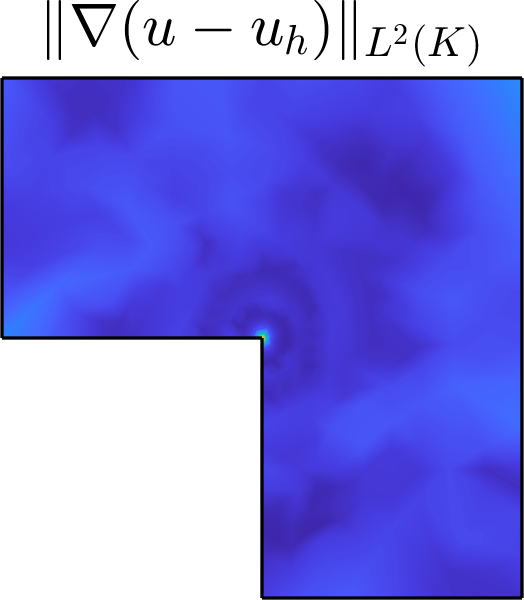} \vspace{0.2cm}\\
		\includegraphics{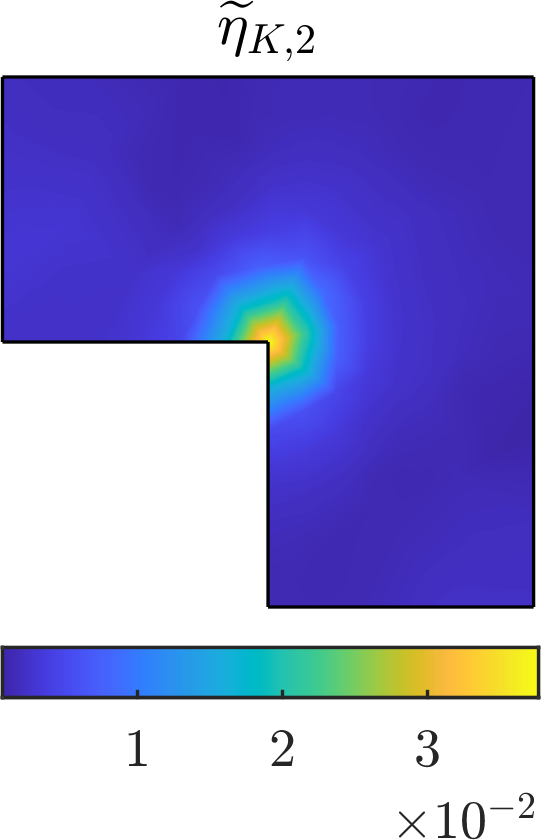} & \includegraphics{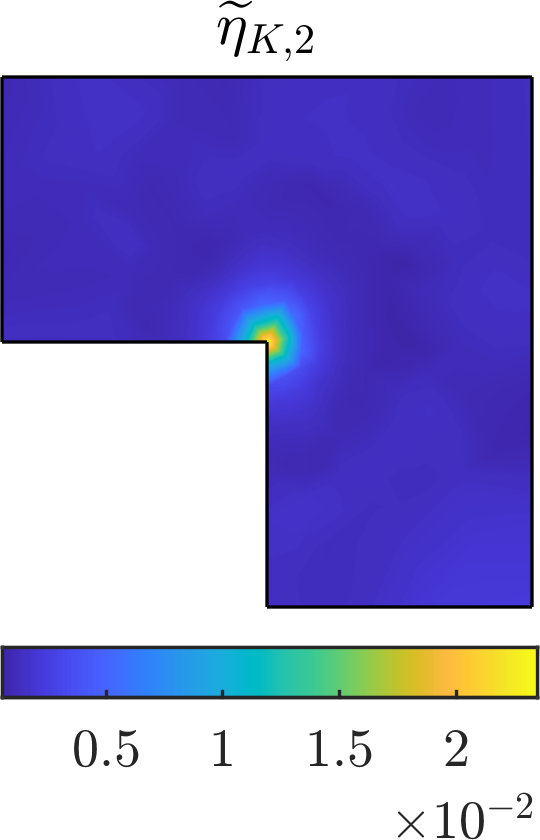} & \includegraphics{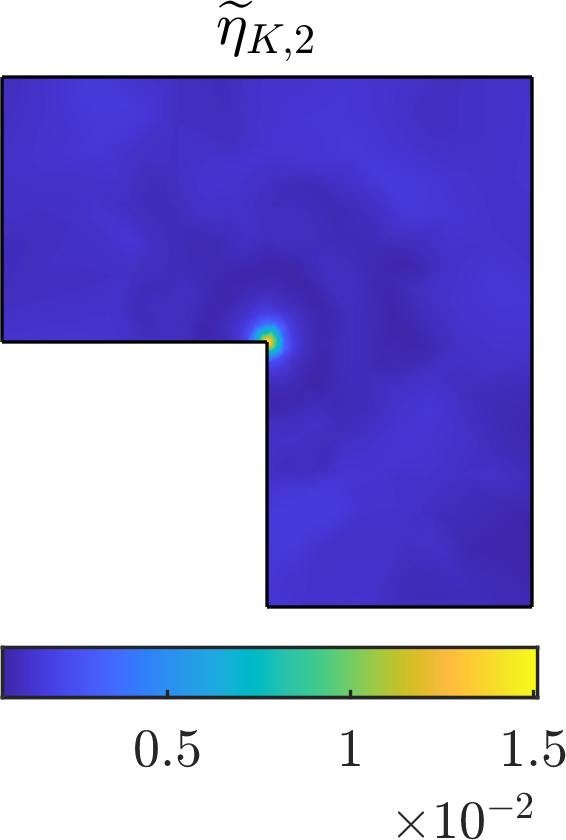} & \includegraphics{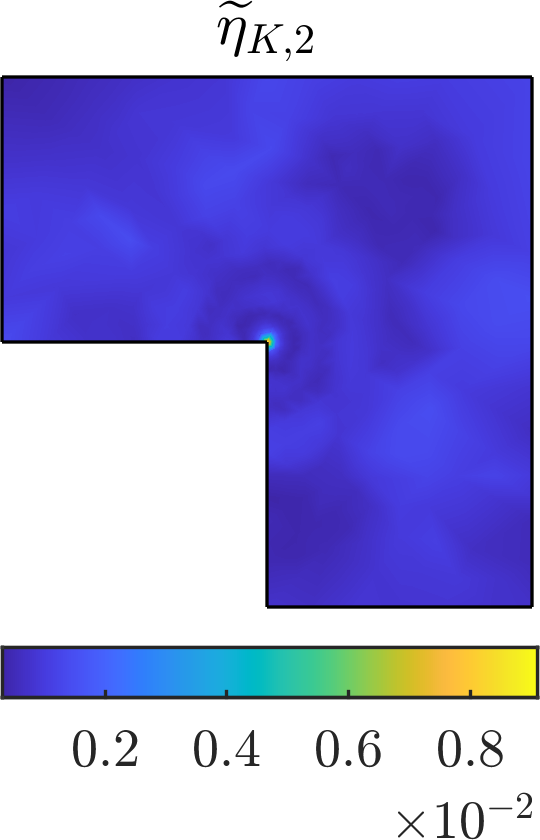}\vspace{0.2cm} \\ 
		\includegraphics{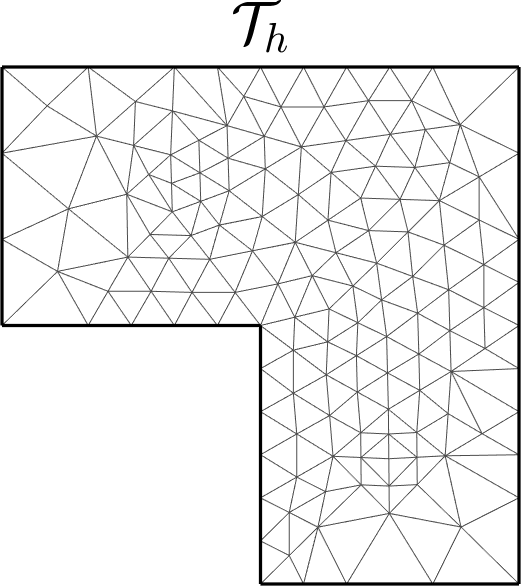} & \includegraphics{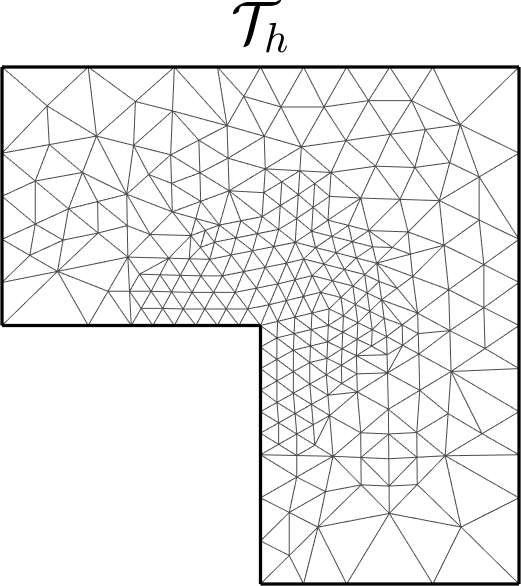} & \includegraphics{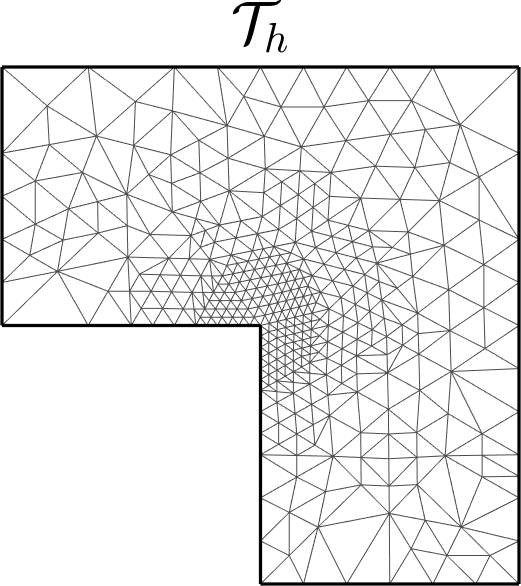} & \includegraphics{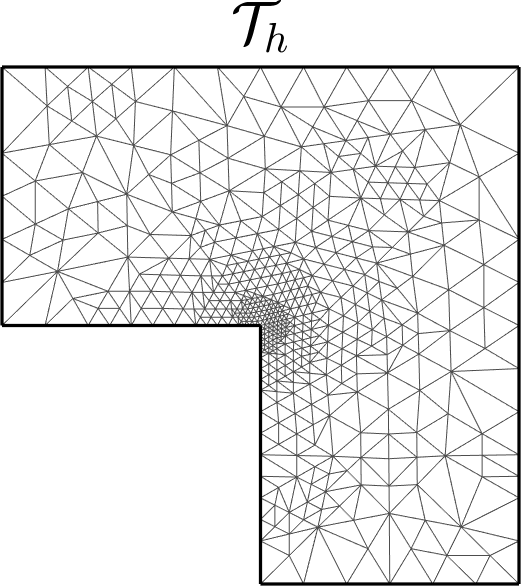}
	\end{tabular}
	\caption{Per row: True local error, local error estimator $\widetilde \eta_{K,2}$ and mesh $\mathcal T_h$ at each iteration of the adaptivity algorithm for the function $u_2$ of \cref{sec:NumExp_2D}. The color bar is shared by the first and the second rows.}
	\label{fig:Adapt2D_L}
\end{figure}

\section{The RM-FEM for Bayesian Inverse Problems}\label{sec:BIP}

Probabilistic numerical methods have been demonstrated to be particularly effective in the context of Bayesian inverse problems \cite{CGS17, AbG20, CCC16, LST18, COS17, OCA19, COS17b}. We consider the framework of \cite[Section 3.4]{DaS16} and introduce the parameterized PDE
\begin{equation}\label{eq:PDE_Inverse}
	\begin{alignedat}{2}
		-\nabla \cdot (\exp(\theta) \nabla u) &= f, \quad &&\text{in } D, \\
		u &= 0, &&\text{on } \partial D,
	\end{alignedat}
\end{equation}
where $D$ is an open bounded set of $\R^d$ and $\theta \colon D \to \R$ is a scalar function. In particular, we let $\theta$ be such that problem \eqref{eq:PDE_Inverse} is well-posed, i.e., $\kappa = \exp(\theta) \in L^\infty(D)$ and $\kappa \geq \underline{\kappa} > 0$, and we denote by $X$ the space of admissible values for $\theta$. We introduce the solution operator $\mathcal S \colon X \to V$ such that $\mathcal S \colon \theta \mapsto u$, and the observation operator $\mathcal O\colon V \to \R^m$, which maps the solution of the PDE to point evaluations inside the domain on points $x^* = x_1^*, \ldots, x_m^*$, i.e. $\mathcal O \colon u \mapsto y \defeq (u(x_1^*), \ldots, u(x_m^*))^\top$. Moreover, we denote by $\mathcal G = \mathcal O \circ \mathcal S$, $\mathcal G \colon X \to \R^m$, the so-called forward operator, which maps the parameter to the observations. We then have the Gaussian observation model
\begin{equation}\label{eq:ObsModel}
	y = \mathcal G(\theta) + \beta, \quad \beta \sim \mathcal N(0, \Sigma),
\end{equation}
where $\Sigma \in \R^{m\times m}$ is a non-singular covariance matrix on $\R^m$. Given an observation $y^* = (u(x_1^*), \ldots, u(x_m^*))^\top$ associated to an unknown value $\theta^* \in X$ and corrupted by observational noise $\beta \in \R^m$ the inverse problem can then be stated as:
\begin{equation}\label{eq:InverseProblem}
	\text{find } \theta^* \in X \text{ given observations } y^* = \mathcal G(\theta^*) + \beta.
\end{equation}
The randomness and the mismatch between the dimensionality of the unknown and of the observation make problem \eqref{eq:InverseProblem} ill-posed. Regularization can be achieved in the Bayesian framework (see e.g. \cite{Stu10,DaS16}) by introducing probability measures on the unknown, which summarize prior knowledge and the information provided by data. We briefly introduce the Bayesian paradigm in the remainder of this section. First, we restrict ourselves to the space $\mathcal H = \mathcal C^0(\overline D) \cap V$, which is a valid subspace of admissible values for $\theta$, i.e., $\mathcal H \subset X$. We then introduce a prior measure $\mu_0$ on $\mathcal H$, encoding all knowledge on the unknown $\theta$ before observations are obtained. In particular, we consider a Gaussian prior measure $\mu_0 = \mathcal N(0, \Gamma_0)$ on $\mathcal H$, where $\Gamma_0$ is a positive semi-definite covariance operator on $\mathcal H$, and such that $\mu_0(\mathcal H)= 1$, so that any sample from $\mu_0$ is in $\mathcal H$, a.s. We set the mean of the prior measure to zero without loss of generality. A broader class of prior measures could be employed, such as Besov or heavy-tailed measures \cite{DaS16, Sul17}, but we restrict ourselves to the Gaussian case for simplicity. Finally, we can obtain a measure $\mu^y$ on $X$ encoding all the knowledge on $\theta$ given the prior and the observations. We call this measure the posterior, and formally compute with Bayes' formula its Radon--Nykodim derivative with respect to the prior as
\begin{equation}\label{eq:Posterior}
	\frac{\d \mu^y}{\d \mu_0}(\theta) = \frac1{Z^y}\exp(-\Phi(\theta; y)),
\end{equation}
where $Z^y$ is the normalization constant
\begin{equation}\label{eq:NormConstant}
	Z^y = \int_{\mathcal H} \exp(-\Phi(\theta;y)) \dd\mu_0(\theta),
\end{equation}
and where for any $y \in \R^m$ the potential $\Phi(\cdot; y) \colon X \to \R$ is given due to the Gaussian assumption on the noise $\beta \sim \mathcal N(0, \Sigma)$ by
\begin{equation}
	\Phi(\theta; y) = \frac12\norm{\Sigma^{-1/2}\left(\mathcal G(\theta) - y\right)}^2_2.
\end{equation}
For economy of notation, in the following we drop the dependence of $\mu^y$ and $Z^y$ on the data, and simply denote these quantities $\mu$ and $Z$. In order for \eqref{eq:Posterior} to be well-defined, the posterior measure needs to be absolutely continuous with respect to the prior. This is ensured under some conditions on the forward map $\mathcal G$, in particular, Lipschitz continuity and some controlled growth condition. Precisely, $\mathcal G$ can be shown to satisfy \cite[Assumption 2.7]{Stu10}. We then choose the covariance $\Gamma_0 = -\Delta^{-\alpha}$ with $\alpha > d/2$ and where we equip the Laplacian with homogeneous boundary conditions. Fractional powers of the Laplacian should be understood as per \cite[Section 2]{Stu10}. With this choice, we have that $\mu_0(\mathcal H) = 1$. Together with the properties of $\mathcal G$ we can conclude by \cite[Theorem 3.4]{Stu10} that the posterior is indeed given by \eqref{eq:Posterior} and that the Bayesian inverse problem is well-posed, meaning that the dependence of the posterior on the observations is absolutely continuous.

In practice to evaluate the solution operator $\mathcal S$ we recur to the FEM and consider for an $h > 0$ the forward operator $\mathcal S_h \colon X \to V_h$ which maps $\mathcal S_h \colon \theta \mapsto u_h$, where $u_h$ is defined in \eqref{eq:FESolution}. Moreover, we denote by $\mathcal G_h = \mathcal O \circ \mathcal S_h$ the resulting approximate forward operator. Maintaining the same observation model and the same prior on $\mathcal H$ for the parameter as above, we consider the approximate posterior $\mu_h$ on $\mathcal H$ whose Radon--Nikodym derivative with respect to the prior is given by
\begin{equation}
	\frac{\d \mu_h}{\d \mu_0}(\theta) = \frac1{Z_h}\exp(-\Phi_h(\theta; y)),
\end{equation}
where the potential $\Phi_h$ is given by
\begin{equation}\label{eq:NumPotential}
	\Phi_h(\theta; y) = \frac12\norm{\Sigma^{-1/2}\left(\mathcal G_h(\theta) - y\right)}^2_2, 
\end{equation}
and where the normalization constant $Z_h$ is defined equivalently to \eqref{eq:NormConstant}. A natural question arising from this setting is whether the approximate posterior $\mu_h$ converges to the true posterior $\mu$ in the limit $h \to 0$. This is indeed true, as it holds $\Hell(\mu, \mu_h) \to 0$ for $h \to 0$, where the Hellinger distance $\Hell(\cdot, \cdot)$ is defined as
\begin{equation}
	\Hell(\mu, \mu_h) \defeq \sqrt{\frac12 \int_{\mathcal H} \left(\sqrt{\frac{\d \mu}{\d \mu_0}} - \sqrt{\frac{\d \mu_h}{\d \mu_0}}\right)^2 \dd \mu_0}.
\end{equation}
For a proof of this result, see \cite[Theorem 4.6]{Stu10}, where the statement above is made more precise by the introduction of convergence rates.

It has been demonstrated heuristically that the approximate posterior measure $\mu_h$ can be overly confident on the parameter if $h > 0$ is a finite value and if the observation model is precise, i.e., when the covariance $\Sigma$ of the observational noise is small \cite{CGS17, AbG20, CCC16, LST18, COS17, OCA19, COS17b}. In particular, this is accentuated when $h$ is big relatively to the forward problem, or in other words when we employ a poor approximation of the forward map $\mathcal G$. It is therefore useful in applications to have a cheap surrogate which can be evaluated quickly, without renouncing to a complete uncertainty quantification of the solution to the inverse problem. Probabilistic numerical methods can be employed for this purpose. Let $\widetilde{\mathcal S}_h \colon X \to \widetilde V_h$ be the solution operator mapping $\widetilde{\mathcal S}_h \colon \theta \mapsto \widetilde u_h$, where $\widetilde u_h$ is the RM-FEM solution of \cref{def:ProbSol}. In particular, in this context it is advisable to fix $p = 1$ in \eqref{eq:PerturbedPoints}, so that the random deviations of the probabilistic numerical solution are of the same order of magnitude than the error itself by \cref{thm:APriori}. Coherently to the definitions above, we define the random forward map $\widetilde{\mathcal G}_h \colon X \to \R^m$ as $\widetilde{\mathcal G}_h = \mathcal O \circ \widetilde{\mathcal S}_h$ and the approximate random posterior measure $\widetilde \mu_h$ on $X$ as
\begin{equation}
	\frac{\d \widetilde \mu_h}{\d \mu_0}(\theta) = \frac{\exp(-\widetilde\Phi_h(\theta; y))}{\widetilde Z_h} ,
\end{equation}
where the potential $\widetilde \Phi_h$ and the normalization constant $\widetilde Z_h$ are defined as above. Let us remark that the posterior $\widetilde \mu_h$ is a random measure, as it depends on the random variable $\alpha \colon \Omega \to \R^{N_I}$ governing the random perturbations of the mesh. To be more precise, the posterior $\widetilde \mu_h$ is a random variable $\widetilde \mu_h \colon \Omega \to \mathcal M(\mathcal H)$, where $\mathcal M(\mathcal H)$ is the space of probability measures over the space $\mathcal H$. Employing the tools of \cite{LST18} and due to \cref{thm:APriori}, it is possible to prove convergence results for $\widetilde \mu_h$ towards the true posterior $\mu$ for $h \to 0$ in the Hellinger metrics.

\begin{remark} There exist other approaches to factor the effects of discretization into Bayesian inverse problems. In particular, numerical error can be treated as modelling discrepancies. Under the assumption that the error is independent of the observational noise and of the inferred parameter $\theta$, a viable alternative to probabilistic methods is employing the techniques introduced in \cite{CDS18, CES14}, and further applied and analysed e.g. in \cite{AbD20, AGZ20}. We argue that while assuming numerical errors to be independent of the observational noise is reasonable, their independence from the inferred parameter itself is not, at least for the inverse problem \eqref{eq:InverseProblem}.  
\end{remark}

\subsection{Implementation Details}\label{sec:BIP_Practice}

We now detail how one can solve the inverse problems above in practice. Given a smooth functional $\Psi$ on $X$, we are interested in approximating the quantities $\E_{\mu_h}[\Psi(\cdot)]$ and $\E_{\widetilde{\mu}_h}[\Psi(\cdot)]$, where $\E_{\mu_h}$ denotes expectation with respect to the measure $\mu_h$ (respectively $\widetilde \mu_h$). To be more precise, in the probabilistic case we are interested to the quantity $\E[\E_{\widetilde \mu_h}[\Psi(\cdot)]]$, where the outer expectation is taken with respect to the random perturbations intrinsic to the RM-FEM, and where for a sufficiently smooth functional $\Psi$ the expectations can be exchanged by means of Fubini's theorem. Due to the high-dimensionality of the problem, Monte Carlo techniques are a natural choice. Let $N_{\mathrm{MC}}$ be a positive integer and let us assume that we have samples $\{\theta^{(i)}\}_{i=0}^{N_{\mathrm{MC}}}\sim \mu_h$, not necessarily independent. Then, defining
\begin{equation}\label{eq:PostMean}
	\widehat \Psi_{\mu_h} \defeq \frac1{N_{\mathrm{MC}}} \sum_{i=1}^{N_{\mathrm{MC}}}\Psi(\theta^{(i)}).
\end{equation}
we have $\widehat \Psi_{\mu_h} \approx \E_{\mu_h}[\Psi(\cdot)]$. For the probabilistic case, let $\widetilde N_{\mathrm{MC}}^{\mathrm{out}}$ and $\widetilde N_{\mathrm{MC}}^{\mathrm{in}}$ be positive integers, let $\{\widetilde \mu_h^{(j)}\}_{j=1}^{\widetilde N_{\mathrm{MC}}^{\mathrm{out}}}$ be a sequence of realizations of the measure $\widetilde \mu_h$, obtained with a corresponding series of random perturbations of the mesh, and let $\{\widetilde \theta^{(j,i)}\}_{i=1}^{\widetilde N_{\mathrm{MC}}^{\mathrm{in}}} \sim \widetilde \mu_h^{(j)}$. Then, we define
\begin{equation}\label{eq:PostMeanProb}
	\widehat \Psi_{\widetilde \mu_h} \defeq \frac1{\widetilde N_{\mathrm{MC}}^{\mathrm{out}}\widetilde N_{\mathrm{MC}}^{\mathrm{in}}} \sum_{i=1}^{\widetilde N_{\mathrm{MC}}^{\mathrm{out}}}\sum_{j=1}^{\widetilde N_{\mathrm{MC}}^{\mathrm{in}}}\Psi(\theta^{(j,i)}),
\end{equation}
and we have $\widehat \Psi_{\widetilde \mu_h} \approx \E[\E_{\widetilde{\mu}_h}[\Psi(\cdot)]]$. Still, the random variable $\theta$ is infinite-dimensional, and we need to define a finite-dimensional approximation in order to obtain a practical procedure to generate the above samples and thus solve the inverse problem. We recur to the Karhunen--Loeve expansion (KL). Denoting by $\{(\lambda_i, \phi_i)\}_{i\geq 1}$ the ordered eigenvalues/eigenfunctions of the prior covariance $\Gamma_0$, a function $\theta \sim \mu_0$ is given by the convergent sum
\begin{equation}
	\theta = \sum_{i \geq 1} \sqrt{\lambda_i} \phi_i \xi_i, 
\end{equation}
where $\{\xi_i\}_{i\geq 1} \iid \mathcal N(0, 1)$. We then let $ N_{\mathrm{KL}}$ be a positive integer and truncate the sum above as
\begin{equation}\label{eq:KarunhenLoeve}
	\theta = \sum_{i = 1}^{ N_{\mathrm{KL}}} \sqrt{\lambda_i} \phi_i \xi_i, 
\end{equation}
thus obtaining a function $\theta \in \mathcal H$ which is approximately sampled from $\mu_0$. Due to the super-quadratic decay of the eigenvalues of $\Gamma_0 = -\Delta^{-\alpha}$ for $\alpha > d/2$, disregarding the tail of the sum causes a negligible error in case $ N_{\mathrm{KL}}$ is chosen appropriately large. Our inversion problem is therefore shifted to computing the posterior distribution on a finite-dimensional parameter, comprising the coefficients of the expansion \eqref{eq:KarunhenLoeve}. We define the mapping $\mathcal K \colon \R^{ N_{\mathrm{KL}}} \to \mathcal H$, $\mathcal K \colon \xi \mapsto \theta$ by \eqref{eq:KarunhenLoeve}. The prior measure on $\xi$ is $\mu_0 = \mathcal N(0, I)$, with $I$ being the identity matrix of dimension $ N_{\mathrm{KL}}\times  N_{\mathrm{KL}}$, and we denote by $\pi_0$ the density of $\mu_0$ with respect to the Lebesgue measure. The density $\pi_h$ of the posterior on the coefficients $\xi$ given the observations is then 
\begin{equation}\label{eq:PosteriorDensity}
	\pi_h(\xi) = \frac1{Z_h} \pi_0(\xi) \, \exp(-\Phi_h(\mathcal K(\xi); y)), 
\end{equation}
where $\Phi_h$ is defined in \eqref{eq:NumPotential} and $Z_h$ is the normalization constant
\begin{equation}
	Z_h = \int_{\R^{ N_{\mathrm{KL}}}} \exp(-\Phi_h(\mathcal K(\xi); y)) \, \pi_0(\xi) \dd \xi.
\end{equation}
The same procedure can be applied seamlessly to the probabilistic case, thus obtaining a random posterior density $\widetilde \pi_h$ over $\R^{ N_{\mathrm{KL}}}$ for the coefficient $\xi$. 

The last detail missing is how to produce samples in order to obtain the approximations \eqref{eq:PostMean} and \eqref{eq:PostMeanProb}. Being the normalizations constant unknown, we employ Markov chain Monte Carlo techniques (MCMC) (see e.g. \cite[Chapter 3]{KaS05} or \cite[Chapter 6]{KTB13}), which proceed by generating an ergodic Markov chain whose invariant density is the desired posterior. Successive states of the aforementioned Markov chain then serve as samples from the posterior density. We choose to employ the Metropolis--Hastings (MH) algorithm, which we here briefly detail. The Markov chain is built employing a symmetric proposal $q \colon \R^{ N_{\mathrm{KL}}} \times \R^{ N_{\mathrm{KL}}} \to \R$ satisfying $q(x,y) = q(y, x)$ for all $x,y \in \R^{N_{\mathrm{KL}}}$ and such that for any fixed $x \in \R^{N_{\mathrm{KL}}}$ the function $q(x, \cdot)$ is a probability density, and with an acceptance-rejection strategy. In particular, given an initial guess $\xi^{(1)}$, the algorithm proceeds for $i = 2, \ldots, N_{\mathrm{MC}}$ as
\begin{enumerate}
	\item Sample $\bar \xi^{(i)} \sim q(\xi^{(i-1)}, \cdot)$;
	\item Set $\xi^{(i)} = \bar \xi^{(i)}$ with probability $\alpha$, and $\xi^{(i)} = \xi^{(i-1)}$ with probability $1-\alpha$, where
	\begin{equation}
		\alpha = \min\left\{\frac{\pi_h(\bar \xi^{(i)})}{\pi_h(\xi^{(i-1)})}, 1\right\}.
	\end{equation} 
\end{enumerate} 
Let us remark that the normalization constant $Z_h$ does not need to be known to run the algorithm, since we only compute ratios of posterior densities. Moreover, we note that the proposal distribution is the only tunable element of the MH algorithm. The easiest choice, at least for implementation, would be to fix $q(x, \cdot) = \mathcal N(x, \sigma^2 I)$ for some user-prescribed variance $\sigma^2$. Unfortunately, the quality of the resulting Markov chain is not robust with respect to $\sigma$. In particular, if $\sigma$ is too small, the probability to accept is too large and the Markov chain fails to effectively explore the posterior. At the other end of the spectrum, if $\sigma$ is too large the probability of accepting a new sample reduces drastically, and the Markov chain presents a sticky behaviour. We therefore decide to employ the robust adaptive Metropolis algorithm (RAM) (see \cite{Vih12} for details), in which the proposal is $q(x, \cdot) = \mathcal N(x, \Sigma_q)$, where the covariance $\Sigma_q$ is adapted on the fly to obtain a user-specified final acceptance ratio, i.e. the ratio between the accepted and the total number samples, which should be roughly $25\%$ (see e.g. \cite{Vih12}). Another viable option for the implementation of MCMC could have been the preconditioned Crank--Nicolson MCMC (pCN-MCMC) of \cite{CRS13, HSV14}, which is tailored for high-dimensional inverse problems. 

For the probabilistic case, we perform a run of the MH, implemented with RAM proposal, with $\widetilde N_{\mathrm{MC}}^{\mathrm{in}}$ iterations for each one of the $\widetilde N_{\mathrm{MC}}^{\mathrm{out}}$ realizations of the random mesh, thus obtaining the approximation \eqref{eq:PostMeanProb}. 

\begin{remark} The total number of samples is given by in the probabilistic case by $\widetilde N_{\mathrm{MC}}^{\mathrm{in}} \cdot \widetilde N_{\mathrm{MC}}^{\mathrm{out}}$. One could argue that, for a fair comparison between the probabilistic and the deterministic case in terms of computational cost, one should choose $N_{\mathrm{MC}} \approx \widetilde N_{\mathrm{MC}}^{\mathrm{in}} \cdot \widetilde N_{\mathrm{MC}}^{\mathrm{out}}$. In fact, since the ``outer'' Monte Carlo simulation can be performed in parallel, the correct scaling is $N_{\mathrm{MC}} \approx \widetilde N_{\mathrm{MC}}^{\mathrm{in}}$. Moreover, due to \cref{rem:CompCost}, the number of random meshes does not need to be chosen excessively large. \end{remark}

\subsection{Numerical Experiments}\label{sec:NumExp_BIP}

In this section we present numerical experiments highlighting the beneficial effects of adopting the probabilistic framework of RM-FEM in the context of Bayesian inverse problems.

\subsubsection{One-Dimensional Case}\label{sec:NumExpBIP_1D}

\begin{figure}[t]
	\centering
	\begin{tikzpicture}
		\draw (0,0) -- (8.7, 0) -- (8.7, 0.5) -- (0, 0.5) -- (0, 0);
		\draw[] (0.1, 0.25) -- (0.6, 0.25) node[right,color=black] {\small Truth $\kappa^*$};
		\draw[style=dashed] (2.3, 0.25) -- (2.8, 0.25) node[right,color=black] {\small Mean $\E_\mu[\kappa]$};
		\draw[fill=gray!50,draw=none] (5.1, 0.2) -- (5.6, 0.2) -- (5.6, 0.3) -- (5.1, 0.3) -- (5.1, 0.2); 
		\node[right] at (5.6, 0.25){\small Confidence Interval} ;
	\end{tikzpicture}
	
	\vspace{0.5cm}
	\begin{tabular}{ccc}
		\includegraphics[]{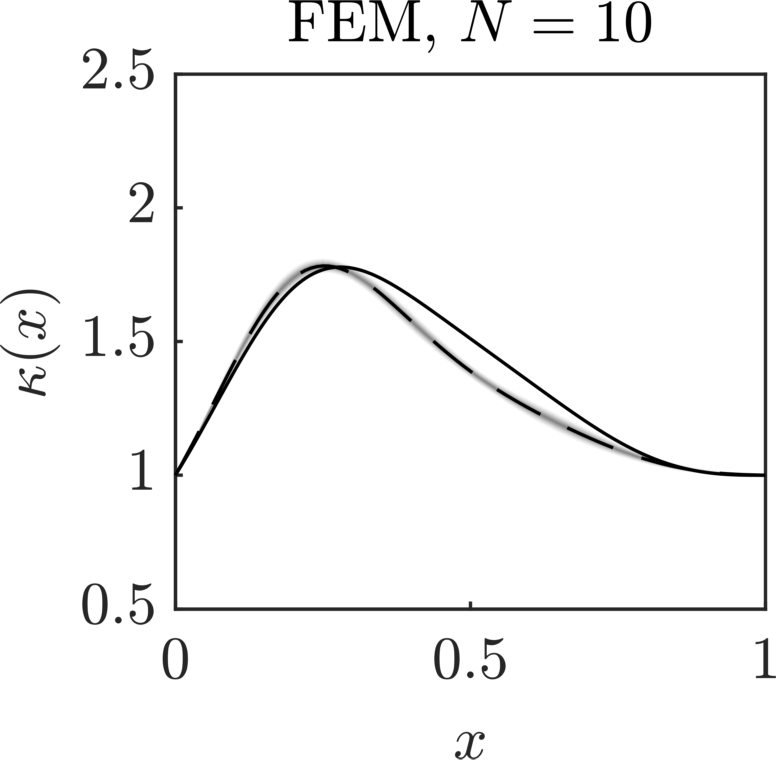} & \includegraphics[]{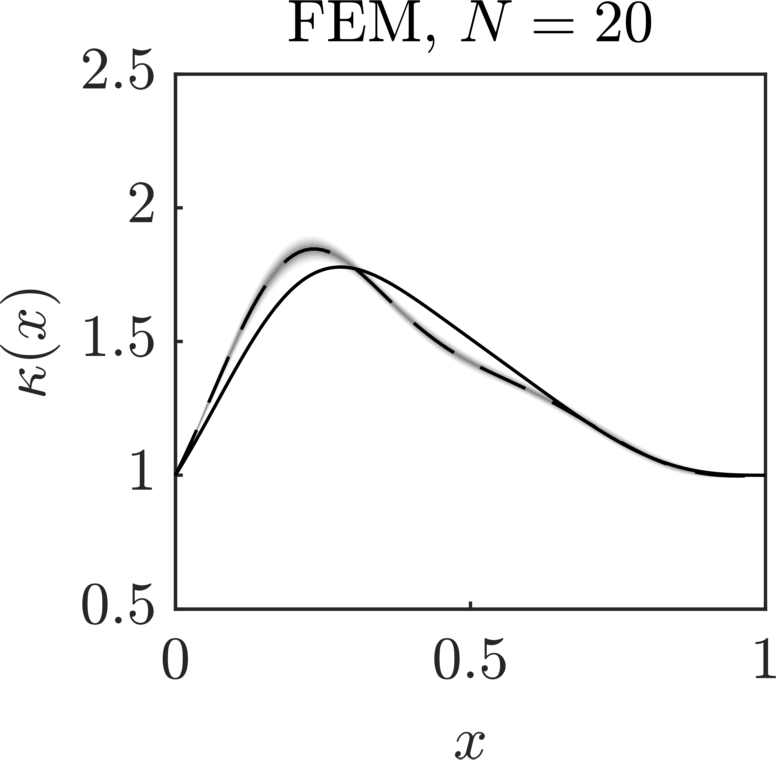} & \includegraphics[]{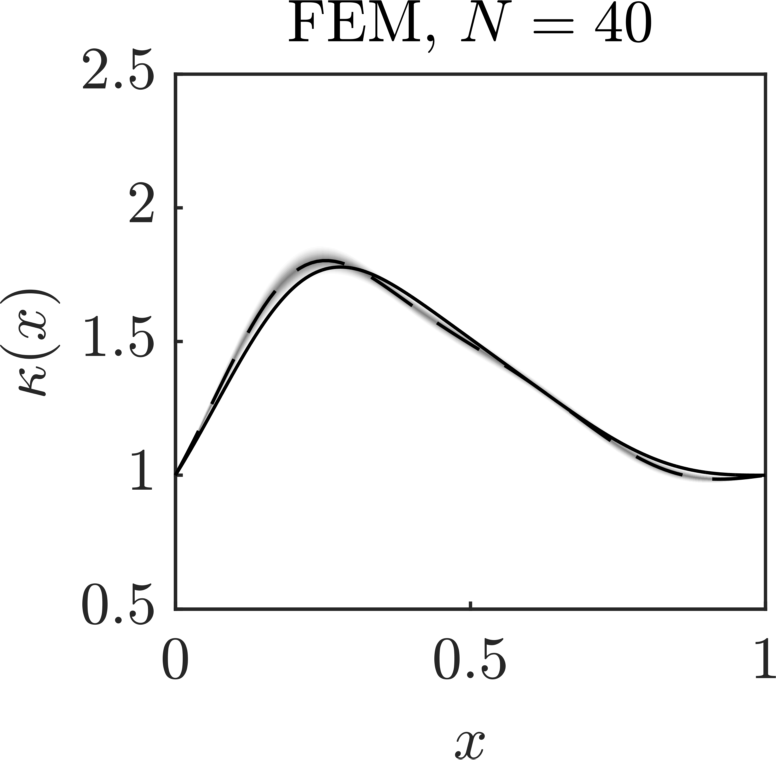} \vspace{0.1cm} \\
		\includegraphics[]{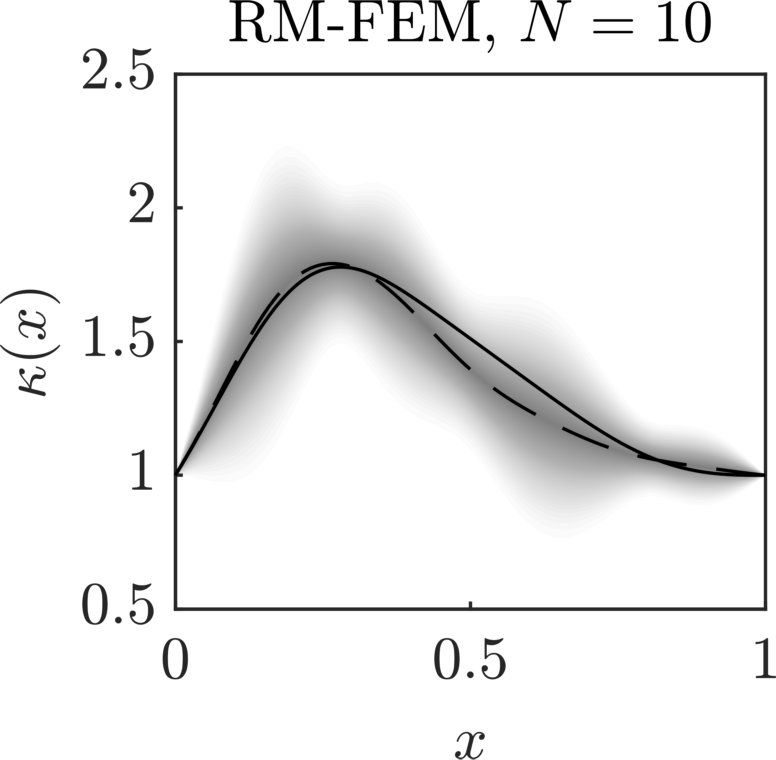} & \includegraphics[]{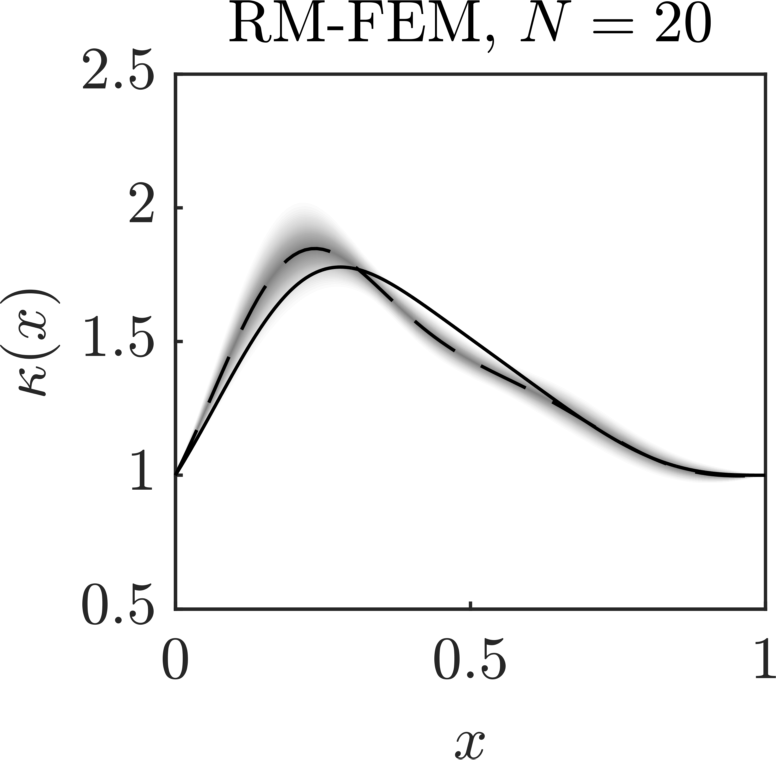} & \includegraphics[]{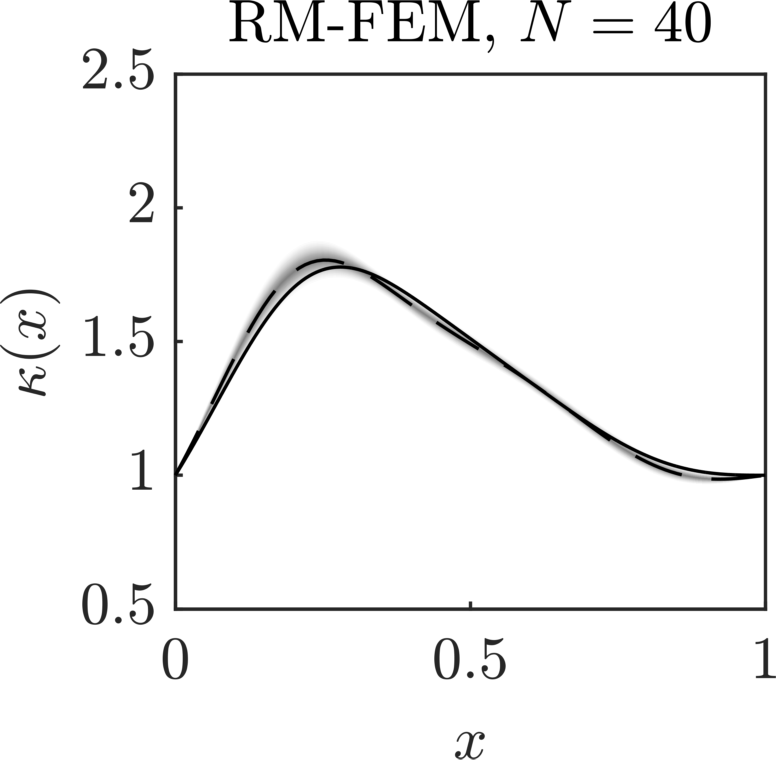}
	\end{tabular}
	\caption{Numerical results for $\kappa_1^*$ in \cref{sec:NumExpBIP_1D}. In all plots, the solid line represents the true conductivity, the dashed line is the posterior mean, and the shaded grey area is a confidence interval. In the first row, results are obtained by approximating the forward map with FEM, and in the second with the RM-FEM.}
	\label{fig:InverseProblem_Inf_Smooth}
\end{figure}

\begin{figure}[th]
	\centering
	\begin{tikzpicture}
		\draw (0,0) -- (8.7, 0) -- (8.7, 0.5) -- (0, 0.5) -- (0, 0);
		\draw[] (0.1, 0.25) -- (0.6, 0.25) node[right,color=black] {\small Truth $\kappa^*$};
		\draw[style=dashed] (2.3, 0.25) -- (2.8, 0.25) node[right,color=black] {\small Mean $\E_\mu[\kappa]$};
		\draw[fill=gray!50,draw=none] (5.1, 0.2) -- (5.6, 0.2) -- (5.6, 0.3) -- (5.1, 0.3) -- (5.1, 0.2); 
		\node[right] at (5.6, 0.25){\small Confidence Interval} ;
	\end{tikzpicture}
	
	\vspace{0.5cm}
	\begin{tabular}{ccc}
		\includegraphics[]{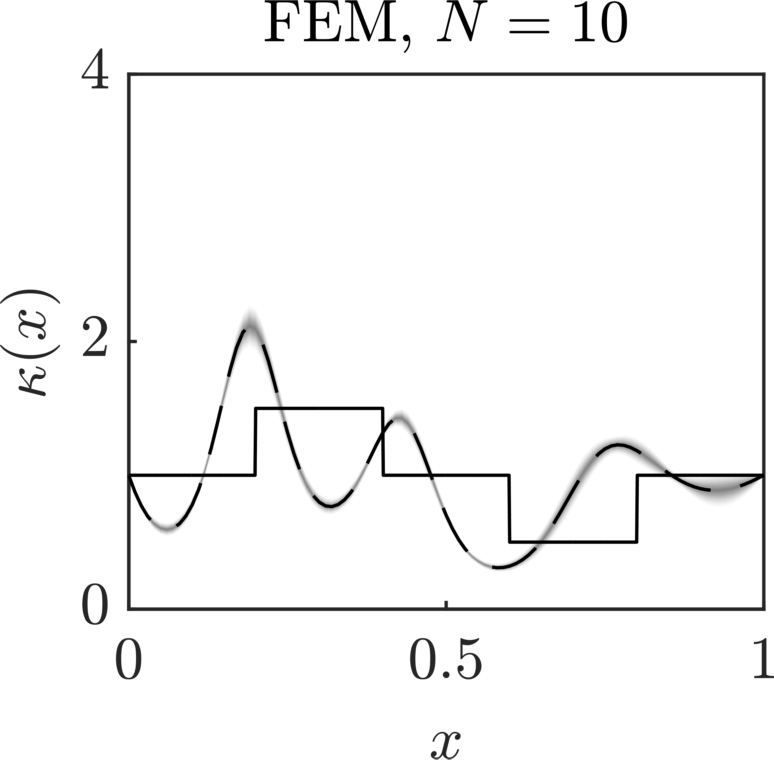} & \includegraphics[]{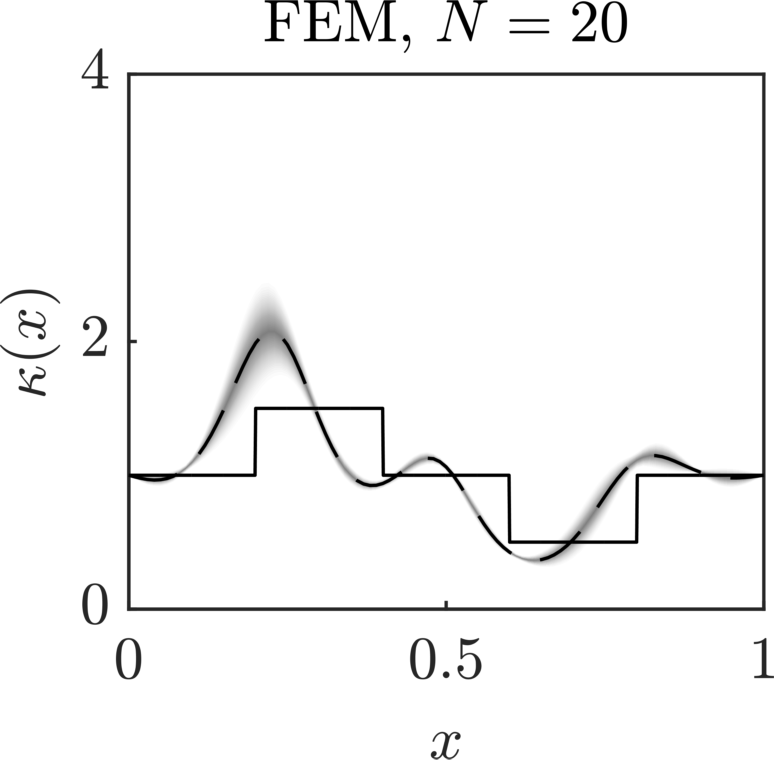} & \includegraphics[]{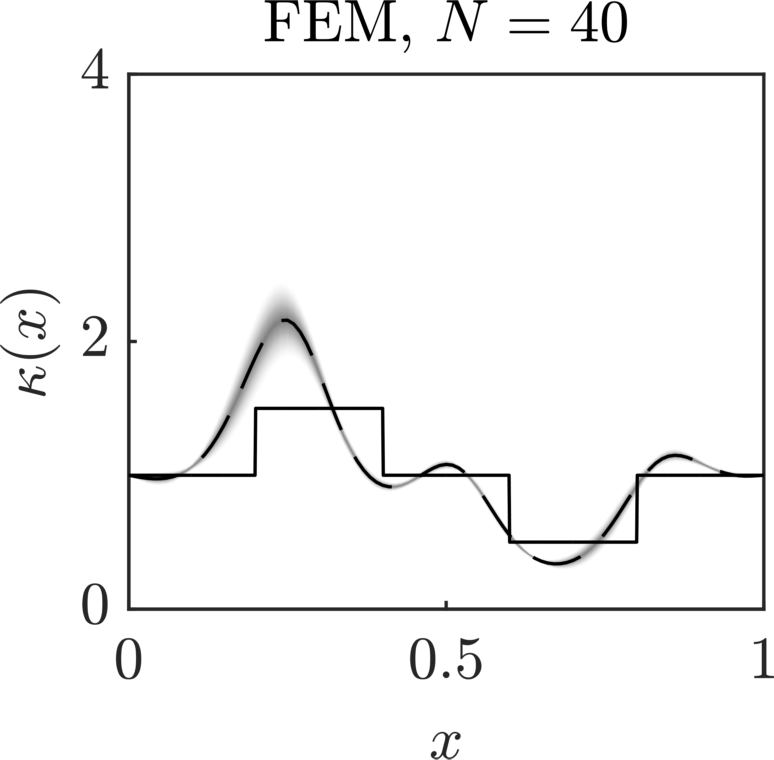} \vspace{0.1cm} \\
		\includegraphics[]{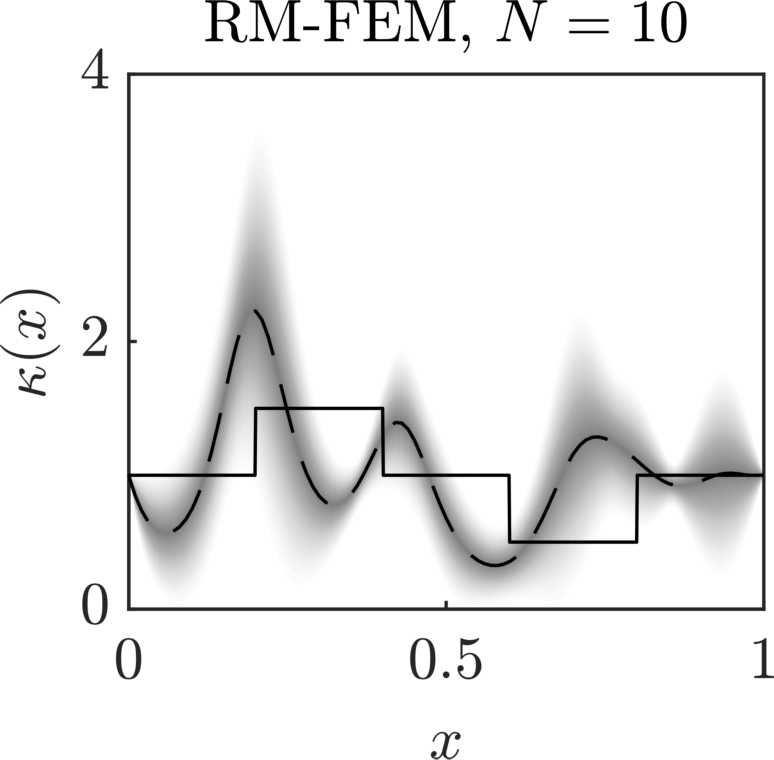} & \includegraphics[]{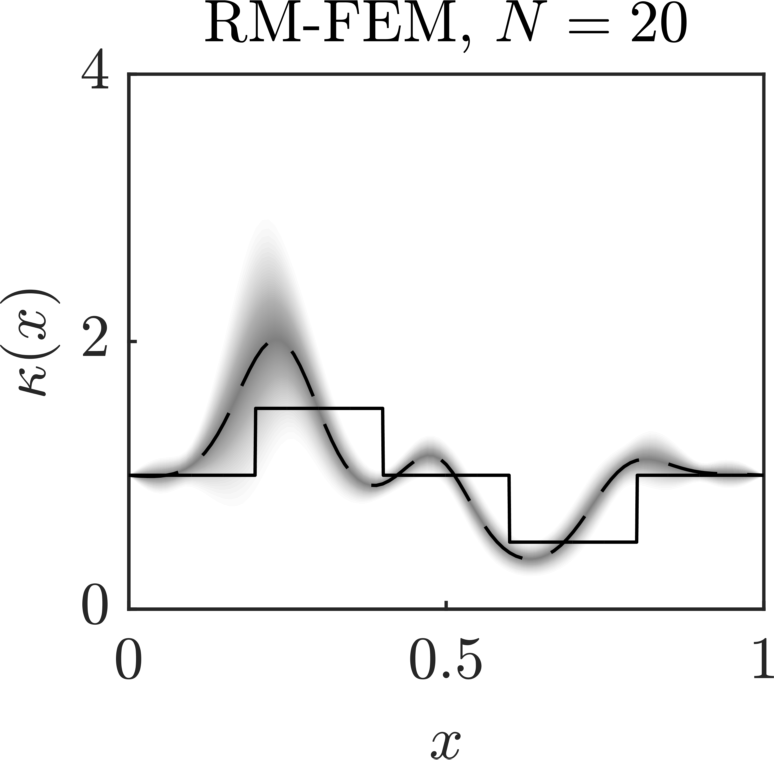} & \includegraphics[]{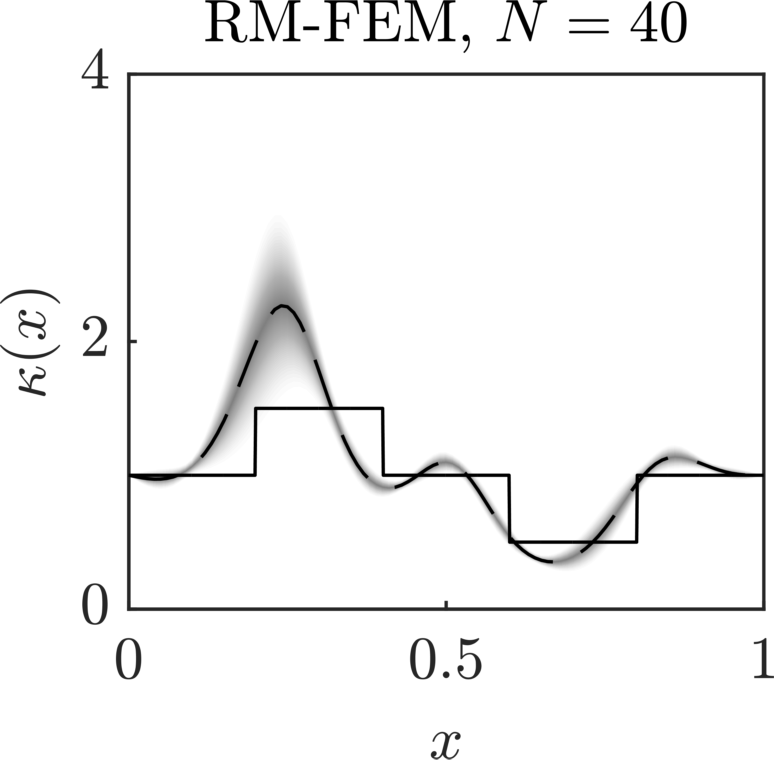}
	\end{tabular}
	\caption{Numerical results for $\kappa_2^*$ in \cref{sec:NumExpBIP_1D} In all plots, the solid line represents the true conductivity, the dashed line is the posterior mean, and the shaded grey area is a confidence interval. In the first row, results are obtained by approximating the forward map with FEM, and in the second with the RM-FEM.}
	\label{fig:InverseProblem_Inf_Disc}
\end{figure}

We first consider $D = (0,1)$ and solve the inverse problem presented above for two different true diffusion fields $\kappa^*$. In both cases, we consider the prior on $\mathcal H$ to be given by $\mathcal N(0, \Gamma_0)$, with $\Gamma_0^{-1} = -\d^2/\d x^2$ with homogeneous boundary conditions, so that the Bayesian inverse problem is well-posed. First, we consider $\kappa_1^* = \exp(\theta_1^*)$, where the log-conductivity $\theta_1^* \in \mathcal H$ is given by
\begin{equation}
	\theta_1^*(x) = \sum_{j = 1}^{4} \xi_j \sqrt{\lambda_j} \phi_j(x),
\end{equation}
with $\xi_1 = \xi_2 = 1$, $\xi_3 = \xi_4 = 1/4$, and where $\{(\lambda_i, \phi_i)\}_{i=1}^4$ are the first four ordered eigenpairs of $\Gamma_0$. Second, we consider $\theta_2^* \in X \cap \mathcal H^C$, so that the true conductivity does not belong to the domain in which we solve the inverse problem, but it is still admissible for \eqref{eq:PDE_Inverse} to be well-posed. In particular, we consider the discontinuous conductivity
\begin{equation}
	\kappa_2^*(x) = \left\{
	\begin{alignedat}{2}
		&1.5, &&\quad 0.2 < x < 0.6, \\
		&0.5, &&\quad 0.6 < x < 0.8,  \\
		&1,   &&\quad \text{otherwise},
	\end{alignedat}
	\right.
\end{equation}
and infer $\theta_2^* = \log(\kappa_2^*)$. For both problems, we choose the right-hand side in \eqref{eq:PDE_Inverse} as $f(x) = \sin(2\pi x)$. Synthetic observations are obtained as point evaluations of a reference solution on points $x_i^* = i/10$, for $i = 1, \ldots, 9$, corrupted by Gaussian noise $\mathcal N(0, 10^{-8} I)$. The forward map is approximated with FEM and RM-FEM. The mesh $\mathcal T_h$ for the FEM is equally spaced, and we vary the number of elements $N = \{10, 20, 40\}$. For the RM-FEM, we consider $p = 1$ in \eqref{eq:PerturbedPoints} as per \cref{thm:APriori} and implement the random perturbations with an uniform distribution as in \cref{ex:ExRandomPerturbation}. 

We sample with the MH algorithm from the posterior distributions $\mu_h$ and $\widetilde \mu_h$, with $N_{\mathrm{MC}} = 2 \cdot 10^5$ for $\mu_h$ and with $\widetilde N_{\mathrm{MC}}^{\mathrm{out}} = 50$ and $\widetilde N_{\mathrm{MC}}^{\mathrm{in}} = 2 \cdot 10^5$ for $\widetilde \mu_h$. Knowing for the first conductivity $\kappa^*_1$ that the true conductivity is fully determined by four coefficients, we fix the truncation index $N_{\mathrm{KL}} = 4$ in the Karhunen--Loève expansion \eqref{eq:KarunhenLoeve}. For the second conductivity $\kappa^*_2$, we fix $N_{\mathrm{KL}} = 9$. We then approximate the mean and pointwise standard deviation with \eqref{eq:PostMean} and \eqref{eq:PostMeanProb} for the deterministic and probabilistic posteriors, respectively. Moreover, we arbitrarily fix a pointwise confidence interval at twice the standard deviation away from the mean. Numerical results are given in \cref{fig:InverseProblem_Inf_Smooth} and \cref{fig:InverseProblem_Inf_Disc}. Results highlight that for a coarse approximation, specifically for $N = 10$, the posterior distribution $\mu_h$ is overly confident on the result. Indeed, the posterior mean fails to capture precisely the true conductivity in both the continuous and discontinuous case, and the confidence interval is extremely sharply concentrated around the mean. Conversely, the distribution $\widetilde \mu_h$ based on the probabilistic forward model accounts better for the uncertainty due to numerical discretization. Increasing the number of elements $N$, the mean computed under $\mu_h$ and $\widetilde \mu_h$ tends to approximate better the true conductivity field. In particular, for $N = 40$ the posteriors $\mu_h$ and $\widetilde \mu_h$ are already practically undistinguishable and are close to the true field. Moreover, let us remark that while the width of the confidence interval seems independent of $N$ for $\mu_h$, it shrinks coherently to the discretization for $\widetilde \mu_h$. Finally, we note that for $\kappa^*_2$ even for larger values of $N$ the posterior $\widetilde \mu_h$ seems to capture with its uncertainty local errors in the solution of the inverse problem. Indeed, the posterior mean is particularly off the true field on the left side of the domain, where the confidence interval is wider with respect to areas where the solution is more accurate. 

\subsubsection{Two-Dimensional Case}\label{sec:NumExpBIP_2D}

\begin{figure}[t]
	\centering
	\begin{tabular}{cccc}
		\includegraphics[]{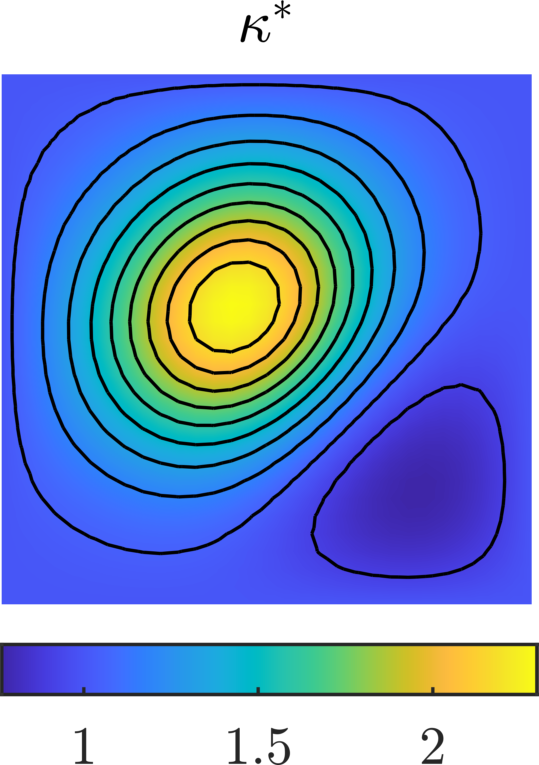} & \includegraphics[]{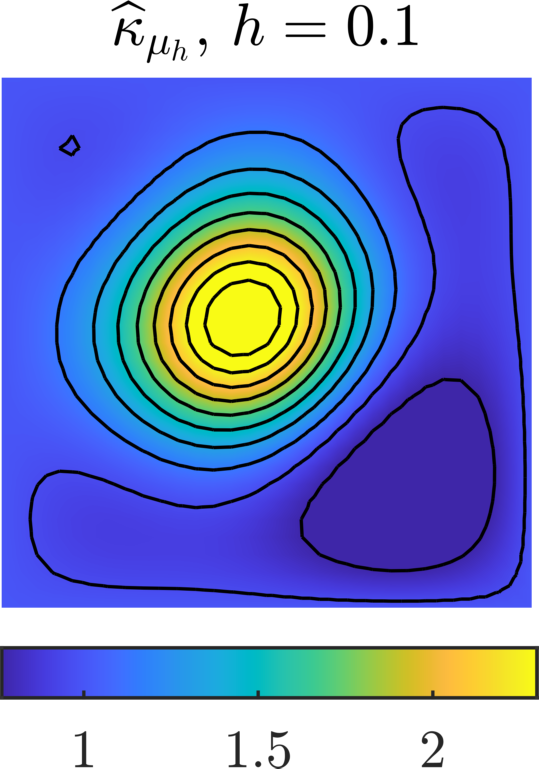} & \includegraphics[]{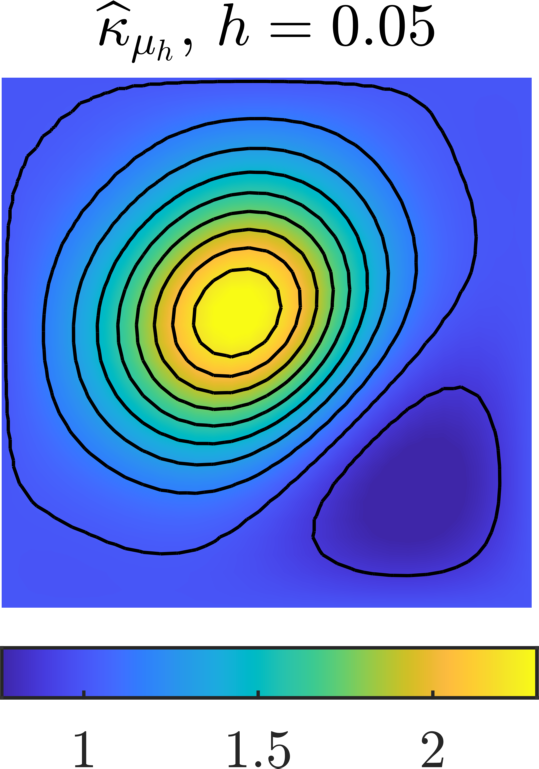} & \includegraphics[]{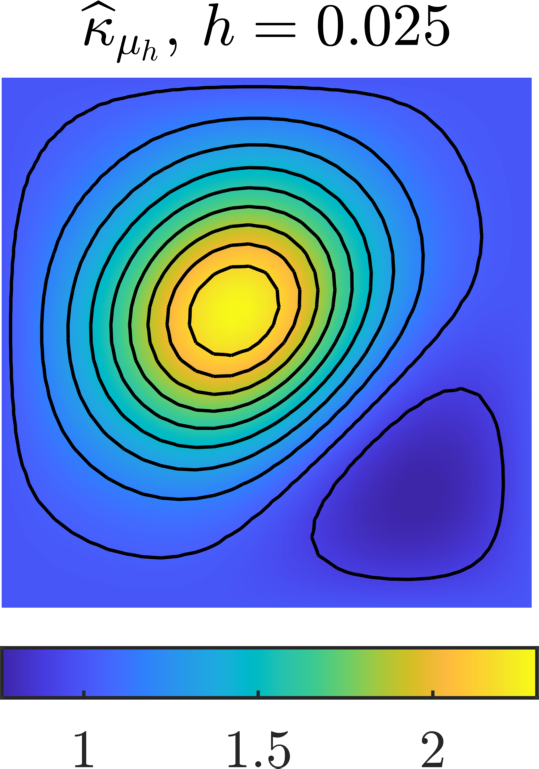}
	\end{tabular}
	
	\vspace{0.5cm}
	\begin{tabular}{cc}	
		\includegraphics[]{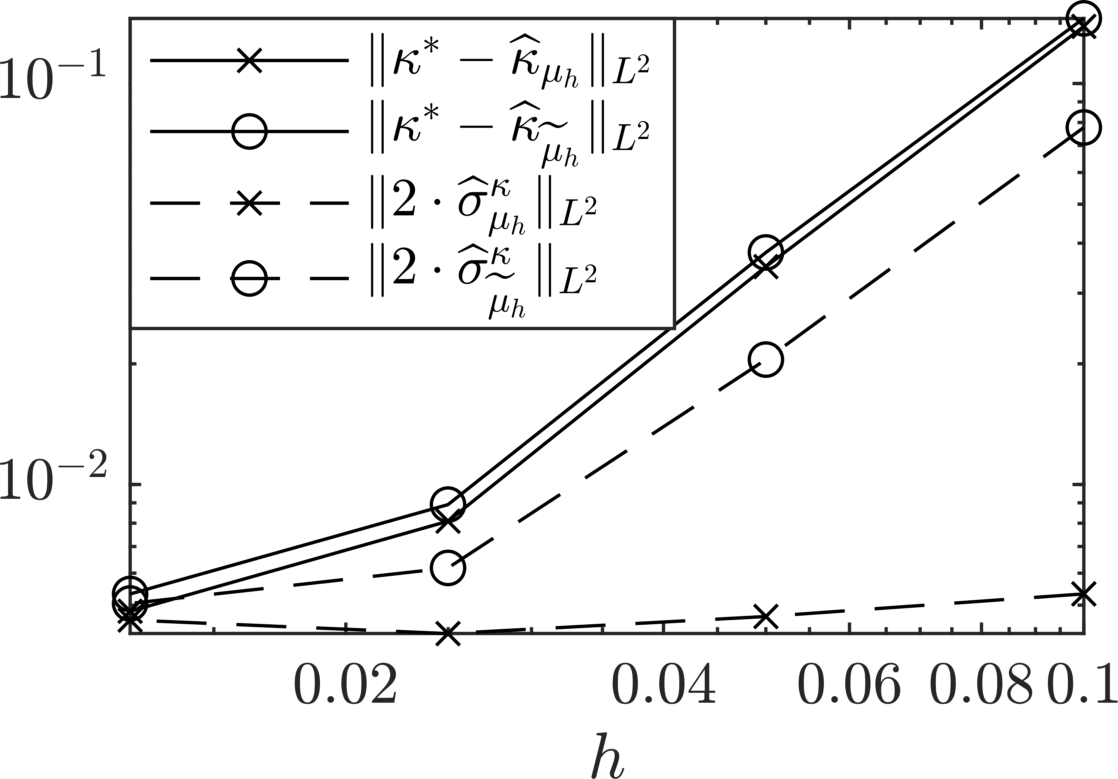} & 	\includegraphics[]{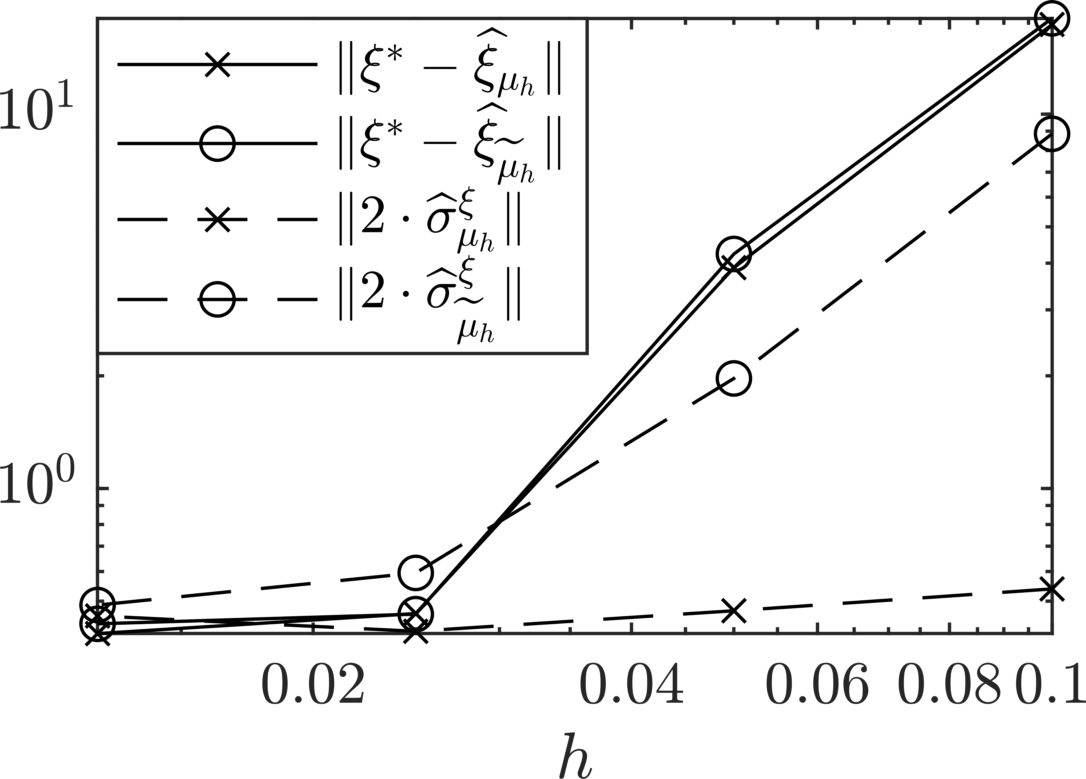}
	\end{tabular}
	
	\caption{Numerical results for \cref{sec:NumExpBIP_2D}. First row: True conductivity field $\kappa^*$ and posterior mean field $\widehat \kappa_{\mu_h}$ estimated with MCMC and different values of $h$. Second row: Mean error vs standard deviation under $\mu_h$ and $\widetilde \mu_h$. On the left, $L^2(D)$ error on the mean field vs $L^2(D)$ norm of the punctual standard deviation under $\mu_h$ and $\widetilde \mu_h$. On the right, error with respect to exact KL coefficients, and standard deviations under $\mu_h$ and $\widetilde \mu_h$.}
	\label{fig:Bay2d}
\end{figure}

We consider now a two dimensional example on the domain $D = (0,1)^2$. We fix a Gaussian prior $\mu_0$ on $\mathcal H$ for the log-conductivity $\theta$ chosen as $\mu_0 = \mathcal N(0, \Gamma_0)$, where $\Gamma_0 = -\Delta^{-1.3}$ with homogeneous boundary conditions, so that the inverse problem is well-posed. We fix $N_{\mathrm{KL}} = 6$ and let the true conductivity $\kappa^* = \exp(\theta^*)$ in \eqref{eq:PDE_Inverse} be given by
\begin{equation}
	\theta^* = \sum_{i=1}^{6} \sqrt{\lambda_i} \phi_i \xi_i^*,
\end{equation}
where $\{(\lambda_i, \phi_i)\}_{i=1}^6$ are the first six ordered eigenpairs of $\Gamma_0$, and where $\xi_i^* = (-1)^{i+1}\cdot 10$ for $i = 1, 2, \ldots, 6$. Let us remark that $\theta^* \in \mathcal H$. The right-hand side in \eqref{eq:PDE_Inverse} is chosen as $f(x, y) = 8\pi^2 \sin(2\pi x) \sin(2\pi y)$. Synthetic observations are obtained by evaluating a reference solution on $m = 50$ random locations sampled from $\mathcal U(D)$ and then corrupted by an observational noise distributed as $\mathcal N(0, 10^{-6}I)$.  We then approximate the forward map in the inverse problem with the FEM and the RM-FEM. We choose a structured mesh $\mathcal T_h$ as the one in \cref{ex:ExRandomPerturbation} (or the second row of \cref{fig:Mesh}). In particular, in this case we let $h$ denote the constant length of the short side of the triangular elements, i.e., the inverse of the number of subdivisions of each side of $D$. In particular, we consider $h = 0.1 \cdot 2^{-i}$, $i = 0,1,\ldots, 3$. The RM-FEM is implemented with $p = 1$ in \eqref{eq:PerturbedPoints} as per \cref{thm:APriori}, and with an uniform choice for the random perturbations as the one described in \cref{ex:ExRandomPerturbation}. 

Employing the notation introduced in \cref{sec:BIP_Practice}, we then sample from the posterior distributions $\mu_h$ and $\widetilde \mu_h$ employing the RAM method, considering only the first $N_{\mathrm{KL}} = 6$ coefficients in the KL expansion \eqref{eq:KarunhenLoeve}. In particular, we consider $N_{\mathrm{MC}} = 10^5$ samples for the deterministic case, and for the probabilistic case we generate $\widetilde N_{\mathrm{MC}}^{\mathrm{in}} = 10^5$ samples for $\widetilde N_{\mathrm{MC}}^{\mathrm{out}} = 24$ parallel chains, each corresponding to a realization of the random mesh in the RM-FEM. We then compute for each value of $h$ the mean and standard deviation of the field computed under $\mu_h$ (resp. $\widetilde \mu_h$) and denote their Monte Carlo approximations as $\widehat \kappa_{\mu_h}$ and $\widehat \sigma^\kappa_{\mu_h}$ (resp. $\widehat \kappa_{\widetilde \mu_h}$, $\widehat \sigma^\kappa_{\widetilde \mu_h}$). Moreover, we consider the statistics of the $6$-dimensional coefficient $\sigma$ of the KL expansion, and denote by $\widehat \xi_{\mu_h}$ and $\widehat \sigma_{\mu_h}^\xi$ the Monte Carlo approximation of mean and standard deviation computed under $\mu_h$ (resp. $\widehat \xi_{\mu_h}$, $\widehat \sigma_{\widetilde \mu_h}^\xi$). We show in \cref{fig:Bay2d} the posterior mean $\widehat \kappa_{\mu_h}$ for three values of $h$, compared to the truth $\kappa^*$, and remark that the mean approximation is sensibly better for smaller values of $h$. The mean value under the probabilistic posterior $\widetilde \mu_h$ is not shown, as it is essentially equal to the deterministic case. The beneficial effect of employing the RM-FEM-based posterior distribution $\widetilde \mu_h$, with respect to the FEM-based posterior $\mu_h$, consists of the approximate equalities
\begin{equation}
	\norm{\widehat\sigma_{\widetilde \mu_h}^\kappa}_{L^2(D)} = \mathcal O\left(\norm{\kappa^* - \widehat \kappa_{\widetilde \mu_h}}_{L^2(D)}\right), \qquad \norm{\widehat \sigma^\xi_{\widetilde \mu_h}} = \mathcal O\left(\norm{\xi^* - \widehat \xi_{\widetilde \mu_h}}\right),
\end{equation}
which indicate that the error on the conductivity field, or on the coefficients of its KL expansion, are well represented by the uncertainty in the posterior distribution. This is shown in \cref{fig:Bay2d}, where we notice that under $\mu_h$ the standard deviation is practically independent of $h$ and small with respect to the error on the solution of the inverse problem. Conversely, under $\widetilde \mu_h$ we have that the posterior standard deviation converges accordingly to the error, both for the $L^2$-norm of the error on the mean and for the coefficients of the KL expansion.

\section{Error Analysis for the RM-FEM}\label{sec:ProbErrEst}

In this section, we present our a priori and our a posteriori error analysis for the RM-FEM. Let us remark that while the a priori error analysis is carried on for a general space dimension $d$ and the adaptive algorithm has been shown to be efficient in higher dimensions (see \cref{sec:NumExp_ErrEst}), we present a rigorous a posteriori error analysis only in case $d = 1$. Conversely, in the a priori analysis we fix the coefficient $p = 1$ in \eqref{eq:PerturbedPoints}, whereas in the a posteriori analysis we consider general perturbations, i.e., general coefficients $p \geq 1$ in the same equality.

\subsection{A Priori Error Estimates}\label{sec:APriori}

We first prove the a priori error estimate given in \cref{thm:APriori}. The convergence properties of the FEM for the elliptic problem \eqref{eq:PDE} are well-established. In particular, without any additional assumptions on the exact solution, i.e., when $u \in V$, it holds $\norm{u-u_h}_V \to 0$ for $h \to 0$. Under the more restrictive assumption $u \in H^2(D) \cap V$, we have a linear convergence rate, i.e. 
\begin{equation}\label{eq:FEMAPriori}
	\norm{u-u_h}_V \leq Ch \abs{u}_{H^2(D)},
\end{equation}
for a constant $C > 0$, which is independent of $h$ and $u$ \cite{Cia02, Qua09, BrS08}. It is desirable that the RM-FEM is endowed with the same property. Moreover, we wish the error due to randomization to be balanced with the error due to the FEM discretization, which is shown in the proof of \cref{thm:APriori} below.
\begin{proof}[Proof of \cref{thm:APriori}] Since \eqref{eq:FEMAPriori} holds independently of the mesh, we have
	\begin{equation}
		\norm{u - \widetilde u_h}_V \leq \widetilde Ch\abs{u}_{H^2(D)}, \quad \text{a.s.},
	\end{equation}
	for a constant $\widetilde C$ independent of $h$ and $u$ and of the coefficient $p$ in \eqref{eq:PerturbedPoints}. Hence, by the triangle inequality we have for $p = 1$ 
	\begin{equation}\label{eq:APrioriDecomp}
		\norm{u_h - \widetilde u_h}_V \leq \norm{u - u_h}_V + \norm{u - \widetilde u_h}_V \leq (C+\widetilde C)h\abs{u}_{H^2(D)}, \quad \text{a.s.},
	\end{equation}
	i.e., we have $\mathcal O(\norm{u_h - \widetilde u_h}_V) = \mathcal O(\norm{u - u_h}_V)$, which shows the desired result.
\end{proof}	

Let us remark that we have shown above that the probabilistic solution converges with the same rate with respect to $h$ in case $p=1$, but we have not considered the case $p>1$, for which the probabilistic term may be of higher order. Indeed, a preliminary theoretical and numerical investigation leads us to conjecture that
\begin{equation}\label{eq:APrioriGenP}
	\left(\E\norm{u_h - \widetilde u_h}_V^2\right)^{1/2} \leq Ch^{(p+1)/2},
\end{equation}
so that, at least in the mean-square sense, the error due to randomization should converge faster than the error due to discretization if $p > 1$.

\subsection{A Posteriori Error Analysis in the One-Dimensional Case}\label{sec:ErrAnalysis}

In this section we prove our main result for the a posteriori error estimator of the RM-FEM given in \cref{def:ProbErrEst}, namely \cref{thm:MainThmAPosteriori}. Our goal is to prove in the one-dimensional case that the probabilistic a posteriori error estimators are reliable and efficient, i.e., that there exist positive constants $\widetilde C_{\mathrm{low}}$ and $\widetilde C_{\mathrm{up}}$ independent of $h$ and $u$ such that
\begin{equation}\label{eq:APosterioriBound_Prob}
	\widetilde C_{\mathrm{low}}\widetilde{\mathcal E}_{h,k} \leq \norm{u-u_h}_V \leq \widetilde C_{\mathrm{low}}\widetilde{\mathcal E}_{h,k}.
\end{equation}
for $k = \{1,2\}$. Consider the elliptic two-point boundary value problem
\begin{equation}\label{eq:PDE_1d}
	\begin{aligned}
		&-(\kappa u')' = f, \quad \text{in } D, \\
		&u(0) = u(1) = 0,
	\end{aligned}
\end{equation}
where $\kappa \in L^\infty(D)$ satisfies $\kappa(x) \geq \underline{\kappa}$ almost everywhere in $D$, and where we assume $f \in L^1(D)$. We recall that the notation for one-dimensional problems has been introduced and discussed in \cref{ex:ExRandomPerturbation} and at the end of \cref{sec:APosteriori}. Additionally, we introduce here for a function $w$ which is piecewise constant on each $K_i \in \mathcal T_h$ the jump operator 
\begin{equation}
	\dbrack{w}_{x_i} \defeq w\eval{K_i} - w\eval{K_{i+1}}, \qquad i=1, \ldots, N-1, \quad \dbrack{w}_{x_0} = \dbrack{w}_{x_N} = 0.
\end{equation}
Our strategy for proving that the error estimator introduced in \cref{def:ProbErrEst} satisfies \eqref{eq:APosterioriBound_Prob} relies on showing it is equivalent to known valid estimators. In particular, we consider the following estimator, defined in \cite[Definition 6.3]{BaR81}.

\begin{definition}\label{def:DetErrEst} Let $\kappa$ be the diffusion coefficient of \eqref{eq:PDE_1d} satisfy $\kappa \in \mathcal C^0(D)$ and $\kappa \geq \underline \kappa > 0$. We define the error estimator 
	\begin{equation}\label{eq:DetErrEst}
		\mathcal E_h^2 \defeq \sum_{j=1}^{N} \eta_j^2, \quad \eta_j \defeq \norm{\kappa^{-1}\ell_j}_{L^2(K_j)},
	\end{equation} 
	with $\ell_j \colon K_j \to \R$ the linear function defined by $\ell_j(x_{j-1}) = \tau_{j,1}$,  $\ell_j(x_j) = -\tau_{j,0}$ where
	\begin{equation}
		\tau_{j,k} = \frac{h_j}{h_{j-k+1} + h_{j-k}}\dbrack{u_h'}_{x_{j-k}}\kappa(x_{j-k}).
	\end{equation}
\end{definition}
Clearly, the quantity $\mathcal E_h$ is computable up to quadrature error due to the approximation of the local estimators $\eta_j$. Let us finally introduce more precisely the higher-order quantity $\Lambda$ appearing in \eqref{eq:BabuskaLambdaIntro}, i.e.,
\begin{equation}\label{eq:BabuskaLambda}
	\Lambda^2 \defeq h^\zeta \sum_{j=1}^N \int_{K_j} (f(x) + \ell_j'(x))^2 \dd x,
\end{equation}
where $\zeta \in (0, 1)$ is arbitrary and $\ell_j$ are the linear functions employed in \cref{def:DetErrEst}. We can now state the main result concerning the estimator $\mathcal E_h$, which summarizes \cite[Theorems 8.1 and 8.2]{BaR81}.

\begin{theorem}\label{thm:BabuskaThm} Let $\mathcal E_h$ and $\Lambda$ be defined in \cref{def:DetErrEst} and \eqref{eq:BabuskaLambda}, respectively. Then, it holds up to higher order terms in $h$
	\begin{equation}
		\norm{u-u_h}_V \leq C\left(\mathcal E_h^2 + \Lambda^2\right)^{1/2},
	\end{equation}
	for a constant $C$ independent of $h$ and of the solution $u$. If moreover the family of meshes $\mathcal T_h$ is $\lambda$-quasi-uniform and if $\kappa \in \mathcal C^2(D)$ and $f \in \mathcal C^1(D)$ then, up to higher order terms, it holds
	\begin{equation}
		C_{\mathrm{low}} \mathcal E_h \leq \norm{u-u_h}_V \leq C_{\mathrm{up}} \mathcal E_h,
	\end{equation}	
	for constants $C_{\mathrm{low}}$, $C_{\mathrm{up}}$ independent of $h$ and $u$.
\end{theorem}

We recall that in the one dimensional case the probabliistic error estimators for the RM-FEM are given by
\begin{equation}
\begin{alignedat}{2}
	&\widetilde{\mathcal E}_{h,1} \defeq \left(\sum_{i=1}^N \widetilde{\eta}_{K_i,1}^2\right)^{1/2}, \quad &&\text{ with } \quad \widetilde{\eta}_{K_i,1}^2 = h_i^{-(p-1)} \Eb{\norm{u_h' - (\widetilde{\mathcal I} u_h)'}_{L^2(\widetilde K_i)}^2}.
	\\
	&\widetilde{\mathcal E}_{h,2} \defeq \left(\sum_{i=1}^N \widetilde{\eta}_{K_i,2}^2\right)^{1/2}, \quad &&\text{ with } \quad \widetilde{\eta}_{K_i,2}^2 = h_i^{-(2p-3)} \Eb{\norm{u_h'^\eval{K} -  (\widetilde{\mathcal I} u_h)'\eval{\widetilde K}}^2}.
\end{alignedat}
\end{equation}

Our strategy to prove \cref{thm:MainThmAPosteriori} relies on showing that the deterministic estimator $\mathcal E_h$ of \cref{def:DetErrEst}, as well as its probabilistic counterparts $\widetilde{\mathcal E}_{h,1}$ and $\widetilde{\mathcal E}_{h,2}$ above are all equivalent to the quantity
\begin{equation}
	\mathcal J(u_h) \defeq \sum_{i=1}^{N-1} \bar h_i \dbrack{u_h'}^2,
\end{equation}
i.e., the sum of all squared jumps of the derivatives on the internal nodes. We first prove the equivalence for $\widetilde {\mathcal E}_{h,1}$. 

\begin{lemma}\label{lem:EquivProb1Jump} Let \cref{as:meshPerturbation} hold. Then, if the mesh is $\lambda$-quasi-uniform it holds
	\begin{equation}
	\left(\frac{\E\abs{\bar\alpha_1}\left(1+\lambda^{-(p-1)}\right)}{2} - 2h^{p-1}\E\abs{\bar\alpha_1}^2\right) \mathcal J^2(u_h)\leq \widetilde{\mathcal E}_{h,1}^2 \leq \frac{\E\abs{\bar\alpha_1}\left(1+\lambda^{p-1}\right)}{2} \mathcal J^2(u_h),
	\end{equation}
	where $\widetilde{\mathcal E}_{h,1}$ is given in \cref{def:ProbErrEst}.
\end{lemma}

\begin{proof} Let $\widetilde K_i$, $i = 1, \ldots, N$, be a generic element of the perturbed mesh and let us compute the derivative of the interpolant on $\widetilde K_i$, which is given by
	\begin{equation}
	(\widetilde{\mathcal I} u_h)'\eval{\widetilde K_i} = \frac{u_h(\widetilde x_i) - u_h(\widetilde x_{i-1})}{\widetilde x_i - \widetilde x_{i-1}},
	\end{equation}
	where an exact Taylor expansion allows to compute
	\begin{equation}
	u_h(\widetilde x_{i-1}) = u_h(x_{i-1}) + h^p\alpha_{i-1} u_h'(\widetilde x_{i-1}).
	\end{equation}
	Hence, it holds
	\begin{equation}
	(\widetilde{\mathcal I} u_h)'\eval{\widetilde K_i} = \frac{x_i-x_{i-1}}{\widetilde x_i - \widetilde x_{i-1}} u_h'\eval{K_i} + h^p \frac{\alpha_i u_h'(\widetilde x_i) - \alpha_{i-1} u_h'(\widetilde x_{i-1})}{\widetilde x_i - \widetilde x_{i-1}},
	\end{equation}
	which we can rewrite rearranging terms as
	\begin{equation}\label{eq:InterpDerDiff}
	(\widetilde{\mathcal I} u_h)'\eval{\widetilde K_i} - u_h'\eval{K_i} = h^p \frac{\alpha_i \left(u_h'(\widetilde x_i) - u_h'\eval{K_i}\right) + \alpha_{i-1} \left(u_h'\eval{K_i} - u_h'(\widetilde x_{i-1})\right)}{\widetilde x_i - \widetilde x_{i-1}}.
	\end{equation}
	It is clear then that the expression above depends on the signs of the variables $\alpha_{i-1}$ and $\alpha_i$. For simplicity of notation, we therefore introduce the events $A_{i,j}^{(s_i, s_j)} \in \mathcal A$, where $s_i, s_j \in \{+,-\}$, defined as
	\begin{equation}
	A_{i,j}^{(s_i, s_j)} \defeq \{\omega \in \Omega \colon \alpha_i(\omega) \in \R^{s_i}, \alpha_j(\omega) \in \R^{s_j}\}.
	\end{equation} 
	We now define $e_h \defeq u_h - \widetilde{\mathcal I}u_h$ and write for any $i = 1, \ldots, N$
	\begin{equation}
		\E\norm{e_h'}_{L^2(\widetilde K_i)}^2 = \E\left(I_{i-1, i} + I_i + I_{i+1,i} \right),
	\end{equation}
	with
	\begin{equation}
		I_{i,j} \defeq \int_{K_i \cap \widetilde K_j}(e_h')^2 \dd x,
	\end{equation}
	and where we write $I_i \defeq I_{i,i}$ and adopt the convention $I_{0,1} = I_{N+1,N} = 0$. In what follows we study $\E I_{i,j}$. We first consider $I_i$, which we express by the law of total expectation as
	\begin{equation}\label{eq:IiSplit}
		\E I_i = \sum_{s_{i-1},s_i \in \{-,+\}} \E \left[I_i \mid A_{i-1,i}^{(s_{i-1},s_i)}\right]P(A_{i-1,i}^{(s_{i-1},s_i)}). 
	\end{equation}
	In the trivial case $\alpha_{i-1} > 0$ and $\alpha_i < 0$, i.e., if $A_{i-1,i}^{(+,-)}$ occurs, we have $\widetilde K_i \cap K_i = \widetilde K_i$ and therefore $\E[I_i\mid A_{i-1,i}^{(+,-)}] = 0$. If $A_{i-1,i}^{(-,-)}$ occurs, the equality \eqref{eq:InterpDerDiff} simplifies to
	\begin{equation}
		(\widetilde{\mathcal I} u_h)'\eval{\widetilde K_i} - u_h'\eval{K_i} = -\frac{h^p\alpha_{i-1}}{\widetilde x_i - \widetilde x_{i-1}} \dbrack{u_h'}_{x_{i-1}}.
	\end{equation}
	Since in this case $\abs{K_i \cap \widetilde K_i} = \widetilde x_i - x_{i-1}$, integrating yields
	\begin{equation}\label{eq:Ii--}
		I_i = \frac{h^{2p}(\widetilde x_i - x_{i-1})}{(\widetilde x_i - \widetilde x_{i-1})^2} \alpha_{i-1}^2\dbrack{u_h'}_{x_{i-1}}^2.
	\end{equation}
	Similar calculations allow to show that if $A_{i-1,i}^{(+,+)}$ occurs, it holds
	\begin{equation}\label{eq:Ii++}
		I_i = \frac{h^{2p}(x_i - \widetilde x_{i-1})}{(\widetilde x_i - \widetilde x_{i-1})^2} \alpha_i^2 \dbrack{u_h'}_{x_i}^2,
	\end{equation}
	Finally, if $A_{i-1,i}^{(-,+)}$ occurs, we get
	\begin{equation}\label{eq:Ii-+}
		I_i = \frac{h^{2p}(x_i - x_{i-1})}{(\widetilde x_i - \widetilde x_{i-1})^2} \xi_i^2.
	\end{equation}
	where we denote 
	\begin{equation}\label{eq:ProofXi}
		\xi_i \defeq \alpha_{i-1}\dbrack{u_h'}_{x_{i-1}} + \alpha_i\dbrack{u_h'}_{x_i}.
	\end{equation}
	We thus have an expression for $\E I_i$ due to \eqref{eq:IiSplit}. We now turn to $I_{i-1,i}$. Since $\widetilde K_i \cap K_{i-1} = \emptyset$ if $\alpha_{i-1} > 0$, we have by the law of total expectation
	\begin{equation}\label{eq:Ii-1iSplit}
		\E I_{i-1,i} = \E\left[I_{i-1,i}\mid A_{i-1,i}^{(-,-)}\right]P(A_{i-1,i}^{(-,-)}) +  \E\left[I_{i-1,i} \mid A_{i-1,i}^{(-,+)}\right]P(A_{i-1,i}^{(-,+)}).
	\end{equation}
	Let us remark that adding and subtracting $u_h'\eval{K_i}$ yields
	\begin{equation}\label{eq:InterpDiffDerivatives2}
		(\mathcal{\widetilde I} u_h)'\eval{\widetilde K_i} - u_h'\eval{K_{i-1}} = (\mathcal{\widetilde I} u_h)'\eval{\widetilde K_i} - u_h'\eval{K_i} - \dbrack{u_h'}_{x_{i-1}}.
	\end{equation}
	The same computations employed for $I_i$ allow to conclude that 
	\begin{equation}
		I_{i-1,i} =
		\left\{
		\begin{alignedat}{2}
			&-h^p \alpha_{i-1} \left(\frac{h^p\alpha_{i-1}}{\widetilde x_i - \widetilde x_{i-1}} + 1\right)^2 \dbrack{u_h'}_{x_{i-1}}^2, &&\text{if } A_{i-1, i}^{(-,-)} \text{ occurs}, \\
			& -h^p \alpha_{i-1} \left(\frac{h^p}{\widetilde x_i - \widetilde x_{i-1}} \xi_i + \dbrack{u_h'}_{x_{i-1}}\right)^2, \quad&& \text{if } A_{i-1, i}^{(-,+)} \text{ occurs}.
		\end{alignedat}
		\right.
	\end{equation}
	which, replaced into \eqref{eq:Ii-1iSplit} gives the final expression for $\E I_{i-1, i}$. Similarly, for $I_{i+1,i}$ we have 
	\begin{equation}
		\E I_{i+1,i} = \E\left[I_{i+1,i}\mid A_{i-1,i}^{(+,+)}\right]P(A_{i-1,i}^{(+,+)}) + \E\left[I_{i+1,i} \mid A_{i-1,i}^{(-,+)}\right]P(A_{i-1,i}^{(-,+)}),
	\end{equation}
	where 
	\begin{equation}
		I_{i+1,i} =
		\left\{
		\begin{alignedat}{2}
			&h^p \alpha_i \left(\frac{h^p\alpha_i}{\widetilde x_i - \widetilde x_{i-1}} - 1\right)^2 \dbrack{u_h'}_{x_i}^2,  &&\text{if } A_{i-1, i}^{(+,+)} \text{ occurs}, \\
			& h^p \alpha_i \left(\frac{h^p}{\widetilde x_i - \widetilde x_{i-1}} \xi_i - \dbrack{u_h'}_{x_i}\right)^2, \quad&& \text{if } A_{i-1, i}^{(-,+)} \text{ occurs}.
		\end{alignedat}
		\right.
	\end{equation}
	We now reassemble the quantity $I_i + I_{i-1,i} + I_{i+1,i}$ by grouping terms with regards to their conditioning on the sign of $(\alpha_{i-1},\alpha_i)$. In particular, some algebraic simplifications yield 
	\begin{equation}
		I_i + I_{i-1,i} + I_{i+1,i} =
		\left\{
		\begin{alignedat}{2}
			&h^p\alpha_i \dbrack{u_h'}_{x_i}^2 - \frac{h^{2p}\alpha_i^2}{\widetilde x_i - \widetilde x_{i-1}}\dbrack{u_h'}_{x_i}^2,  &&\text{if } A_{i-1, i}^{(+,+)} \text{ occurs}, \\
			&-h^p\alpha_{i-1} \dbrack{u_h'}_{x_{i-1}}^2 - \frac{h^{2p}\alpha_{i-1}^2}{\widetilde x_i - \widetilde x_{i-1}}\dbrack{u_h'}_{x_{i-1}}^2, &&\text{if } A_{i-1, i}^{(-,-)} \text{ occurs}, \\
			& h^p\alpha_i \dbrack{u_h'}_{x_i}^2 - h^p\alpha_{i-1} \dbrack{u_h'}_{x_{i-1}}^2 - \frac{h^{2p}}{\widetilde x_i - \widetilde x_{i-1}} \xi_i^2, \quad&& \text{if } A_{i-1, i}^{(-,+)} \text{ occurs}.
		\end{alignedat}
		\right.
	\end{equation}
	We now can compute the estimator $\widetilde{\mathcal E}_{h,1}$ by summing its local contributions, as in
	\begin{equation}
	\begin{aligned}
		\widetilde{\mathcal E}_{h,1}^2 &= \sum_{i=1}^N \eta_{K,1}^2 = \sum_{i=1}^N h_i^{-(p-1)} \norm{e_h'}_{L^2(\widetilde K_i)}^2\\
		&= \sum_{i=1}^N h_i^{-(p-1)} \E(I_i + I_{i-1,i} + I_{i+1,i}) \eqdef J_1 + J_2,
	\end{aligned}
	\end{equation}
	where $J_1$ and $J_2$ are given by
	\begin{equation}
		\begin{aligned}
			J_1 &\defeq \frac{h^p}{2} \sum_{i=1}^N h_i^{-(p-1)}\left(\dbrack{u_h'}_{x_i}^2 \Eb{\alpha_i  \mid \alpha_i > 0} -  \dbrack{u_h'}_{x_{i-1}}^2 \Eb{\alpha_{i-1}\mid \alpha_i < 0}\right),
			\\
			J_2 &\defeq -\frac{h^{2p}}{4} \sum_{i=1}^N h_i^{-(p-1)}
			\begin{aligned}[t]&\left(\dbrack{u_h'}_{x_{i-1}}^2\Eb{\frac{\alpha_{i-1}^2}{\widetilde x_i - \widetilde x_{i-1}} \mid A_{i-1,i}^{(-,-)}} + \dbrack{u_h'}_{x_i}^2 \Eb{\frac{\alpha_i^2}{\widetilde x_i - \widetilde x_{i-1}} \mid A_{i-1,i}^{(+,+)}}\right.\\
				&\left. + \Eb{\frac{\xi_i^2}{\widetilde x_i - \widetilde x_{i-1}} \mid A_{i-1,i}^{(-,+)}} \right).
			\end{aligned}
		\end{aligned}
	\end{equation}
	Let us consider $J_1$ and $J_2$ separately. Rearranging the sum, noticing that under \cref{as:meshPerturbation}\ref{as:meshPerturbation_sym} it holds $\Eb{\alpha_i \mid \alpha_i > 0} = -\Eb{\alpha_i \mid \alpha_i < 0} = \E\abs{\alpha_i}$ and recalling that $\alpha_i = (\bar h_i h^{-1})^p \bar \alpha_i$, we obtain
	\begin{equation}
	\begin{aligned}
		J_1 &= \frac{h^p}2 \sum_{i=1}^{N-1} \left(h_i^{-(p-1)} + h_{i+1}^{-(p-1)}\right)\dbrack{u_h'}_{x_i}^2 \Eb{\alpha_i  \mid \alpha_i > 0}\\
		&= \frac{\E\abs{\bar\alpha_1}}{2} \sum_{i=1}^{N-1} \left(h_i^{-(p-1)} + h_{i+1}^{-(p-1)}\right) \bar h_i^p \dbrack{u_h'}_{x_i}^2 .
	\end{aligned}
	\end{equation}
	Now, let us remark that if the mesh is $\lambda$-quasi-uniform, it holds
	\begin{equation}
		\left(1+\lambda^{-(p-1)}\right) \bar h_i \leq \left(h_i^{-(p-1)} + h_{i+1}^{-(p-1)}\right)\bar h_i^p \leq \left(1+\lambda^{p-1}\right) \bar h_i,
	\end{equation}
	which implies 
	\begin{equation}\label{eq:ProofInterpBoundJ1}
		\frac{\E\abs{\bar\alpha_1}\left(1+\lambda^{-(p-1)}\right)}{2} \mathcal J^2(u_h) \leq J_1 \leq \frac{\E\abs{\bar\alpha_1}\left(1+\lambda^{p-1}\right)}{2} \mathcal J^2(u_h).
	\end{equation}
	We now turn to $J_2$. Clearly, we have $J_2 \leq 0$, which implies the desired upper bound together with \eqref{eq:ProofInterpBoundJ1}. For the lower bound, we remark that in both cases $A_{i-1,i}^{(+,+)}$ and $A_{i-1,i}^{(-,-)}$ occur, we have that $\widetilde x_i - \widetilde x_{i-1} \geq h_i/2$, and if $A_{i-1,i}^{(-,+)}$ occurs, we have $\widetilde x_i - \widetilde x_{i-1} \geq h_i$. Hence, simplifying the conditioning in the first and second terms, we obtain
	\begin{equation}
		J_2 \geq -\frac{h^{2p}}{4}\sum_{i=1}^N h_i^{-p}
		\begin{aligned}[t]
			&\left(2\dbrack{u_h'}_{x_i}^2\Eb{\alpha_i^2 \mid \alpha_i > 0} + 2\dbrack{u_h'}_{x_{i-1}}^2 \Eb{\alpha_{i-1}^2\mid \alpha_{i-1} < 0} \right. \\
			&+ \left. \Eb{\xi_i^2\mid A_{i-1,i}^{(-, +)}}\right).	
		\end{aligned}
	\end{equation}
	We now consider $\xi_i$ given in \eqref{eq:ProofXi} and use $(a+b)^2 \leq 2(a^2 + b^2)$ for $a = \alpha_{i-1}\dbrack{u_h'}_{x_{i-1}}$ and $b = \alpha_i \dbrack{u_h'}_{x_i}$ to obtain
	\begin{equation}
		\Eb{\xi_i^2\mid A_{i-1,i}^{(-, +)}} \leq 2\dbrack{u_h'}_{x_{i-1}}^2 \Eb{\alpha_{i-1}^2 \mid \alpha_{i-1} < 0} + 2\dbrack{u_h'}_{x_i}^2 \Eb{\alpha_i^2 \mid \alpha_i > 0}. 
	\end{equation}
	Therefore 
	\begin{equation}
		J_2 \geq -h^{2p}\sum_{i=1}^N h_i^{-p} \left(\dbrack{u_h'}_{x_i}^2\Eb{\alpha_i^2 \mid \alpha_i > 0} + \dbrack{u_h'}_{x_{i-1}}^2 \Eb{\alpha_{i-1}^2\mid \alpha_{i-1} < 0}\right).
	\end{equation}
	Rewriting the sum and replacing the definition of $\alpha_i$ yields
	\begin{equation}
		J_2 \geq -\sum_{i=1}^{N-1}\bar h_i^{2p}\dbrack{u_h'}_{x_i}^2 \left(h_i^{-p}\Eb{\bar\alpha_i^2 \mid \bar\alpha_i > 0} + h_{i+1}^{-p}\Eb{\bar \alpha_i^2 \mid \bar\alpha_i < 0}\right).
	\end{equation}
	Now $\bar h_i = \min\{h_i, h_{i+1}\}$ implies $h_i^{-p} \leq \bar h_i^{-p}$ and $h_{i+1}^{-p} \leq \bar h_i^{-p}$, which gives
	\begin{equation}
	\begin{aligned}
		J_2 &\geq -2\sum_{i=1}^{N-1} \bar h_i^{p}\dbrack{u_h'}_{x_i}^2 \left(\frac12\Eb{\bar\alpha_i^2 \mid \bar \alpha_i > 0} + \frac12 \Eb{\bar \alpha_i^2 \mid \bar \alpha_i < 0}\right) \\
		&\geq -2\E\abs{\bar \alpha_1}^2 \sum_{i=1}^{N-1} \bar h_i^{p}\dbrack{u_h'}_{x_i}^2,
	\end{aligned}
	\end{equation}
	where we applied the law of total expectation on the second line. Finally, we have $\bar h_i \leq h$ and $p \geq 1$, which yield
	\begin{equation}
		J_2 \geq -2 \E\abs{\bar \alpha_1}^2 h^{p-1} \mathcal J^2(u_h).
	\end{equation}
	Combining this with \eqref{eq:ProofInterpBoundJ1} then yields the desired lower bound and thus concludes the proof.	
\end{proof}

Let us remark that the coefficient appearing in the lower bound of \cref{lem:EquivProb1Jump} is positive if \cref{as:AssumptionAPosteriori} holds. We now prove the equivalence of the estimator $\widetilde{\mathcal E}_{h,2}$ given in \cref{def:ProbErrEst} with $\mathcal J(u_h)$. 

\begin{lemma}\label{lem:EquivProb2Jump} Let \cref{as:meshPerturbation} hold and let the mesh $\mathcal T_h$ be $\lambda$-quasi uniform. Then, it holds
	\begin{equation}
		\frac{\E\abs{\bar \alpha_1}^2}{2(1+\lambda)^2\lambda^{2p-1}} \mathcal J^2(u_h) \leq \widetilde{\mathcal E}_{h,2}^2 \leq 3\E\abs{\bar \alpha_1}^2 \mathcal J^2(u_h),
	\end{equation}
	where $\widetilde {\mathcal E}_{h,2}$ is given in \cref{def:ProbErrEst}.
\end{lemma}

\begin{proof} As $\abs{K_i} = h_i$, we have
	\begin{equation}
		\widetilde{\mathcal E}_{h,2} = \sum_{i=1}^N h_i^{-(2p-3)} \Eb{\abs{u_h'\eval{K_i} - (\widetilde {\mathcal I} u_h)'\eval{\widetilde K_i}}^2}.
	\end{equation}
	Proceeding similarly to \eqref{eq:Ii--}, \eqref{eq:Ii++} and \eqref{eq:Ii-+} and applying the law of total expectation, we obtain
	\begin{equation}\label{eq:ProofSecondInterp}
	\begin{split}
		\Eb{\abs{u_h'\eval{K_i} - (\widetilde {\mathcal I} u_h)'\eval{\widetilde K_i}}^2} = 
		\begin{aligned}[t]
			&\frac{h^{2p}}4\dbrack{u_h'}_{x_i}^2\Eb{\frac{\alpha_i^2}{(\widetilde x_i - \widetilde x_{i-1})^2} \mid A_{i-1,i}^{(+,+)}}\\
			&+ \frac{h^{2p}}4\dbrack{u_h'}_{x_{i-1}}^2 \Eb{\frac{\alpha_{i-1}^2}{(\widetilde x_i - \widetilde x_{i-1})^2} \mid A_{i-1,i}^{(-,-)}} \\
			&+ \frac{h^{2p}}4 \Eb{\frac{\xi_i^2}{(\widetilde x_i - \widetilde x_{i-1})^2} \mid A_{i-1,i}^{(-,+)}},
		\end{aligned}
	\end{split}	
	\end{equation}
	where we recall the notation $\xi_i$ introduced in \eqref{eq:ProofXi}. Let us first consider the lower bound. Since $\xi_i^2 \geq 0$ a.s., and $\widetilde x_i - \widetilde x_{i-1} \leq (1+\lambda)h_i$ a.s. under the assumption that the mesh is $\lambda$-quasi-uniform, we have 
	\begin{equation}
		\Eb{\abs{u_h'\eval{K_i} - (\widetilde {\mathcal I} u_h)'\eval{\widetilde K_i}}^2} \geq 
		\frac{h^{2p} h_i^{-2}}{4(1+\lambda)^2}
		\begin{aligned}[t]
			&\left(\dbrack{u_h'}_{x_i}^2 \Eb{\alpha_i^2 \mid \alpha_i > 0} \right. \\
		 	&\left. + \dbrack{u_h'}_{x_{i-1}}^2 \Eb{\alpha_{i-1}^2 \mid \alpha_{i-1} < 0} \right).
		\end{aligned}
	\end{equation}	
	Assembling the sum, rearranging terms and recalling that $\alpha_i = (h^{-1}\bar h_i)^{p} \bar \alpha_i$ with $\bar h_i = \min\{h_i, h_{i+1}\}$, we then obtain
	\begin{equation}
	\begin{aligned}
		\widetilde{\mathcal E}_{h,2}^2 &\geq 
		\frac1{2(1+\lambda)^2}\sum_{i=1}^{N-1} \bar h_i^{2p} h_i^{1-2p} \dbrack{u_h'}_{x_i}^2 \left(\frac12\Eb{\bar\alpha_i^2 \mid \bar\alpha_i > 0} + \frac12\Eb{\bar\alpha_i^2 \mid \bar\alpha_i < 0}\right) \\
		&\geq \frac{\E\abs{\bar\alpha_1}^2}{2(1+\lambda)^2\lambda^{2p-1}} \sum_{i=1}^{N-1} \bar h_i  \dbrack{u_h'}_{x_i}^2  = \frac{\E\abs{\bar\alpha_1}^2}{2(1+\lambda)^2\lambda^{2p-1}}\mathcal J^2(u_h),
	\end{aligned}
	\end{equation}
	where we employed the law of total expectation and the inequality $h_i^{1-2p} \leq \lambda^{1-2p} \bar h_i^{1-2p}$ on the second line. Hence, we proved the lower bound. For the upper bound, using again the inequality $(a+b)^2 \leq 2(a^2 + b^2)$ we obtain 
	\begin{equation}
		\xi_i^2 \leq 2\alpha_i^2 \dbrack{u_h'}_{x_i}^2 + 2\alpha_{i-1}^2 \dbrack{u_h'}_{x_{i-1}}^2, \quad \text{a.s},
	\end{equation}
	so that
	\begin{equation}
		\Eb{\frac{\xi_i^2}{(\widetilde x_i - \widetilde x_{i-1})^2} \mid A_{i-1,i}^{(-,+)}} \leq 
		\begin{aligned}[t] 
			&2 \dbrack{u_h'}_{x_i}^2 \Eb{\frac{\alpha_i^2}{(\widetilde x_i - \widetilde x_{i-1})^2} \mid A_{i-1,i}^{(-,+)}} \\
			&+ 2\dbrack{u_h'}_{x_{i-1}}^2 \Eb{\frac{\alpha_{i-1}^2}{(\widetilde x_i - \widetilde x_{i-1})^2} \mid A_{i-1,i}^{(-,+)}}.
		\end{aligned}
	\end{equation}
	Under $A_{i-1,i}^{(-,+)}$, we have $\widetilde x_i - \widetilde x_{i-1} \geq h_i$, which implies
	\begin{equation}
		\Eb{\frac{\xi_i^2}{(\widetilde x_i - \widetilde x_{i-1})^2} \mid A_{i-1,i}^{(-,+)}} \leq 2h_i^{-2} \left( \dbrack{u_h'}_{x_i}^2 \Eb{\alpha_i^2 \mid \alpha_i > 0} + \dbrack{u_h'}_{x_{i-1}}^2 \Eb{\alpha_{i-1}^2 \mid \alpha_{i-1} < 0} \right).
	\end{equation}
	Then, considering that under $A_{i-1,i}^{(+,+)}$ or $A_{i-1,i}^{(-,-)}$ it holds $\widetilde x_i - \widetilde x_{i-1} \geq h_i/2$ and plugging into \eqref{eq:ProofSecondInterp} we have
	\begin{equation}
		\Eb{\abs{u_h'\eval{K_i} - (\widetilde {\mathcal I} u_h)'\eval{\widetilde K_i}}^2} \leq 
		\frac32 h_i^{-2} h^{2p} \left( \dbrack{u_h'}_{x_i}^2 \Eb{\alpha_i^2 \mid \alpha_i > 0} + \dbrack{u_h'}_{x_{i-1}}^2 \Eb{\alpha_{i-1}^2 \mid \alpha_{i-1} < 0} \right).
	\end{equation}
	We can therefore reassemble and rearrange the sum following the same procedure as for the lower bound, which, together with $h_i^{1-2p} \leq \bar h_i^{1-2p}$, yields
	\begin{equation}
		\widetilde{\mathcal E}_{h,2}^2 \leq 3\E\abs{\bar\alpha_1}^2 \mathcal J^2(u_h),
	\end{equation}
	which proves the desired result.
\end{proof}

We finally prove the equivalence of the deterministic error estimator $\mathcal E_h$ given in \cref{def:DetErrEst} with the quantity $\mathcal J(u_h)$.

\begin{lemma}\label{lem:EquivDetJump} Let the mesh $\mathcal T_h$ be $\lambda$-quasi-uniform. Then, it holds
	\begin{equation}
		\frac{\lambda m^2}{6(1+\lambda)^3 M^2}\mathcal J^2(u_h) \leq \mathcal E_h^2 \leq \frac{2\lambda^2 M^2}{3(1+\lambda)m^2} \mathcal J^2(u_h),
	\end{equation}
	where $\mathcal E_h$ is given in \cref{def:DetErrEst} and where $m = \underline \kappa$ and $M = \norm{\kappa}_{L^\infty(D)}$.
\end{lemma}
\begin{proof} Simple algebraic computations yield
	\begin{equation}
		\norm{\ell_j}_{L^2(K_j)}^2 = \frac{h_j}{3} \left(\tau_{j,0}^2 - \tau_{j,0}\tau_{j,1} + \tau_{j,1}^2\right),
	\end{equation}
	where $\ell_j$ are the linear functions employed in \cref{def:DetErrEst}. Applying the inequalities $(a^2+b^2)/2 \leq a^2-ab+b^2 \leq 2(a^2 + b^2)$ we obtain
	\begin{equation}
		\frac{h_j}{6M^2} \left(\tau_{j,0}^2 + \tau_{j,1}^2\right) \leq \eta_j^2 \leq \frac{2h_j}{3m^2} \left(\tau_{j,0}^2 + \tau_{j,1}^2\right).
	\end{equation}
	We now remark that if the mesh $\mathcal T_h$ is $\lambda$-quasi-uniform and under the assumptions on $\kappa$ it holds for $k \in \{0,1\}$
	\begin{equation}
		\frac{m^2}{(1+\lambda)^2} \dbrack{u_h'}_{x_{j-k}}^2\leq \tau_{j,k}^2 \leq \frac{\lambda^2 M^2}{(1+\lambda)^2}\dbrack{u_h'}_{x_{j-k}}^2,
	\end{equation}
	which, in turn, implies
	\begin{equation}
		\frac{m^2h_j }{6(1+\lambda)^2M^2} \left(\dbrack{u_h'}_{x_{j-1}}^2 + \dbrack{u_h'}_{x_j}^2\right) \leq \eta_j^2 \leq \frac{2\lambda^2 M^2 h_j}{3(1+\lambda)^2m^2} \left(\dbrack{u_h'}_{x_{j-1}}^2 + \dbrack{u_h'}_{x_j}^2\right).
	\end{equation}
	We now focus on the upper bound. Reassembling the global error estimator $\mathcal E_h$, we have
	\begin{equation}
		\begin{aligned}
			\mathcal E_h^2 &\leq \frac{2\lambda^2 M^2}{3(1+\lambda)^2m^2} \sum_{j=1}^N h_j\left(\dbrack{u_h'}_{x_{j-1}}^2 + \dbrack{u_h'}_{x_j}^2\right)\\
			&= \frac{2\lambda^2 M^2}{3(1+\lambda)^2m^2} \sum_{j=1}^{N-1} (h_j + h_{j+1}) \dbrack{u_h'}_{x_j}^2 \\
			&\leq \frac{2\lambda^2 M^2}{3(1+\lambda)m^2} \mathcal J^2(u_h),
		\end{aligned}
	\end{equation}
	where we recall $\bar h_j = \min\{h_j, h_{j+1}\}$, so that $h_j + h_{j+1} \leq (1+\lambda)\bar h_j$. We conclude the proof proceeding similarly for the lower bound as in \cref{lem:EquivProb2Jump}.
\end{proof}

We can finally prove \cref{thm:MainThmAPosteriori} and conclude the error analysis.

\begin{proof}[Proof of \cref{thm:MainThmAPosteriori}] Let us first consider $\widetilde{\mathcal E}_{h,1}$. Under \cref{as:AssumptionAPosteriori}, we have for the lower bound of \cref{lem:EquivProb1Jump}
	\begin{equation}
		\left(\frac{\E\abs{\bar\alpha_1}\left(1+\lambda^{-(p-1)}\right)}{2} - 2h^{p-1}\E\abs{\bar\alpha_1}^2\right) \mathcal J^2(u_h) \geq C \E\abs{\bar\alpha_1} \mathcal J^2(u_h)
	\end{equation}
	for a constant $C > 0$. Hence, due to \cref{lem:EquivDetJump} we have that there exists a constant $\widehat C$ such that
	\begin{equation}
		\widetilde{\mathcal E}_{h,1} \geq \widehat C \mathcal E_h,
	\end{equation}
	and therefore, \cref{thm:BabuskaThm} implies
	\begin{equation}
		\norm{u-u_h}_V \leq C_{\mathrm{up}} \mathcal E_h \leq C_{\mathrm{up}} \widehat C\widetilde{\mathcal E}_{h,1},
	\end{equation}
	which yields the desired upper bound with $\widetilde C_{\mathrm{up}} = \widehat C C_{\mathrm{up}}$. The lower bound follows equivalently under the additional regularity required by \cref{thm:BabuskaThm}. Similarly, the results for $\widetilde{\mathcal E}_{h,2}$ follows from \cref{lem:EquivProb1Jump,lem:EquivDetJump}, together with \cref{thm:BabuskaThm}.
\end{proof}

\section{Conclusion}\label{sec:Conclusion}

We have introduced a novel probabilistic methods for PDEs based on the FEM and random meshes, the RM-FEM. We demonstrated how our methodology can be successful when employed in pipelines of computations, such as Bayesian inverse problems. We also show a rigorous use of probabilistic methods for a posteriori error estimators, often speculated in the field. Extending such analysis to PN methods for ODEs would be of interest, thus creating a link between the guiding principles of PN and more classical theories. Generalizing the analysis of the RM-FEM to higher-dimensional PDEs as well as for parabolic or hyperbolic problems represent also interesting future work.

\subsection*{Acknowledgements} 

The authors are partially supported by the Swiss National Science Foundation, under grant No. 200020\_172710.

\def\cprime{$'$}


\begin{thebibliography}{10}
	
	\bibitem{AbD20}
	{\sc A.~Abdulle and A.~Di~Blasio}, {\em A {B}ayesian {N}umerical
		{H}omogenization {M}ethod for {E}lliptic {M}ultiscale {I}nverse {P}roblems},
	SIAM/ASA J. Uncertain. Quantif., 8 (2020), pp.~414--450.
	
	\bibitem{AbG20}
	{\sc A.~Abdulle and G.~Garegnani}, {\em Random time step probabilistic methods
		for uncertainty quantification in chaotic and geometric numerical
		integration}, Stat. Comput., 30 (2020), pp.~907--932.
	
	\bibitem{AGZ20}
	{\sc A.~Abdulle, G.~Garegnani, and A.~Zanoni}, {\em Ensemble {K}alman {F}ilter
		for {M}ultiscale {I}nverse {P}roblems}, Multiscale Model. Simul., 18 (2020),
	pp.~1565--1594.
	
	\bibitem{AiO00}
	{\sc M.~Ainsworth and J.~T. Oden}, {\em A posteriori error estimation in finite
		element analysis}, Pure and Applied Mathematics (New York),
	Wiley-Interscience [John Wiley \& Sons], New York, 2000.
	
	\bibitem{BaR81}
	{\sc I.~Babu\v{s}ka and W.~C. Rheinboldt}, {\em A posteriori error analysis of
		finite element solutions for one-dimensional problems}, SIAM J. Numer. Anal.,
	18 (1981), pp.~565--589.
	
	\bibitem{BHT20}
	{\sc N.~Bosch, P.~Hennig, and F.~Tronarp}, {\em Calibrated adaptive
		probabilistic {ODE} solvers}.
	\newblock arXiv preprint arXiv:2012.08202, 2020.
	
	\bibitem{BrS08}
	{\sc S.~C. Brenner and L.~R. Scott}, {\em The mathematical theory of finite
		element methods}, vol.~15 of Texts in Applied Mathematics, Springer, New
	York, third~ed., 2008.
	
	\bibitem{CDS18}
	{\sc D.~Calvetti, M.~Dunlop, E.~Somersalo, and A.~M. Stuart}, {\em Iterative
		updating of model error for {B}ayesian inversion}, Inverse Problems, 34
	(2018), pp.~025008, 38.
	
	\bibitem{CES14}
	{\sc D.~Calvetti, O.~Ernst, and E.~Somersalo}, {\em Dynamic updating of
		numerical model discrepancy using sequential sampling}, Inverse Problems, 30
	(2014), pp.~114019, 19.
	
	\bibitem{ChC19}
	{\sc O.~A. Chkrebtii and D.~A. Campbell}, {\em Adaptive step-size selection for
		state-space probabilistic differential equation solvers}, Stat. Comput., 29
	(2019), pp.~1285--1295.
	
	\bibitem{CCC16}
	{\sc O.~A. Chkrebtii, D.~A. Campbell, B.~Calderhead, and M.~A. Girolami}, {\em
		Bayesian solution uncertainty quantification for differential equations},
	Bayesian Anal., 11 (2016), pp.~1239--1267.
	
	\bibitem{Cia02}
	{\sc P.~G. Ciarlet}, {\em The finite element method for elliptic problems.},
	vol.~40 of Classics Appl. Math., SIAM, Philadelphia, 2002.
	
	\bibitem{COS17b}
	{\sc J.~Cockayne, C.~J. Oates, T.~J. Sullivan, and M.~Girolami}, {\em
		Probabilistic numerical methods for partial differential equations and
		{B}ayesian inverse problems}.
	\newblock arXiv preprint arXiv:1605.07811, 2017.
	
	\bibitem{COS17}
	{\sc J.~Cockayne, C.~J. Oates, T.~J. Sullivan, and M.~Girolami}, {\em
		Probabilistic numerical methods for {PDE}-constrained {B}ayesian inverse
		problems}, AIP Conference Proceedings, 1853 (2017), p.~060001.
	
	\bibitem{COS19}
	{\sc J.~Cockayne, C.~J. Oates, T.~J. Sullivan, and M.~Girolami}, {\em Bayesian
		probabilistic numerical methods}, SIAM Rev., 61 (2019), pp.~756--789.
	
	\bibitem{CGS17}
	{\sc P.~R. Conrad, M.~Girolami, S.~S\"{a}rkk\"{a}, A.~M. Stuart, and
		K.~Zygalakis}, {\em Statistical analysis of differential equations:
		introducing probability measures on numerical solutions}, Stat. Comput., 27
	(2017), pp.~1065--1082.
	
	\bibitem{CRS13}
	{\sc S.~L. Cotter, G.~O. Roberts, A.~M. Stuart, and D.~White}, {\em M{CMC}
		methods for functions: modifying old algorithms to make them faster},
	Statist. Sci., 28 (2013), pp.~424--446.
	
	\bibitem{CrF20}
	{\sc M.~Croci and P.~E. Farrell}, {\em Complexity bounds on supermesh
		construction for quasi-uniform meshes}, J. Comput. Phys., 414 (2020),
	pp.~109459, 7.
	
	\bibitem{CGR18}
	{\sc M.~Croci, M.~B. Giles, M.~E. Rognes, and P.~E. Farrell}, {\em Efficient
		white noise sampling and coupling for multilevel {M}onte {C}arlo with
		nonnested meshes}, SIAM/ASA J. Uncertain. Quantif., 6 (2018), pp.~1630--1655.
	
	\bibitem{DaS16}
	{\sc M.~Dashti and A.~M. Stuart}, {\em The {B}ayesian {A}pproach to {I}nverse
		{P}roblems}, in Handbook of Uncertainty Quantification, Springer, 2016,
	pp.~1--118.
	
	\bibitem{GFY21}
	{\sc M.~Girolami, E.~Febrianto, G.~Yin, and F.~Cirak}, {\em The statistical
		finite element method (stat{FEM}) for coherent synthesis of observation data
		and model predictions}, Comput. Methods Appl. Mech. Engrg., 375 (2021),
	pp.~113533, 32.
	
	\bibitem{HSV14}
	{\sc M.~Hairer, A.~M. Stuart, and S.~J. Vollmer}, {\em Spectral gaps for a
		{M}etropolis--{H}astings algorithm in infinite dimensions}, Ann. Appl.
	Probab., 24 (2014), pp.~2455--2490.
	
	\bibitem{HOG15}
	{\sc P.~Hennig, M.~A. Osborne, and M.~Girolami}, {\em Probabilistic numerics
		and uncertainty in computations}, Proc. A., 471 (2015), pp.~20150142, 17.
	
	\bibitem{KaS05}
	{\sc J.~Kaipio and E.~Somersalo}, {\em Statistical and computational inverse
		problems}, vol.~160 of Applied Mathematical Sciences, Springer-Verlag, New
	York, 2005.
	
	\bibitem{KeH16}
	{\sc H.~Kersting and P.~Hennig}, {\em Active uncertainty calibration in
		{B}ayesian {ODE} solvers}, in Proceedings of the 32nd Conference on
	Uncertainty in Artificial Intelligence (UAI 2016), {AUAI} Press, 2016,
	pp.~309--318.
	
	\bibitem{KSH20}
	{\sc H.~Kersting, T.~J. Sullivan, and P.~Hennig}, {\em Convergence rates of
		{G}aussian {ODE} filters}, Stat. Comput., 30 (2020), pp.~1791--1816.
	
	\bibitem{KTB13}
	{\sc D.~P. Kroese, T.~Taimre, and Z.~I. Botev}, {\em Handbook of {M}onte
		{C}arlo methods}, vol.~706, John Wiley \& Sons, 2013.
	
	\bibitem{LSS19b}
	{\sc H.~C. Lie, A.~M. Stuart, and T.~J. Sullivan}, {\em Strong convergence
		rates of probabilistic integrators for ordinary differential equations},
	Stat. Comput., 29 (2019), pp.~1265--1283.
	
	\bibitem{LST18}
	{\sc H.~C. Lie, T.~J. Sullivan, and A.~L. Teckentrup}, {\em Random {F}orward
		{M}odels and {L}og-{L}ikelihoods in {B}ayesian {I}nverse {P}roblems},
	SIAM/ASA J. Uncertain. Quantif., 6 (2018), pp.~1600--1629.
	
	\bibitem{OCA19}
	{\sc C.~J. Oates, J.~Cockayne, R.~G. Aykroyd, and M.~Girolami}, {\em Bayesian
		probabilistic numerical methods in time-dependent state estimation for
		industrial hydrocyclone equipment}, J. Amer. Statist. Assoc., 114 (2019),
	pp.~1518--1531.
	
	\bibitem{OaS19}
	{\sc C.~J. Oates and T.~J. Sullivan}, {\em A modern retrospective on
		probabilistic numerics}, Stat. Comput., 29 (2019), pp.~1335--1351.
	
	\bibitem{Owh15}
	{\sc H.~Owhadi}, {\em Bayesian numerical homogenization}, Multiscale Model.
	Simul., 13 (2015), pp.~812--828.
	
	\bibitem{Owh17}
	{\sc H.~Owhadi}, {\em Multigrid with rough coefficients and multiresolution
		operator decomposition from hierarchical information games}, SIAM Rev., 59
	(2017), pp.~99--149.
	
	\bibitem{OwZ17}
	{\sc H.~Owhadi and L.~Zhang}, {\em Gamblets for opening the
		complexity-bottleneck of implicit schemes for hyperbolic and parabolic
		{ODE}s/{PDE}s with rough coefficients}, J. Comput. Phys., 347 (2017),
	pp.~99--128.
	
	\bibitem{Qua09}
	{\sc A.~Quarteroni}, {\em Numerical Models for Differential Problems}, vol.~2
	of Modeling, Simulation \& Applications, Springer, 2009.
	
	\bibitem{RPK17}
	{\sc M.~Raissi, P.~Perdikaris, and G.~E. Karniadakis}, {\em Inferring solutions
		of differential equations using noisy multi-fidelity data}, J. Comput. Phys.,
	335 (2017), pp.~736--746.
	
	\bibitem{RPK17b}
	{\sc M.~Raissi, P.~Perdikaris, and G.~E. Karniadakis}, {\em Machine learning of
		linear differential equations using {G}aussian processes}, J. Comput. Phys.,
	348 (2017), pp.~683--693.
	
	\bibitem{SDH14}
	{\sc M.~Schober, D.~Duvenaud, and P.~Hennig}, {\em Probabilistic {ODE} solvers
		with {R}unge--{K}utta means}, in Advances in Neural Information Processing
	Systems 27, Curran Associates, Inc., 2014, pp.~739--747.
	
	\bibitem{SSH19}
	{\sc M.~Schober, S.~S\"{a}rkk\"{a}, and P.~Hennig}, {\em A probabilistic model
		for the numerical solution of initial value problems}, Stat. Comput., 29
	(2019), pp.~99--122.
	
	\bibitem{Ski92}
	{\sc J.~Skilling}, {\em Bayesian solution of ordinary differential equations},
	in Maximum entropy and Bayesian methods, Springer, 1992, pp.~23--37.
	
	\bibitem{Stu10}
	{\sc A.~M. Stuart}, {\em Inverse problems: a {B}ayesian perspective}, Acta
	Numer., 19 (2010), pp.~451--559.
	
	\bibitem{Sul17}
	{\sc T.~J. Sullivan}, {\em Well-posed {B}ayesian inverse problems and
		heavy-tailed stable quasi-{B}anach space priors}, Inverse Probl. Imaging, 11
	(2017), pp.~857--874.
	
	\bibitem{TLS18}
	{\sc O.~Teymur, H.~C. Lie, T.~Sullivan, and B.~Calderhead}, {\em Implicit
		probabilistic integrators for {ODE}s}, in Advances in Neural Information
	Processing Systems, 2018, pp.~7244--7253.
	
	\bibitem{TZC16}
	{\sc O.~Teymur, K.~Zygalakis, and B.~Calderhead}, {\em Probabilistic linear
		multistep methods}, in Advances in Neural Information Processing Systems,
	2016, pp.~4321--4328.
	
	\bibitem{TKS19}
	{\sc F.~Tronarp, H.~Kersting, S.~S\"{a}rkk\"{a}, and P.~Hennig}, {\em
		Probabilistic solutions to ordinary differential equations as nonlinear
		{B}ayesian filtering: a new perspective}, Stat. Comput., 29 (2019),
	pp.~1297--1315.
	
	\bibitem{Ver94}
	{\sc R.~Verf\"{u}rth}, {\em A posteriori error estimation and adaptive
		mesh-refinement techniques}, J. Comput. Appl. Math., 50 (1994), pp.~67--83.
	
	\bibitem{Ver13}
	{\sc R.~Verf{\"u}rth}, {\em A posteriori error estimation techniques for finite
		element methods}, Numerical Mathematics and Scientific Computation, Oxford
	University Press, Oxford, 2013.
	
	\bibitem{Vih12}
	{\sc M.~Vihola}, {\em Robust adaptive {M}etropolis algorithm with coerced
		acceptance rate}, Stat. Comput., 22 (2012), pp.~997--1008.
	
	\bibitem{ZiZ92b}
	{\sc O.~C. Zienkiewicz and J.~Z. Zhu}, {\em The superconvergent patch recovery
		and a posteriori error estimates. {I}. {T}he recovery technique}, Internat.
	J. Numer. Methods Engrg., 33 (1992), pp.~1331--1364.
	
	\bibitem{ZiZ92}
	{\sc O.~C. Zienkiewicz and J.~Z. Zhu}, {\em The superconvergent patch recovery
		and a posteriori error estimates. {II}. {E}rror estimates and adaptivity},
	Internat. J. Numer. Methods Engrg., 33 (1992), pp.~1365--1382.
\end{thebibliography}
\end{document}